\documentclass[microtype]{gtpart}
\usepackage{xcolor,graphicx}
\usepackage[OT2,T1]{fontenc}
\usepackage{hyperref}
\usepackage{amsthm}
\usepackage{amsfonts}
\usepackage{amssymb}
\usepackage[mathscr]{euscript}
\usepackage[all]{xy}
\usepackage{amsmath}
\usepackage{epsfig}
\usepackage{latexsym}
\usepackage{stackengine}

\title[Milnor's concordance invariants for knots on surfaces]{Milnor's concordance invariants for knots on surfaces}

\author[M. Chrisman]{Micah Chrisman}
\address{Department of Mathematics, The Ohio State University, Marion Campus, Marion, Ohio, 43302}
\email{chrisman.76@osu.edu}
\urladdr{https://micah46.wixsite.com/micahknots}


\arxivreference{2002.01505}

\newtheorem{theorem}{Theorem}[subsection]
\newtheorem{lemma}[theorem]{Lemma}
\newtheorem{proposition}[theorem]{Proposition}
\newtheorem{corollary}[theorem]{Corollary}

\theoremstyle{definition}
\newtheorem{definition}[theorem]{Definition}
\newtheorem{example}[theorem]{Example}
\theoremstyle{remark}
\newtheorem{remark}[theorem]{Remark}

\newcommand\cyr
{
\renewcommand\rmdefault{wncyr} \renewcommand\sfdefault{wncyss} \renewcommand\encodingdefault{OT2} \normalfont
\selectfont
}
\DeclareTextFontCommand{\textcyr}{\cyr}

\newcommand{\Zh}{\hbox{\hspace{-5mm} \textcyr{Zh}}} %Use -6.4mm for 11 pt font and -5.8mm for 10pt font
\newcommand{\zh}{\hbox{\hspace{-4.6mm} \textcyr{zh}}} %Use -6.0mm for 11 pt font and -5.8mm for 10pt font

\newcommand{\lk}{\operatorname{\ell{\it k}}}
\newcommand{\vlk}{\operatorname{{\it v}\ell{\it k}}}

\begin{document}

\begin{abstract} Milnor's $\bar{\mu}$-invariants of links in the $3$-sphere $S^3$ vanish on any link concordant to a boundary link. In particular, they are trivial on any knot in $S^3$. Here we consider knots in thickened surfaces $\Sigma \times [0,1]$, where $\Sigma$ is closed and oriented. We construct new concordance invariants by adapting the Chen-Milnor theory of links in $S^3$ to an extension of the group of a virtual knot. A key ingredient is the Bar-Natan $\Zh$ map, which allows for a geometric interpretation of the group extension. The group extension itself was originally defined by Silver-Williams. Our extended $\bar{\mu}$-invariants obstruct concordance to homologically trivial knots in thickened surfaces. We use them to give new examples of non-slice virtual knots having trivial Rasmussen invariant, graded genus, affine index (or writhe) polynomial, and generalized Alexander polynomial. Furthermore, we complete the slice status classification of all virtual knots up to five classical crossings and reduce to four (out of 92800) the number of virtual knots up to six classical crossings having unknown slice status. 

Our main application is to Turaev's concordance group $\mathscr{VC}$ of long knots on surfaces. Boden and Nagel proved that the concordance group $\mathscr{C}$ of classical knots in $S^3$ embeds into the center of $\mathscr{VC}$. In contrast to the classical knot concordance group, we show $\mathscr{VC}$ is not abelian; answering a question posed by Turaev. 
\end{abstract}
\maketitle

\section{Overview}
For an $m$-component link $L$ in the $3$-sphere $S^3$,  Milnor \cite{milnor} introduced an infinite family of link isotopy invariants called the $\bar{\mu}$-invariants. Collectively they determine how deep the longitudes of $L$ lie in the lower central series of $\pi_1(S^3\smallsetminus L)$. Milnor's invariants are well-defined only up to a certain indeterminacy and give integer valued invariants when all lower order invariants are trivial. The $\bar{\mu}$-invariants are also invariants of link concordance. The set of $\bar{\mu}$-invariants vanishes for any link that is concordant to a boundary link, i.e. a link whose components bound pairwise disjoint Seifert surfaces. In particular, the $\bar{\mu}$-invariants are trivial for all knots in $S^3$. The aim of this paper is to construct an extension of Milnor's concordance invariants to knots in thickened surfaces $\Sigma \times I$, where $\Sigma$ is closed and oriented and $I=[0,1]$. The extended $\bar{\mu}$-invariants introduced here obstruct concordance to homologically trivial knots in $\Sigma \times I$.

By concordance of knots in $\Sigma \times I$, we mean concordance in the sense of Turaev \cite{turaev_cobordism}. Let $\Sigma,\Sigma^*$ be closed oriented surfaces and $K\subset \Sigma \times I$, $K^*\subset \Sigma^* \times I$ be oriented knots. Then $K$ and $K^*$ are said to be \emph{virtually concordant} if there is a compact oriented $3$-manifold $W$ and a smoothly and properly embedded annulus $\Lambda \subset W \times I$ such that $\partial W=\Sigma^* \sqcup -\Sigma$ and $\partial \Lambda=K^* \sqcup -K$ (see Figure \ref{fig_conc_defn}). The minus signs here indicate a change in the given orientation. A knot is called \emph{virtually slice} if it is concordant to the unknot in the thickened $2$-sphere. 

\begin{figure}[htb]
\tiny
\centerline{
\def\svgwidth{2.7in}
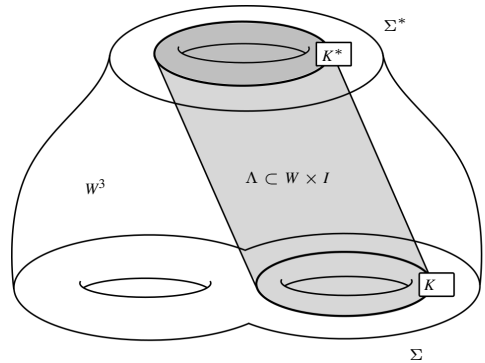} 
\normalsize
\caption{Schematic depiction of a virtual concordance.}
\label{fig_conc_defn}
\end{figure}

In the smooth category, virtual concordance is equivalent to concordance of virtual knots, as defined by Kauffman \cite{lou_cob}. Two main applications of the extended Milnor invariants are given. The first application is to the study of Turaev's group of long knots on surfaces, or equivalently, the concordance group $\mathscr{VC}$ of long virtual knots (see Definition \ref{defn_vc} ahead). Secondly, we continue the program begun by Boden-Chrisman-Gaurdreau \cite{bcg1,bcg2} and Rushworth \cite{rush} of determining the slice status of the 92800 virtual knots from Green's table \cite{green} having at most six classical crossings. Using the extended $\bar{\mu}$-invariants, we reduce the number of virtual knots having unknown slice status to four. The key ideas and their motivations are sketched below. Definitions of terms used in the remainder of this section are presented mostly in Section \ref{sec_virt_weld_ribbon} and otherwise as needed throughout the text. Henceforth, decimal numbers will correspond to virtual knots in Green's table, so that $3.5$ refers to the fifth 3-crossing entry therein.

\subsection{The invariants} Let $K$ be a virtual knot diagram and $G(K)$ its group. The main tool used is an extension $\widetilde{G}(K)$ of $G(K)$. The extension $\widetilde{G}(K)$ is originally due to Silver-Williams \cite{silwill0,silwill1,silwill}. As in the case of classical knots, the nilpotent quotients of $G(K)$ are always free and thus do not provide any useful slice obstructions. Here we will show that the nilpotent quotients of $\widetilde{G}(K)$ are generally not free. Hence, by adapting the Chen-Milnor theory of links in $S^3$ to the nilpotent quotients of $\widetilde{G}(K)$, we can construct new concordance invariants for virtual knots. They collectively determine how deep certain elements of $\widetilde{G}(K)$, called \emph{extended longitudes}, live in the lower central series. For a long virtual knot $\vec{K}$, we derive an extended Artin representation from $\widetilde{G}(\vec{K})$. Following Habegger-Lin \cite{HL2}, we obtain new concordance invariants of long virtual knots from certain automorphisms of these nilpotent quotients.

\subsection{Application: $\mathscr{VC}$ is not abelian} It is well-known that the smooth concordances classes of knots in $S^3$ form an abelian group $\mathscr{C}$ under the connected sum operation. In \cite{boden_nagel}, Boden and Nagel proved that $\mathscr{C}$ embeds as a subgroup into the center of the long virtual knot concordance group $\mathscr{VC}$. 

While the structure of the classical knot concordance group has been extensively studied in both the smooth and topological categories, little is known about the structure of $\mathscr{VC}$. Cochran, Orr, and Teichner \cite{COT} showed there is a filtration $\mathscr{C}^{\text{top}} \supset\mathscr{F}_{0} \supset \mathscr{F}_{.5}\supset F_{1} \supset \cdots$ of the topological concordance group of classical knots in $S^3$. The quotient $\mathscr{C}^{\text{top}}/\mathscr{F}_0 \cong \mathbb{Z}_2$ corresponds to the two band-pass classes of knots: those that are band-pass equivalent to the trefoil and those that are band-pass equivalent to the unknot. On the other hand, for every concordance class $[K]\in\mathscr{VC}$, there is a $J\in [K]$ that is not band-pass equivalent to $K$ (Chrisman \cite{band_pass}). Hence, concordance does not guarantee band-pass equivalence for long virtual knots. Furthermore, there are more than two band-pass equivalence classes in the virtual case: for every $[K]$, there is an $L\in [K]$ that is not band-pass equivalent to either the unknot or the trefoil. This divergence from the classical case suggests that the structure of $\mathscr{VC}$ is in need of deeper study.

Here we show that in contrast to the classical knot concordance group, $\mathscr{VC}$ is not abelian. This answers a question posed by Turaev (see \cite{turaev_cobordism}, Section 6.5). Previously, Manturov \cite{manturov_long} proved that the equivalence classes of long virtual knots form a non-commutative monoid.  Our result shows that non-commutativity also holds under the coarser relation of virtual concordance.

\subsection{Application: virtual slice obstructions} A virtual knot is said to be \emph{almost classical (AC)} if it can be represented by a homologically trivial knot in some thickened surface $\Sigma \times I$. In \cite{bbc}, Boden and the author showed that the generalized Alexander polynomial vanishes on any virtual knot that is concordant to an AC knot. Here we prove that the extended $\bar{\mu}$-invariants also vanish on any virtual knot concordant to an AC knot. Furthermore, it will be shown that the vanishing of the extended $\bar{\mu}$-invariants implies the vanishing of the generalized Alexander polynomial. As slice obstructions, the extended $\bar{\mu}$-invariants are therefore at least strong as the generalized Alexander polynomial and the concordance invariants derived from it i.e., the odd writhe, Henrich-Turaev polynomial, and affine index (or writhe) polynomial.

On the other hand, the extended $\bar{\mu}$-invariants can be nontrivial when the generalized Alexander polynomial vanishes. In addition to the generalized Alexander polynomial, other known slice obstructions are: Turaev's graded genus \cite{bcg1}, the Rasmussen invariant (Dye-Kaestner Kauffman \cite{DKK}, Karimi \cite{karimi}, Rushworth \cite{rush}), and the signature of almost classical knots \cite{bcg2}. These tools have been collectively used to determine the slice status of most of the knots in Green's table. In addition, there are 1299 virtual knots having at most six classical crossings that are known to be slice \cite{bbc,bcg1}. However, there are 38 virtual knots in Green's table where all these invariants are trivial and a slice movie has yet to be found. Using the extended $\bar{\mu}$-invariants and Artin representation, we show that 12 of these are not slice and provide slice movies for an additional 22. In particular, we complete the slice status classification of the 2564 nontrivial virtual knots having at most five crossings. The four remaining virtual knots of unknown slice status all have six crossings and are shown in Figure \ref{fig_4_unknown} at the end of the paper.

The extended $\bar{\mu}$-invariants beyond the first non-vanishing order can also be used to separate the concordance classes of some virtual knots. An example will be given of two non-concordant virtual knots having the same slice genus, graded genus, generalized Alexander polynomial, and first non-vanishing extended $\bar{\mu}$-invariants, but different higher order extended $\bar{\mu}$-invariants.

\subsection{$\Zh$ and Tube} The main difficulty in constructing the extended $\bar{\mu}$-invariants and the extended Artin representation lies in proving invariance under concordance. For this we use a two-step geometric realization process. The first step is the Bar-Natan $\Zh$ map. The $\Zh$ map assigns to each $m$-component virtual link an $(m+1)$-component welded link $\Zh(L)=L \cup \omega$. In \cite{bbc}, this map was shown to be functorial under concordance, so that concordant virtual links are mapped to concordant welded links having one additional component. The second step is Satoh's Tube map \cite{satoh}. It sends a welded link $L$ to a ribbon torus link $\text{Tube}(L)$ in $S^4$. Satoh proved that if two welded links $L,L^*$ are equivalent, then $\text{Tube}(L),\text{Tube}(L^*)$ are ambient isotopic ribbon torus links. In \cite{bbc}, it was shown that the Tube map is also functorial under concordance. Hence, $\text{Tube} \circ \Zh$ maps concordant virtual links to concordant ribbon torus links in $S^4$.

The $\Zh$ map gives a geometric realization of the extended group of a virtual link while the $\text{Tube}$ map gives a geometric realization of the group of a virtual link. Together they imply that $\widetilde{G}(L)\cong G(\Zh(L)) \cong \pi_1(S^4 \smallsetminus \text{Tube} \circ \Zh(L))$. These isomorphisms provide a natural way to port the standard geometric techniques of classical link concordance back to the purely algebraic definition of the extension $\widetilde{G}(L)$. The extended $\bar{\mu}$-invariants are constructed by first developing the Chen-Milnor theory for the groups $G(L)$ of multi-component virtual links and then pre-composing with the $\Zh$ map. In particular, the invariants for virtual knots, which we will call the $\overline{\zh}$-invariants, are the Milnor invariants of the $2$-component welded link $\Zh(L)$. A similar approach is applied to obtain the extended Artin representation, where we pre-compose the usual Artin representation of multi-component virtual string links with the $\Zh$ map.

\subsection{Organization} The organization of this paper is as follows. In Section \ref{sec_virt_weld_ribbon}, we review virtual link concordance, the $\Zh$ map, the $\text{Tube}$ map, and the extensions $\widetilde{G}(L)$. Section \ref{sec_gen_chen_milnor} proves a Chen-Milnor type theorem that will be applied in later sections to the groups and extended groups of both virtual links and virtual string links.  Section \ref{sec_milnor_extended milnor} develops the $\bar{\mu}$-invariants and extended $\bar{\mu}$-invariants for virtual knots and links. The Artin and extended Artin representations of virtual string links are discussed in Section \ref{sec_artin}. Properties and relations of the $\overline{\zh}$-invariants of virtual knots are proved in Section \ref{sec_prop_rel}. The two main applications, example calculations, and slice movies are given in Section \ref{sec_calc_app}. The paper concludes in Section \ref{sec_questions} with questions for further research.

\section{Virtual, Welded, $\&$ Ribbon Torus links} \label{sec_virt_weld_ribbon}
\subsection{Virtual concordance} \label{sec_vknot_defn} Here we discuss how virtual concordance of knots in thickened surfaces can be reformulated as concordance of virtual knots. Recall that a \emph{virtual knot diagram} is an oriented immersed curve in the plane having only transverse double points, where each crossing is marked as a usual over/under crossing or as a virtual crossing (Kauffman \cite{KaV}). A virtual crossing is denoted with a small circle around the double point. Two virtual knot diagrams $K,K^*$ are \emph{equivalent}, denoted $K \leftrightharpoons K^*$, if they may be obtained from one another by Reidemeister moves and the \emph{detour move} (Figure \ref{fig_detour}). The detour move allows for the arbitrary placement of any arc containing only virtual crossings. Here we will denote the three Reidemeister moves as $\Omega 1, \Omega 2, \Omega 3$. An equivalence class of diagrams is called a \emph{virtual knot}. Virtual link diagrams are defined analogously, where a finite number of oriented immersed circles is allowed. The components of a virtual link diagram are assumed to be ordered.

\begin{figure}[htb]
\centerline{
\begin{tabular}{c}
\def\svgwidth{2.5in}
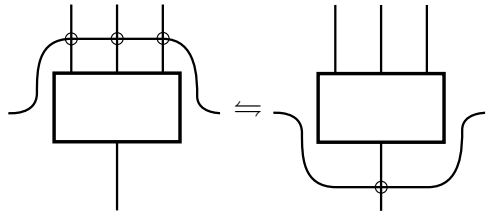
\end{tabular}
} 
\caption{The detour move.}\label{fig_detour}
\end{figure}

A link diagram $\mathscr{L}$ on a closed oriented surface $\Sigma$ corresponds to a virtual link diagram $L$. The diagram $L$ is obtained by drawing the crossing data of $\mathscr{L}$ in the plane and connecting the ends of the arcs exactly as they occur on $\Sigma$ (see Figure \ref{fig_convert}, left). If this cannot be done without adding additional intersections of arcs, the new crossings are marked as virtual crossings. Conversely, a link diagram on a surface can be obtained from any virtual link diagram $L$.  Thicken the arcs of the diagram of $L$, as in the right of Figure \ref{fig_convert}. At a virtual crossing, two bands pass by one another without intersecting. Adding discs to the boundary of this surface gives a link diagram $\mathscr{L}$ on a closed oriented surface $\Sigma$. The surface $\Sigma$ is called the \emph{Carter surface} (or \emph{abstract surface}) of the virtual link diagram $L$ (Carter-Kamada-Saito \cite{CKS}).

\begin{figure}[htb]
\begin{tabular}{c}
\def\svgwidth{5in}
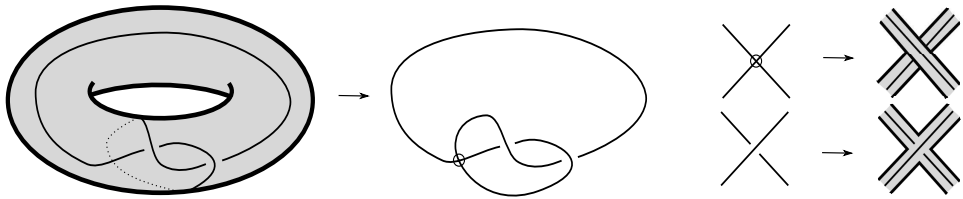
\end{tabular} 
\caption{Converting a knot on a surface to a virtual knot and vice versa.}\label{fig_convert}
\end{figure}

A smooth concordance between two links in $S^3$ can be visualized as a movie between diagrams of the links. The movie consists of a finite connected sequence of Reidemeister moves, births, deaths, and saddles (see Figure \ref{fig_concordance_moves}). This motivates Kauffman's definition of virtual link concordance \cite{lou_cob}. 

\begin{definition}[Virtual link concordance, slice virtual link] Two virtual knots $K,K^*$ are said to be \emph{concordant} if and only if they may be obtained from one another by a finite connected sequence of Reidemeister moves, detour moves, births, deaths, and saddle moves (see Figure \ref{fig_concordance_moves}) such that:
\begin{eqnarray} \label{eqn_euler}
\#(\text{births})-\#(\text{saddles})+\#(\text{deaths})&=& 0.
\end{eqnarray}
Two $m$-component virtual links are \emph{concordant} if there is a sequence of moves as above that restricts to a concordance of virtual knots on each component. An $m$-component virtual link is said to be \emph{slice} if it is concordant to the $m$-component unlink.
\end{definition}

\begin{figure}[htb]
\small
\centerline{
\begin{tabular}{c} \\
\def\svgwidth{4.3in}
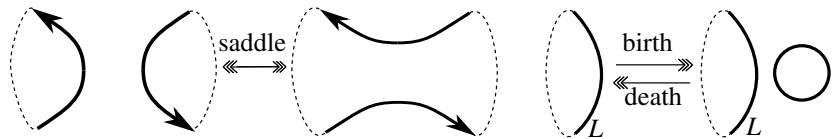
\end{tabular}
} 
\caption{Births, deaths, and saddles.}\label{fig_concordance_moves}
\end{figure}

\begin{example} Figure \ref{fig_conc_example} shows the diagram of a virtual knot $K$ equivalent to 5.1216. Perform a saddle move to the two arcs connected by the blue dashed chord. The result is a $2$-component virtual link diagram that is equivalent to the $2$-component unlink. Using a death move on the red component gives a concordance to the unknot. Hence, $5.1216$ is slice.
\end{example}

\begin{figure}[htb]
\begin{tabular}{c}
\centerline{
\def\svgwidth{5in}
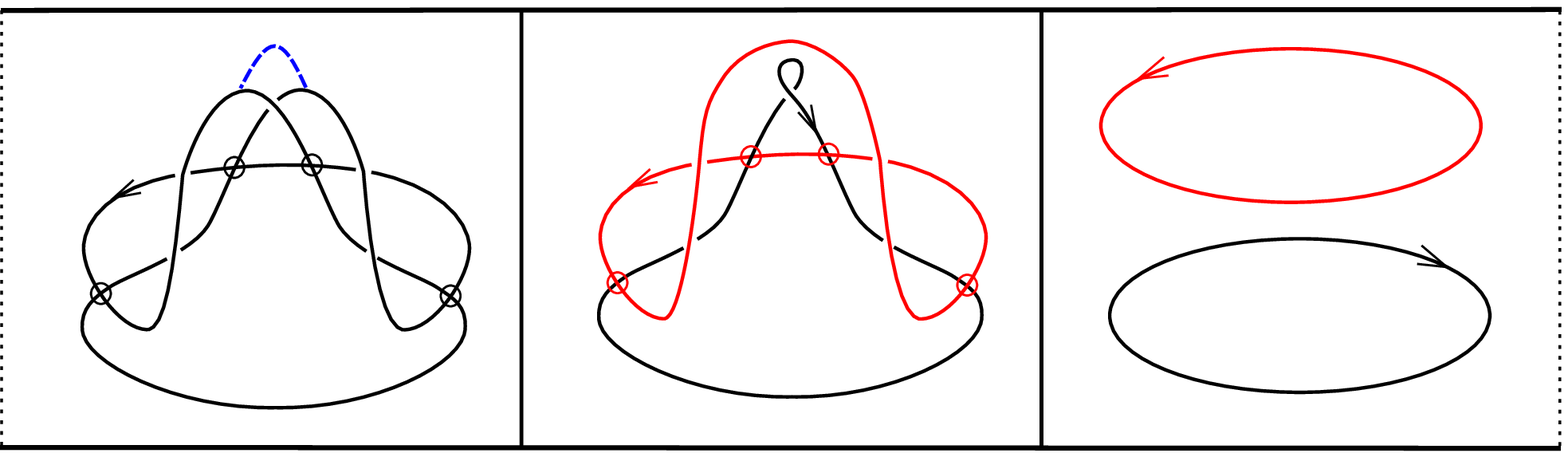
}
\end{tabular} 
\caption{A slice movie for a virtual knot $K \leftrightharpoons 5.1216$.}\label{fig_conc_example}
\end{figure}

A result of Carter, Kamada, and Saito (see \cite{CKS}, Lemma 4.5) implies that $\mathscr{K} \subset \Sigma \times I,\mathscr{K}^* \subset \Sigma^* \times I$ are virtually concordant if and only if they represent concordant virtual knots. A brief sketch of their argument is as follows. Let $W$ be a compact oriented $3$-manifold such that $\partial W=\Sigma^* \sqcup -\Sigma$ and suppose $\Lambda \subset W \times I$ is a concordance between $\mathscr{K},\mathscr{K}^*$. Assume that the projection $W \times I \to W$ maps $\Lambda$ to a generically immersed surface in $W$, also denoted by $\Lambda$. Consider a Morse function $h:W \to [0,1]$, where $h^{-1}(0)=\Sigma$ and $h^{-1}(1)=\Sigma^*$. The singularities of $\Lambda$ correspond to Reidemeister moves and the critical points of $h$ in $\Lambda$ correspond to births, deaths, and saddles. The critical points of $h$ in $W\smallsetminus \Lambda$ can be viewed as additions/deletions of $0$, $1$, and $2$-handles. These correspond to stabilizations/destabilizations of knots in thickened surfaces. Since virtual knots are in one-to-one correspondence with knots in thickened surfaces modulo ambient isotopy, stabilization/destabilization, and orientation preserving diffeomorphisms of surfaces (Kuperberg \cite{kuperberg}), the result follows.

\subsection{The $\Zh$ Map} Two $m$-component virtual link diagrams are said to be \emph{welded equivalent} if they may be obtained from one another by a finite sequence of Reidemeister moves, detour moves, and forbidden overpass moves (see Figure \ref{fig_forbidden}). The forbidden overpass move allows for the arbitrary rearrangement of successive overcrossings and is hence also referred to as the \emph{overcrossings-commute relation}. If $L,L^*$ are welded equivalent, we will write $L\leftrightharpoons_w L^*$. Two virtual link diagrams are said to be \emph{welded concordant} if they admit a concordance of virtual link diagrams when forbidden overpass moves are also allowed.

\begin{figure}[htb]
\centerline{
\begin{tabular}{c} \\ \\ 
\def\svgwidth{1.3in}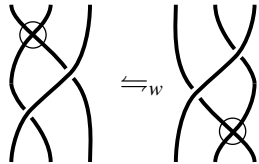
\end{tabular}
}
\caption{The overcrossings-commute relation (or forbidden overpass move).} \label{fig_forbidden}
\end{figure}

Bar-Natan's $\Zh$ map associates to each $m$-component virtual link $L$ an $(m+1)$-component semi-welded link $L\cup\omega$. Before defining the $\Zh$ map, we recall the definitions of semi-welded equivalence and semi-welded concordance from \cite{bbc}.

\begin{definition}[Semi-welded equivalence $\&$ concordance] \label{defn_semi}  Two $(m+1)$-component virtual link diagrams $L \cup \omega$, $L^* \cup \omega^*$ are said to be \emph{semi-welded equivalent} if they may be obtained from one another by a finite sequence of Reidemeister moves and detour moves involving all components of $L\cup \omega$, together with forbidden overpass moves where the overcrossing arc in Figure \ref{fig_forbidden} belongs to the $\omega$ component. Semi-welded equivalence is denoted with $\leftrightharpoons_{sw}$. A \emph{semi-welded concordance} is welded concordance between $L \cup \omega$, $L^* \cup \omega^*$ where the forbidden overpass move is only allowed on the $\omega$ component.
\end{definition}

The diagram $\Zh(L)=L \cup \omega$ is constructed as follows. At each classical crossing, draw a new oriented arc through the double point that crosses over both of the other arcs (see Figure \ref{fig_zh_defn}). The new arcs are assembled into a single component $\omega$ by connecting their ends together arbitrarily. The only requirement is that orientation of the arcs is preserved. Any new intersections of arcs that are introduced during this process are marked as virtual crossings. The new $(m+1)$-component link is denoted by $\Zh(L)$ and it is well-defined up to semi-welded equivalence. An example of the $\Zh$ map is given in Figure \ref{fig_zh_ex}, where two possible outcomes of the construction are depicted.

\begin{figure}[htb]
\centerline{
\begin{tabular}{c} \\ \\ 
\def\svgwidth{4in}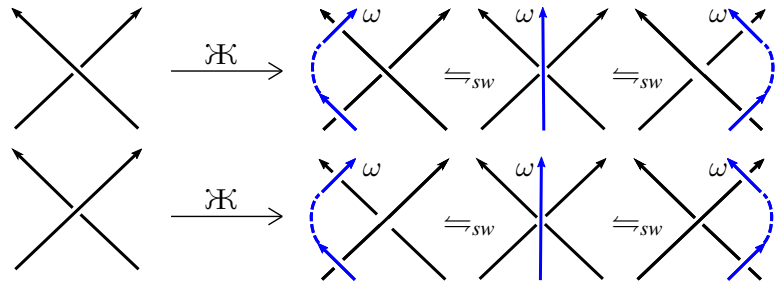
\end{tabular}
}
\caption{Behaviour of the $\Zh$ map at a classical crossing.} \label{fig_zh_defn}
\end{figure}

\begin{figure}[htb]
\small
\centerline{
\begin{tabular}{c} \\ \\ 
\def\svgwidth{4in}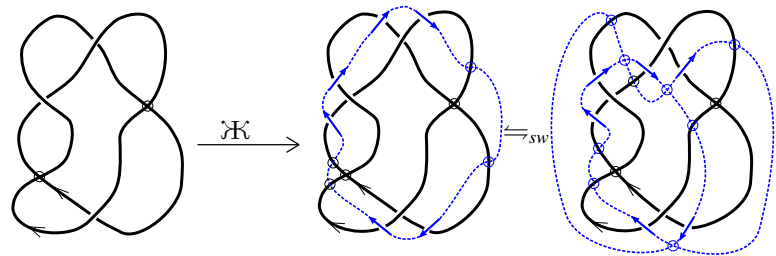
\end{tabular}
}
\normalsize
\caption{The $\Zh$ map applied to the virtual knot $K=3.5$. The two depicted ways of drawing the $\omega$ component of $\Zh(K)$ are semi-welded equivalent.} \label{fig_zh_ex}
\end{figure}

For the reader's convenience, we provide a sketch of the argument given in \cite{bbc} that the $\Zh$ map is well-defined. Suppose we are given an initial connection of the arcs of the $\omega$ component. This may be viewed as a choice of ordering of the arcs of the $\omega$ component. It must be shown that any other ordering may be obtained from this by semi-welded moves. Consider the effect of traveling along the $\omega$ component from a base point. Using detour moves, it may be arranged that all of the virtual crossings appear before the classical crossings. Next observe that the order of any two consecutive classical crossings on the $\omega$ component can be transposed using a detour move and a forbidden overpass move on the $\omega$ component. See Figure \ref{fig_zh_wd}. Now, an alternative ordering of the arcs of the $\omega$ component can be viewed as a permutation of the initial ordering. Since permutations are generated by transpositions, it follows that we may obtain the alternative ordering using a sequence of forbidden overpass and detour moves. Thus, the arcs of the $\omega$ component may be stitched together in any order and we see that $\Zh(L)$ is well-defined.

\begin{figure}[htb]
\small
\centerline{
\begin{tabular}{c} \\ \\ 
\def\svgwidth{5in}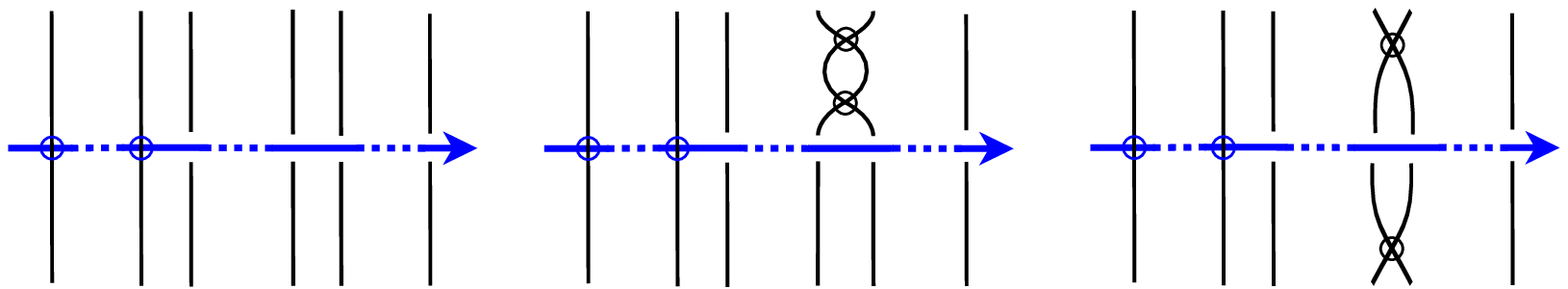
\end{tabular}
}
\normalsize
\caption{The $\Zh$ map is well-defined.} \label{fig_zh_wd}
\end{figure}

If $L,L^*$ are equivalent virtual links, then $\Zh(L), \Zh(L^*)$ are semi-welded equivalent. This result was first proved Bar-Natan \cite{dror}. A detailed proof is given in \cite{bbc}, Proposition 4.1. Furthermore, a concordance movie between two virtual links $L,L^*$ gives a semi-welded concordance movie of $\Zh(L)$, $\Zh(L^*)$. This implies the following result. 

\begin{theorem}[\cite{bbc}, Proposition 4.3] \label{thm_semi_weld_conc} If $L,L^*$ are concordant virtual links, then $\Zh(L),\Zh(L^*)$ are  semi-welded concordant.
\end{theorem}

\subsection{The Tube Map} Satoh's Tube map \cite{satoh} assigns to each welded link a ribbon torus link in $S^4$. The Tube map is defined using \emph{broken surface diagrams}. After an ambient isotopy, a smoothly embedded surface in $\mathbb{R}^4$ can be projected to a generically immersed surface in $\mathbb{R}^3$. For concreteness, say the projection is $(x_1,x_2,x_3,x_4) \to (x_1,x_2,x_3)$. A broken surface diagram is constructed by breaking double lines of the immersion along the sheet having smaller $x_4$ coordinate (Figure \ref{fig_broken}). 

\begin{figure}[htb]
\centerline{
\begin{tabular}{cc} \begin{tabular}{c}
\def\svgwidth{1.5in}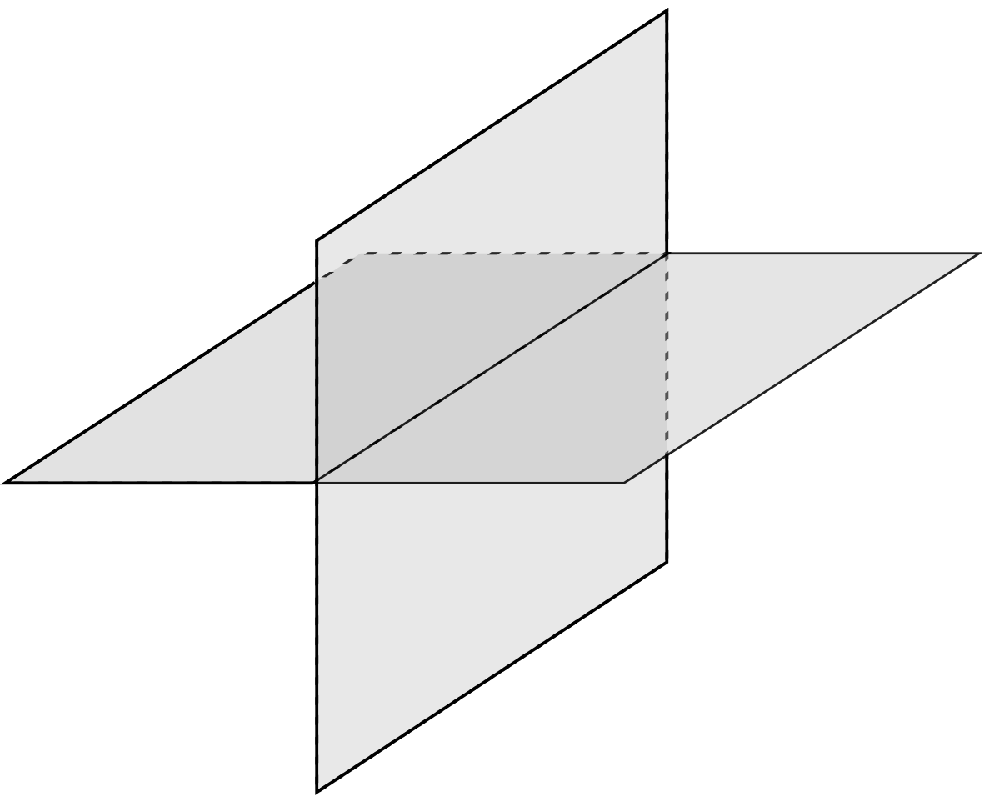 \end{tabular} & \begin{tabular}{c}
\def\svgwidth{1.85in}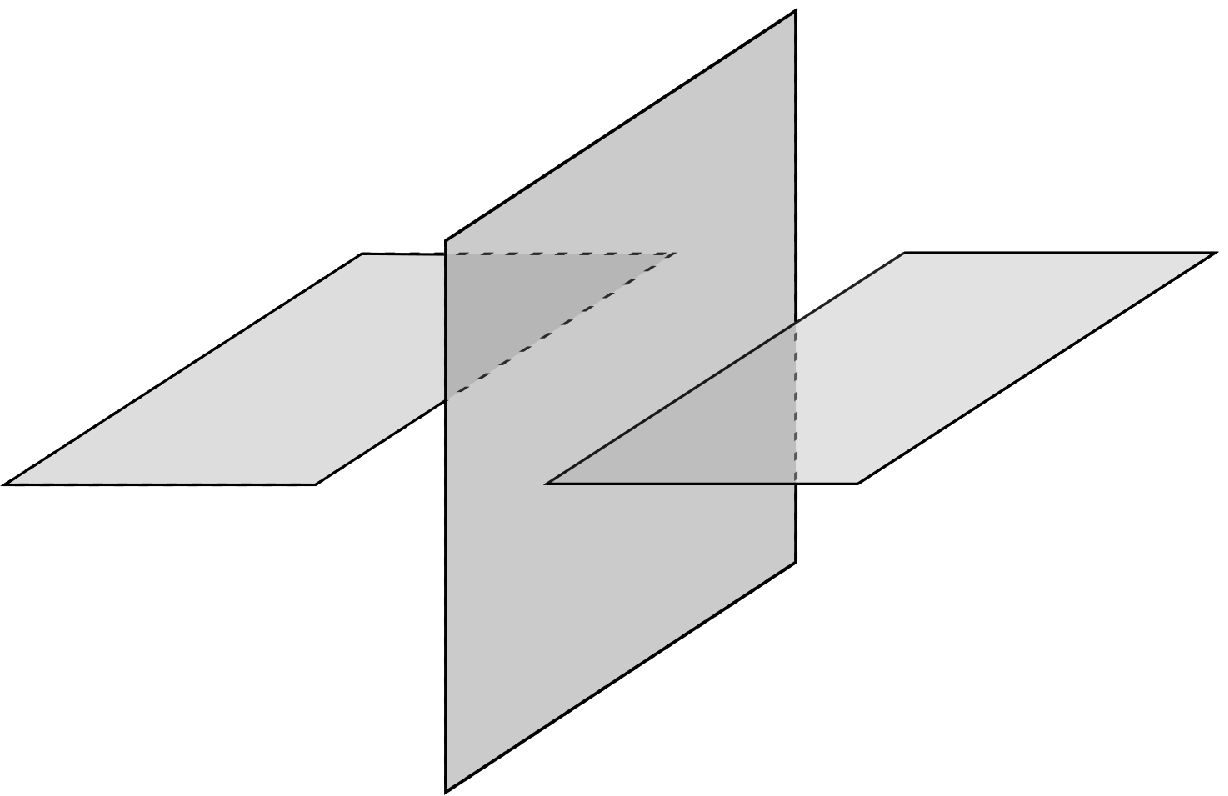 \end{tabular}  \end{tabular}
}
\caption{Breaking a double line (left) to form a broken surface diagram (right).} \label{fig_broken}
\end{figure}

The Tube map associates a broken surface diagram $\text{Tube}(L)$ to each virtual link diagram $L$ as follows. Each arc of the diagram is mapped to an annulus. At virtual crossings, two cylinders pass over and under one another as in Figure \ref{fig_tube}, right. At a classical crossing, the broken surface diagram appears as in Figure \ref{fig_tube}, left. Then $\text{Tube}(L)$ is a broken surface diagram of a collection of knotted tori in $S^4$, with one torus component for each component of $L$. Moreover, $\text{Tube}(L)$ is a \emph{ribbon torus link}: it bounds an immersed $3$-dimensional handlebody in $S^4$ having only ribbon singularities.

\begin{figure}[htb]
\centerline{
\begin{tabular}{cc} 
$$\xymatrix{\begin{tabular}{c}\psfig{file=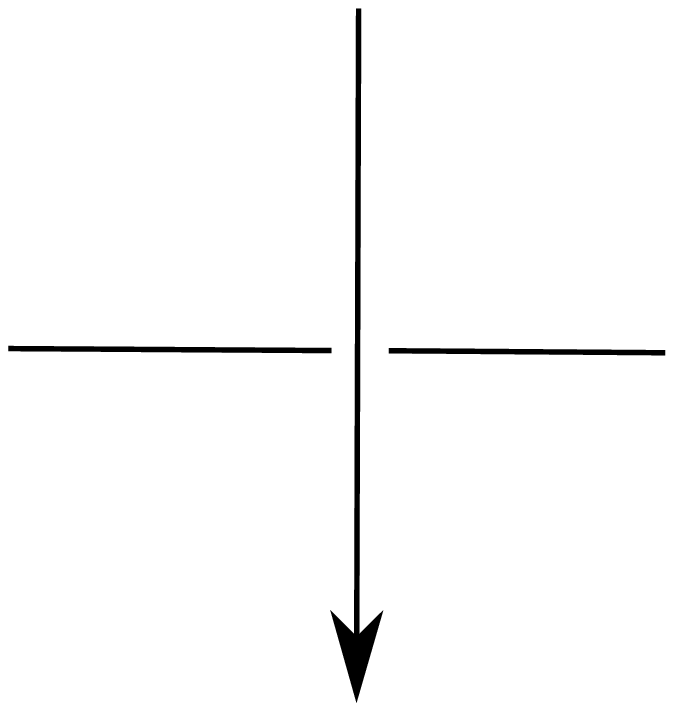,width=1.1in} \end{tabular} \ar[r]^-{\text{Tube}} & \begin{tabular}{c}\psfig{file=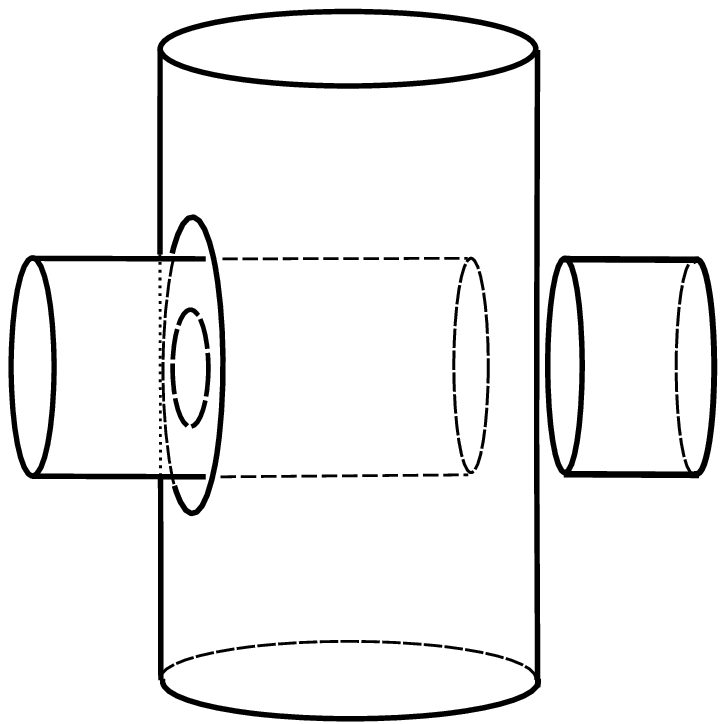,width=1.1in} \end{tabular}}$$ &
$$\xymatrix{
\begin{tabular}{c}\psfig{file=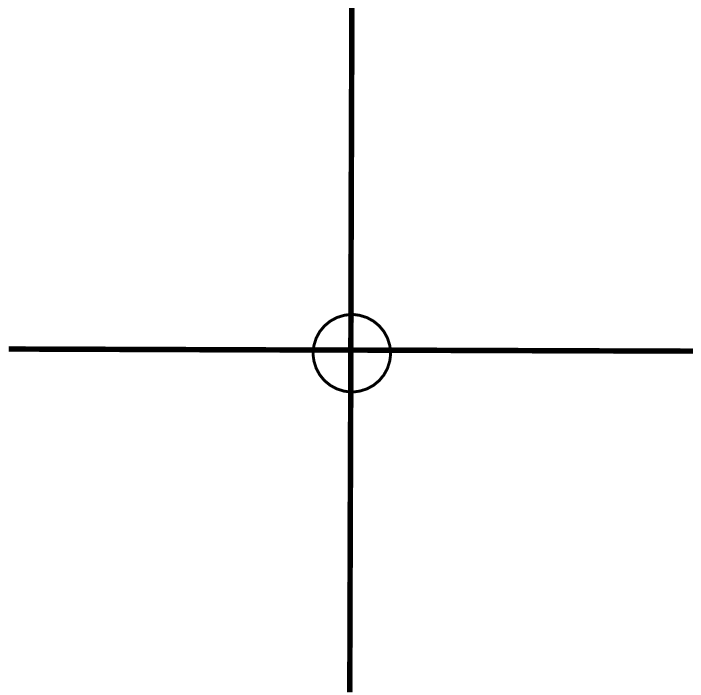,width=1.1in} \end{tabular} \ar[r]^-{\text{Tube}} & \begin{tabular}{c}\psfig{file=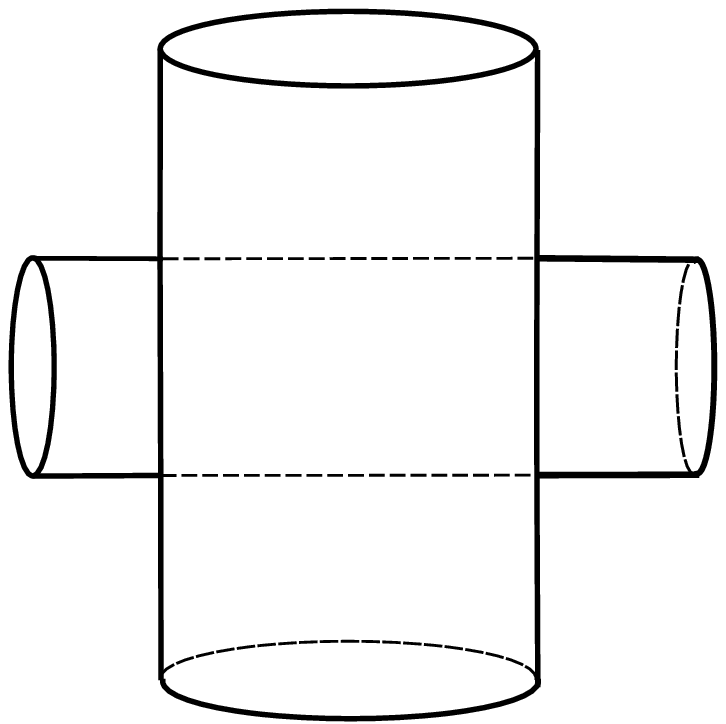,width=1.1in} \end{tabular}}$$
\end{tabular}
}
\caption{The Tube map.} \label{fig_tube}
\end{figure}

Ribbon torus links are considered equivalent up to ambient isotopy in $S^4$. If $L^*$ is obtained from $L$ by a single Reidemeister move, detour move, or forbidden overpass move, then the Roseman \cite{roseman} moves can be used to show that $\text{Tube}(L)$ and $\text{Tube}(L^*)$ are equivalent ribbon torus links.

\begin{theorem}[Satoh \cite{satoh}, Proposition 3.3] If $L,L^*$ are equivalent welded links, then $\text{Tube}(L)$ and $\text{Tube}(L^*)$ are equivalent ribbon torus links.
\end{theorem}

Two $m$-component ribbon torus links $T,T^*$ are said to be \emph{concordant} if there is a smooth and proper embedding $\Lambda$ of $\bigsqcup_{i=1}^m (S^1 \times S^1) \times I$ into $S^4 \times I$ such that $\Lambda \cap (S^4 \times 1)=T^*$ and $\Lambda \cap (S^4 \times 0)=-T$. In \cite{bbc}, it was shown that a concordance of welded links generates a handle decomposition of a concordance between their corresponding ribbon torus links. This implies that the Tube map is also functorial under concordance. 

\begin{theorem}[\cite{bbc}, Proposition 4.7] \label{thm_tube_conc} If $L,L^*$ are welded-concordant virtual link diagrams, then $\text{Tube}(L)$ and $\text{Tube}(L^*)$ are concordant ribbon torus links.
\end{theorem}

\subsection{Groups $\&$ extended groups} The group of an $m$-component virtual link $L$ was first defined by Kauffman \cite{KaV}. Recall that the \emph{arcs} of a virtual link diagram are the paths along the diagram going from one undercrossing of $L$ to the next, where virtual crossings are ignored.

\begin{definition}[Group of $L$] The group $G(L)$ of a virtual link diagram $L$ is the group whose generators are in one-to-one correspondence with the arcs of $L$ and whose relations are in one-to-one correspondence with the crossings of $L$, as shown in Figure \ref{fig_group_rels}.
\end{definition}

\begin{figure}[htb]
\centerline{
\begin{tabular}{ccccc} \begin{tabular}{c}
\def\svgwidth{.75in}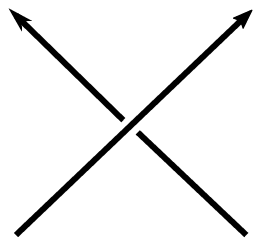 \\ \\ $d=a^{-1}ba$\end{tabular} & & \begin{tabular}{c}
\def\svgwidth{.75in}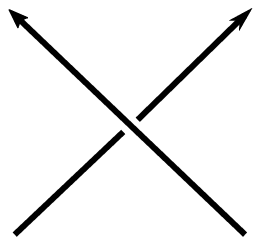 \\  \\ $c=bab^{-1}$ \end{tabular} & \begin{tabular}{c}
\def\svgwidth{.75in}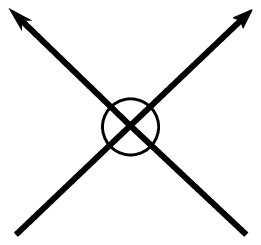 \\  \\ none \end{tabular} \end{tabular}
}
\caption{Relations for the group $G(L)$ of a virtual link diagram $L$.} \label{fig_group_rels}
\end{figure}

The group of a link is invariant under Reidemeister moves, detour moves, and the forbidden overpass move. Thus it is an invariant of welded equivalence. The geometric realization of $G(L)$ is given by the following theorem of Satoh (see also Silver-Williams \cite{silwill_vkg}, Theorem 2.2).

\begin{theorem}[Satoh \cite{satoh}, Proposition 3.4] \label{thm_satoh} For a virtual link diagram $L$, $G(L) \cong \pi_1(S^4 \smallsetminus \text{Tube}(L))$.
\end{theorem}

Extensions of the group of a virtual link have been studied by Manturov \cite{manturov_long}, Silver and Williams \cite{silwill}, and Bardakov and Bellingeri \cite{bardakov_bellingeri}. These extensions were unified by Boden et al. \cite{alex_boden} using a further extension called the \emph{virtual knot group}. In \cite{acpaper}, Boden et al. introduced a subgroup of the virtual knot group called the \emph{reduced virtual knot group} and proved that this subgroup is isomorphic to the extended group from Silver-Williams \cite{silwill} and the quandle group from Manturov \cite{manturov_long} and Bardakov-Bellingeri \cite{bardakov_bellingeri} (Boden et al. \cite{acpaper}, Theorem 3.3). As these three extensions are all equivalent, we will refer to them collectively here as the \emph{extended group} of $L$. It will be denoted by $\widetilde{G}(L)$. The presentation given by Boden et al. in \cite{acpaper} is defined as follows. The \emph{short arcs} of a virtual link diagram are paths along the diagram from one classical crossing to the next, where virtual crossings are ignored. 

\begin{figure}[htb]
\centerline{
\begin{tabular}{ccccc} \begin{tabular}{c}
\def\svgwidth{.75in}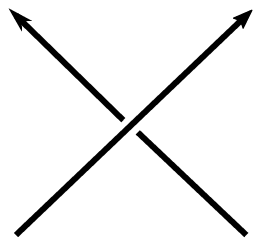 \\ \\ \begin{tabular}{l} $c = vav^{-1}$ \\ $d = a^{-1}v^{-1}bva$\end{tabular} \end{tabular} & & \begin{tabular}{c}
\def\svgwidth{.75in}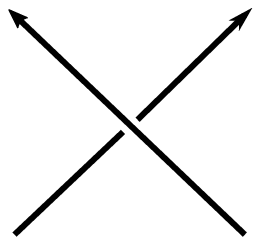 \\  \\ \begin{tabular}{l} $c=bvav^{-1}b^{-1}$ \\ $d=v^{-1}bv$ \end{tabular} \end{tabular} & \begin{tabular}{c}
\def\svgwidth{.75in}\input{group_rels_virt.eps_tex} \\  \\ none \end{tabular} \end{tabular}
}
\caption{Relations for the extension $\widetilde{G}(L)$ of $G(L)$.} \label{fig_ext_group_rels}
\end{figure}

\begin{definition}[The group $\widetilde{G}(L)$ \cite{alex_boden}]  The group $\widetilde{G}(L)$ is the group whose generators are in one-to-one correspondence with the short arcs of $L$, together with an auxiliary variable $v$. $\widetilde{G}(L)$ has two relations for each classical crossing of $L$, as shown in Figure \ref{fig_ext_group_rels}.
\end{definition}

\begin{remark} In \cite{bbc,acpaper}, the extension $\widetilde{G}(L)$ is denoted $\overline{G}_L$. We use the term \emph{extended group} rather than \emph{reduced virtual knot group} to avoid confusion with the term \emph{reduced link group} which is prevalent in the Milnor invariant literature (see e.g. \cite{HL}, \cite{abmw}).
\end{remark}

The group $\widetilde{G}(L)$ is an invariant of virtual link equivalence. To see that $\widetilde{G}(L)$ is an extension of $G(L)$, observe that the relations for $G(L)$ are recovered from $\widetilde{G}(L)$ by setting $v=1$ in Figure \ref{fig_ext_group_rels}. This implies that the following sequence is exact and that $\widetilde{G}(L)$ is indeed an extension of $G(L)$ (see also Silver-Williams \cite{silwill}).
\[
\xymatrix{ 1 \ar[r] & \langle\langle v \rangle\rangle \ar[r] & \widetilde{G}(L) \ar[r]^-{v \to 1} & G(L) \ar[r] & 1}.
\]
The above presentation for $\widetilde{G}(L)$ will be used here for the extended group as it is more natural with respect to the $\Zh$ map. Indeed, the extended group of $L$ is isomorphic to the group of $\Zh(L)$. 

\begin{theorem}[\cite{bbc}, Proposition 4.5] \label{thm_zh} For a virtual link $L$, $\widetilde{G}(L)\cong G(\Zh(L))$. The isomorphism sends the auxiliary variable $v$ to a meridian (i.e. arc) of the $\omega$ component of $\Zh(L)$.
\end{theorem}

Combining this with Theorem \ref{thm_satoh}, we have the following geometric realization of $\widetilde{G}(L)$.

\begin{corollary} $\widetilde{G}(L) \cong G(\Zh(L))\cong \pi_1(S^4 \smallsetminus \text{Tube} \circ\Zh(L))$.
\end{corollary}

\section{Chen-Milnor Type Theorems} \label{sec_gen_chen_milnor}
\setcounter{subsection}{1}
The Chen-Milnor Theorem gives a presentation of the nilpotent quotients of the group of a classical link. Here we will require Chen-Milnor type theorems for the groups and extended groups of both virtual links and virtual string links. In this section, we give a single Chen-Milnor Theorem which will later be applied to each of these scenarios. First we fix some notation. Let $G$ be a group and $x,y \in G$. We will write commutators as $[x,y]=x^{-1}y^{-1}xy$. Higher order left bracketed commutators are defined inductively by:
\[
[x_1,x_2,x_3\ldots,x_k]:=[[x_1,\ldots,x_{k-1}],x_k].
\]
For example, $[w,x,y,z,w]=[[[w,x],y],z]$. For $x \in G$, we will sometimes write $\overline{x}=x^{-1}$. For an arbitrary group $G$ and subgroups $A,B \subseteq G$, let $[A,B]$ denote the subgroup generated by elements of the form $[a,b]$ for $a \in A$ and $b \in b$. The \emph{lower central series} is defined by $G_1=G$ and $G_{k+1}=[G_{k},G]$ for $k \ge 1$. This gives the series of normal subgroups $G \triangleright G_2 \triangleright G_3 \triangleright \cdots$. The \emph{nilpotent residual} is defined to be $G_{\omega}=\bigcap_{i=1}^{\infty} G_i$. The \emph{nilpotent quotients of $G$} are the quotients $Q_{q}(G)=G/G_q$ for $q \ge 2$.

Let $\langle t_1,\ldots,t_k | r_1,\ldots,r_l \rangle$ be a presentation of finitely presented group $G$. The presentation is said to be a \emph{Wirtinger-type presentation} if for all $i$, $r_i=t_{j}^{-1} u_i^{-1} t_k u_i$ for some $j,k$ and some word $u_i$ in $t_1,\ldots,t_k$. If $G$ has a Wirtinger-type presentation, we may partition the set of generators $\{t_1,\ldots,t_k\}$ into $m$ conjugacy classes $C_1,\ldots,C_m$. Let $n_i=|C_i|$. 

\begin{definition}[Serial presentation] \label{defn_serial} The set $C_i$ is said to be \emph{cyclic} if there is an ordering $c_{i,1},\ldots,c_{i,n_i}$ of the elements of $C_i$ such that the relations defining $C_i$ can be written in the form $c_{i,j+1}=\overline{w}_{i,j} c_{i,j} w_{i,j}$ for $1 \le j \le n_i$, with the indices of $c_{i,j}$ taken modulo $n_i$. The relation for a cyclic $C_i$ of the form $c_{i,1}=\overline{w}_{i,n_i} c_{i,n_i} w_{i,n_i}$ will be called the \emph{return relation}. A conjugacy class containing a single generator and no defining relations will be called \emph{(trivially) cyclic}. A set $C_i$ will be called \emph{serial} if $n_i \ge 2$ and its elements may be ordered $c_{i,1},\ldots,c_{i,n_i}$ so that its defining relations are of the form $c_{i,j+1}=\overline{w}_{i,j} c_{i,j} w_{i,j}$ for $1 \le j \le n_i-1$. A Wirtinger-type presentation will be called \emph{serial} if each of its conjugacy classes $C_1,\ldots,C_m$ is either cyclic or serial.
\end{definition}

\begin{remark} Serial conjugacy classes are needed for virtual string links (see Section \ref{sec_vsl}).
\end{remark}

\begin{example} Let $G=\langle a_1,a_2,a_3,a_4|a_2=\overline{a}_1 a_2 a_1,a_4=\overline{\overline{a}}_1 a_3\wwbar{a}_1 \rangle$. Then $G$ is a serial Wirtinger-type presentation, where $C_1=\{a_2\}$ is cyclic, $C_2=\{a_3,a_4\}$ is serial, and $C_3=\{a_1\}$ is trivially cyclic.
\end{example}

If $C_i$ is cyclic, set $l_{i,j}=w_{i,1} w_{i,2} \cdots w_{i,j}$ for $1 \le j \le n_i$. If $C_i$ is serial, set $l_{i,j}=w_{i,1} w_{i,2} \cdots w_{i,j}$ for $1 \le j \le n_i-1$. Then in both cases, we have the equivalent set of defining relations $c_{i,j+1}=\bar{l}_{i,j} c_{i,1} l_{i,j}$ for $G$. In the cyclic case, the return relation is of the form $c_{i,1}=\bar{l}_{i,n_i} c_{i,1} l_{i,n_i}$, where $l_{i,n_i}$ is the identity $e \in G$ if $C_i$ is trivially cyclic.

\begin{definition}\label{defn_word} For a cyclic conjugacy class (serial conjugacy class), set $\lambda_i=l_{i,n_i}$ (respectively, $\lambda_i=l_{i,n_i-1}$). Then $\lambda_i$ will be called a \emph{parallel word} and the pair $(c_{i,1},\lambda_i)$  will be called a \emph{meridian-parallel word pair}. If the exponent sum of the letters $c_{i,j}$ in $\lambda_i$ is zero, then $\lambda_i$ will be called a \emph{longitude word}. 
\end{definition}

Suppose $G$ has a serial Wirtinger-type presentation and let $\Phi$ be the free group generated by the set of $c_{i,j}$. For $q \ge 1$ we follow Milnor \cite{milnor} and define maps $p^{(q)}:\Phi \to \Phi$ recursively by.
\begin{align*}
p^{(1)}(c_{i,j}) &=c_{i,j} \\
p^{(q)}(c_{i,1}) &=c_{i,1} \\
p^{(q+1)}(c_{i,j+1}) &= p^{(q)}(\bar{l}_{i,j} c_{i,1} l_{i,j})
\end{align*}

\begin{lemma} \label{thm_pq_commutes} If $G$ has a serial Wirtinger-type presentation as above, then $p^{(q)}:\Phi \to \Phi$ descends to an automorphism $p^{(q)}:Q_{q}(G)\to Q_{q}(G)$. Moreover, the following diagram commutes, where the horizontal arrows are the natural projections. 
\[
\xymatrix{Q_{2}(G) \ar[d]^{p^{(2)}}_{\cong} & \ar[l] Q_{3}(G) \ar[d]^{p^{(3)}}_{\cong} & \ar[l] \cdots & \ar[l] Q_{q}(G) \ar[d]^{p^{(q)}}_{\cong} & \ar[l] Q_{q+1}(G) \ar[d]^{p^{(q+1)}}_{\cong} & \ar[l] \cdots \\ 
Q_{2}(G)  & \ar[l] Q_{3}(G)  & \ar[l] \cdots & \ar[l] Q_{q}(G) & \ar[l] Q_{q+1}(G) & \ar[l] \cdots}
\]
\end{lemma}

\begin{proof} The proof follows that of Milnor \cite{milnor}, Theorem 4. First observe that if $a,b,x \in G$ and $a \equiv b \mod{G_{q}}$ then $\bar{a} x a \equiv \bar{b} x b \mod G_{q+1}$. Now, to see that $p^{(q)}$ is an isomorphism, it is sufficient to show that $p^{(q)}(c_{i,j}) \equiv c_{i,j} \mod G_{q}$. This will be proved by induction. Certainly the claim is true for $q=1$ and for all $q$ if $j=1$. Assume the claim is true up to $q$. Then $p^{(q)}(l_{i,j}) \equiv l_{i,j} \!\!\mod G_{q}$ and it follows from the initial observation that: 
\[
p^{(q)}(\bar{l}_{i,j} c_{i,1} l_{i,j})=p^{(q)}(l_{i,j})^{-1}p^{(q)}(c_{i,1}) p^{(q)}(l_{i,j}) \equiv \bar{l}_{i,j} c_{i,1} l_{i,j} \!\!\!\mod G_{q+1}.
\]
Thus, $p^{(q+1)}(c_{i,j+1})=c_{i,j+1} \mod G_{q}$ and the claim follows by induction. This implies that $p^{(q)}(G_{q}) \subseteq G_{q}$ and that $p^{(q)}$ is both injective and surjective on the nilpotent quotients. 

To prove that the diagram commutes, it is sufficient to show that $p^{(q)}(c_{i,j}) \equiv p^{(q+1)}(c_{i,j}) \!\mod \Phi_{q}$.  The claim holds if $q=1$ or $j=1$ by definition of $p^{(q)}$. Suppose then the claim holds up to $q$. Then $p^{(q)}(l_{i,j}) \equiv p^{(q+1)}(l_{i,j}) \!\mod \Phi_{q}$. Again by the initial observation:
\[
p^{(q+1)}(c_{i, j+1}) = p^{(q)}(\bar{l}_{i,j} c_{i,1} l_{i,j}) \equiv p^{(q+1)}(\bar{l}_{i,j}) c_{i,1} p^{(q+1)}(l_{i,j}) \equiv p^{(q+2)}(c_{i,j+1})\!\!\!\mod \Phi_{q+1}. 
\] 
Thus, $p^{(q)}(c_{i,j}) \equiv p^{(q+1)}(c_{i,j}) \!\mod \Phi_{q}$ and the lemma is proved.
\end{proof}

Let $G$ be a serial Wirtinger-type presentation and let $C_1,\ldots,C_r,C_{r+1}\ldots,C_m$ be the conjugacy classes of its generators, where $C_{r+1},\ldots,C_m$ are the cyclic classes and $r \le m$. For $1 \le i \le m$, let $c_i=c_{i,1}$. Then the return relations are given by $c_{i}=\bar{\lambda}_ic_i\lambda_i$, where $(c_{i},\lambda_i)$ is a meridian-parallel word pair. Let $F$ be the free group on $c_1,\ldots,c_m$. Define $\phi:\Phi \to F$ by $\phi(c_{i,j})=c_i$ and $\phi^{(q)}:\Phi \to F$ by $\phi^{(q)}=\phi\circ p^{(q)}$. Below is the Chen-Milnor type theorem that will be used throughout the text. The proof again follows Milnor \cite{milnor}, Theorem 4.

\begin{theorem} \label{thm_gen_chen_milnor} For a serial Wirtinger-type presentation of $G$ and notation as above, the nilpotent quotients of $G$ have the following presentation:  
\[
Q_{q}(G) \cong \langle c_{1},\ldots,c_r,c_{r+1},\ldots,c_{m}| [c_{r+1},\phi^{(q)}(\lambda_{r+1})],\ldots,[c_{m},\phi^{(q)}(\lambda_{ m})],F_{q} \rangle.
\]
\end{theorem}
\begin{proof} First note that $G/G_{2}=G/[G,G]$ is abelian. Then $p^{(2)}(c_{i,j}) \equiv c_{i,1} \mod G_{2}$ for all $i,j$. Indeed, this is true by definition if $j=1$ and if $j>1$, we have: 
\[
p^{(2)}(c_{i,j+1})=p^{(1)}(\bar{l}_{i,j} c_{i,1} l_{i,j})=\bar{l}_{i,j} c_{i,1} l_{i,j} \equiv c_{i,1} \mod G_{2}
\]
Since $p^{(q+1)}$ is defined recursively in terms of $p^{(q)}$, this implies that the image of $p^{(q)}$ is generated by $c_{1,1},\ldots,c_{1,m}$ for all $q\ge 2$. 

Let $N$ be the normal subgroup of $\Phi$ generated by the relations of $G$. Since $\phi(\Phi_{q})=F_{q}$, it follows that $\phi^{(q)}(\Phi_{q})=F_{q}$.  By Lemma \ref{thm_pq_commutes},
\[
p^{(q)}(c_{i,j+1}) \equiv p^{(q+1)}(c_{i,j+1})=p^{(q)}(\bar{l}_{i,j}c_{i,1}l_{i,j}) \mod \Phi_{q}.
\]
Thus, if $j\le n_i-1$, $\phi^{(q)}(\bar{c}_{i,j+1} \bar{l}_{i,j}c_{i,1} l_{i,j}) \in F_{q}$. If $ i \le r$, all relations for the conjugacy class $C_i$ are of this form. Furthermore, if $r+1 \le i \le m$, this implies that every relation of $G$ that is not a return relation is mapped to $F_{q}$. Thus, if $R$ is the normal subgroup of $F$ generated by the image of the return relations, $\phi^{(q)}$ descends to an isomorphism: 
\[
\phi^{(q)}:\frac{\Phi}{N\Phi_{q}} \stackrel{\cong}{\longrightarrow}  \frac{F}{RF_{q}}
\]
Lastly, observe that if $i \ge r+1$, the effect of $\phi^{(q)}$ on the return relation of $C_i$ is:
\[
\phi^{(q)}(\bar{c}_{i,1}\bar{l}_{i,n_i}c_{i,1}l_{i,n_i})=[c_i,\phi \circ p^{(q)}(l_{i,n_i})].
\]
Thus, $\phi^{(q)}$ is an isomorphism and the proof is complete. \end{proof}

\section{Milnor $\&$ extended Milnor invariants} 
\label{sec_milnor_extended milnor}
This section defines concordance invariants from the nilpotent quotients of $G(L)$ and the extension $\widetilde{G}(L)$. The case of $G(L)$ is treated first and these results are extended to $\widetilde{G}(L)$ using the $\Zh$ map. First, we give a direct proof that the virtual linking number is a welded-concordance invariant (Section \ref{sec_vlk}). The $\bar{\mu}$-invariants of virtual links can be viewed as higher order virtual linking numbers (see Section \ref{sec_shuffles}). Section \ref{sec_longitudes} discusses longitudes of virtual links and their relation to longitudes of ribbon torus links. The properties of these longitudes are needed in the proof that the $\bar{\mu}$-invariants are invariant under concordance (Section \ref{sec_mu}). The definition of the $\overline{\zh}$-invariants of virtual knots appears in Section \ref{sec_wbar_zh}.

\subsection{Virtual linking number} \label{sec_vlk} Recall that the \emph{writhe} of a virtual knot diagram is the sum of its crossing signs. The writhe is invariant only under $\Omega 2$, $\Omega 3$, and detour moves. If $J \cup K$ is a $2$-component oriented virtual link diagram, the \emph{virtual linking number} $\vlk(J,K)$ is the number of overcrossings of $J$ with $K$, counted with sign. The virtual linking number is unchanged by Reidemeister moves and detour moves and thus it is an invariant of virtual links. Unlike the classical linking number, the virtual linking number is not symmetric.  

\begin{example} For the link in Figure \ref{fig_writhe_example}, $\vlk(J,K)=1 \ne 0=\vlk(K,J)$. Also, both $J,K$ are equivalent to the unknot. However, $\text{writhe}(J)=0$ and $\text{writhe}(K)=-2$.
\end{example}

\begin{figure}[htb]
\centerline{
\begin{tabular}{c} \\
 \def\svgwidth{1.6in}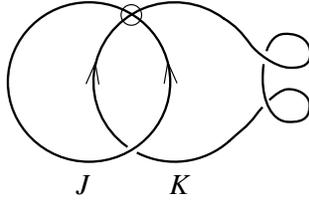
\end{tabular}
}
\caption{The virtual linking number is not symmetric.} \label{fig_writhe_example}
\end{figure}

\begin{proposition} \label{prop_vlk} The virtual linking number is a welded-concordance invariant.
\end{proposition}
\begin{proof} Let $J=J_1 \cup \cdots \cup J_m$ be an $m$-component virtual link diagram, $K=K_1 \cup \cdots \cup K_n$ be an $n$-component virtual link diagram, and $J \cup K$ an $(m+n)$-component virtual link diagram containing $J$ and $K$ as sublinks with no components in common. Define the \emph{total linking number} to be $\vlk(J,K)=\sum_{i=1}^m \sum_{j=1}^n \vlk(J_i,K_j)$. Clearly, $\vlk(J,K)$ is unchanged by Reidemeister moves, detours, and forbidden overpass moves. Suppose $J^*$ is an $(m+1)$-component virtual link diagram obtained from $J$ by performing a saddle move on $J_i$. Then $J_i$ is replaced with $J_i^* \cup J_i^{**}$. Since every overcrossing of $J_i$ with $K$ occurs as either an overcrossing of $J_i^*$ or $J_i^{**}$ with $K$, it follows that $\vlk(J^*,K)=\vlk(J,K)$. Similarly, if $K^*$ is obtained from $K$ by a saddle, then $\vlk(J,K)=\vlk(J,K^*)$. If $L \sqcup \bigcirc$ denotes a birth, then $\vlk(J \sqcup \bigcirc,K)=\vlk(J,K \sqcup \bigcirc)=\vlk(J,K)$.

Now suppose that $L=J \cup K$ and $L^*=J^*\cup K^*$ are welded-concordant virtual links. Then there is a sequence $L=L_1,\ldots,L_r=L^*$ of links such that $L_{i+1}$ can be obtained from $L_i$ by a single birth, death, saddle, detour, forbidden overpass, or Reidemeister move. The above argument implies that the total virtual linking number is unchanged at each step, and hence $\vlk(J,K)=\vlk(J^*,K^*)$. 
\end{proof}

\subsection{Virtual ribbons and longitudes} \label{sec_longitudes} We begin with a diagrammatic definition of a parallel of a virtual knot which is then related to the algebraic definition given in Section \ref{sec_gen_chen_milnor}. A longitude is simply a parallel of a virtual knot $\upsilon$ having virtual linking number $0$ with $\upsilon$ (see Definition \ref{defn_long} ahead). Lastly, we relate longitudes of a virtual link $L$ to longitudes of the ribbon torus link $\text{Tube}(L)$.

\begin{definition}[Virtual ribbon, parallel] A \emph{virtual ribbon} is an immersed annulus in $\mathbb{R}^2$ whose singularities are double points occurring only in classical or virtual crossings of bands, as in Figure \ref{fig_band_cross}. The image of the boundary of the virtual ribbon is a $2$-component virtual link diagram whose components are assumed to be co-oriented. If $K$ is a virtual knot diagram that is one component of a virtual ribbon boundary, then the other component is called a \emph{parallel} of $K$.   
\end{definition}

\begin{figure}[htb]
\centerline{
\begin{tabular}{cc}
\begin{tabular}{c}
\def\svgwidth{1.3in}
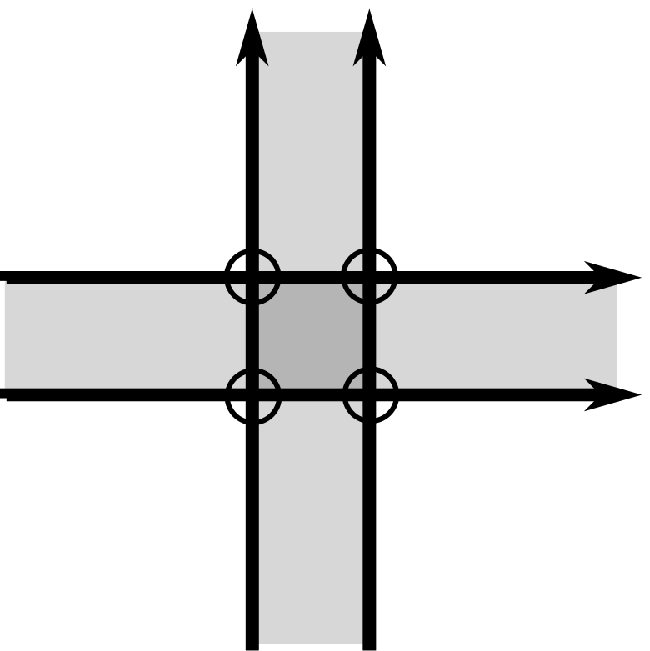
\end{tabular} &
\begin{tabular}{c}
\def\svgwidth{1.3in}
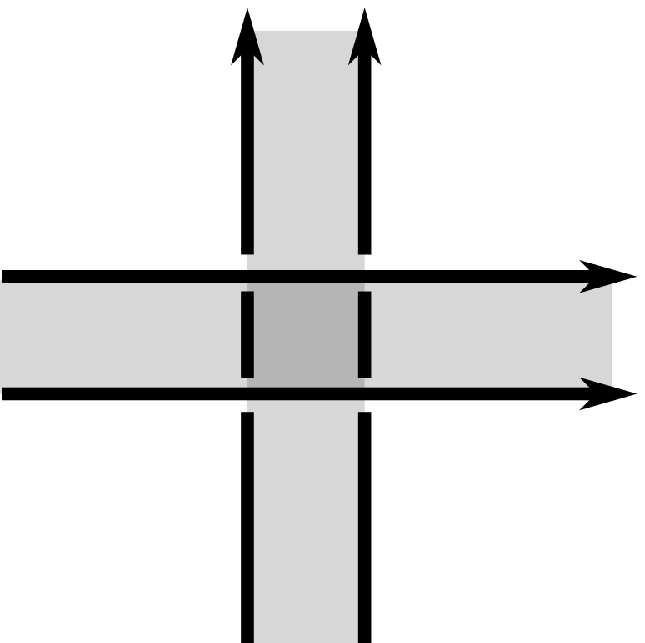
\end{tabular}
\end{tabular}
}
\caption{A virtual crossing of bands (left) and a classical crossing of bands (right).} \label{fig_band_cross}
\end{figure}

\begin{example} Given a diagram of $K$, a virtual ribbon can be found by thickening the arcs of $K$ and drawing the crossings of bands as in Figure \ref{fig_band_cross}. A virtual ribbon and a parallel of 2.1 is drawn on the left of Figure \ref{fig_ribbon_ex}. The center of Figure \ref{fig_ribbon_ex} shows a different virtual ribbon and a different parallel $\ell$ of $2.1$.
\end{example}

\begin{figure}[htb]
\centerline{
\begin{tabular}{ccc} \\
\begin{tabular}{c} \def\svgwidth{1.65in}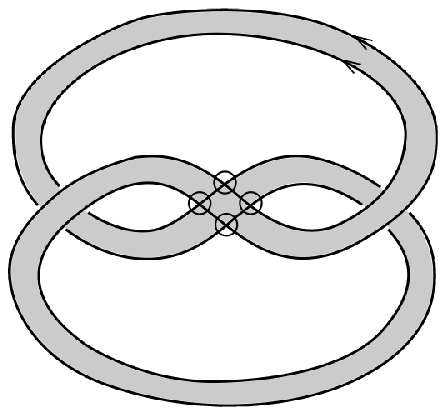 \end{tabular} &
\begin{tabular}{c} \def\svgwidth{1.65in}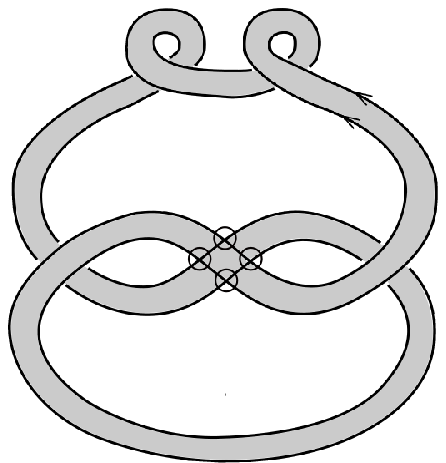 \end{tabular} &
\begin{tabular}{c} \def\svgwidth{1.65in}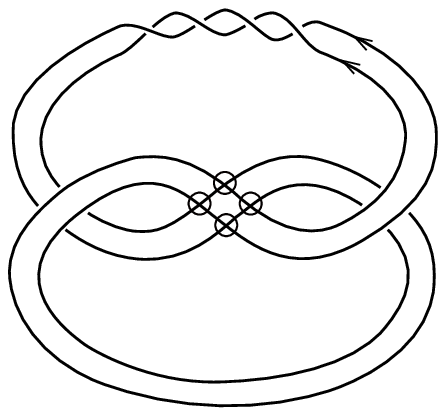 \end{tabular} 
\end{tabular}
}
\caption{A virtual ribbon showing a parallel $K'$ of $K$ (left), a virtual ribbon showing a longitude $\ell$ of $K$ (center), and $\ell$ without the virtual ribbon (right).} \label{fig_ribbon_ex}
\end{figure}

\begin{lemma} If $K'$ is a parallel of $K$. Then $K' \leftrightharpoons K$ and: 
\[
\vlk(K,K')=\vlk(K',K)=\text{writhe}(K)=\text{writhe}(K').
\]
\end{lemma}
\begin{proof} The first claim is clear from the definition. For the second claim, observe that each classical crossing of bands has four crossings: a crossing of $K$ with itself, a crossing of $K'$ with itself, an overcrossing of $K$ with $K'$, and an overcrossing of $K'$ with $K$. Since each of these crossings has the same sign, the claim follows.
\end{proof}

\begin{lemma} \label{lemma_parallels_wd} Let $K,K^*$ be virtual knot diagrams. If $K'$ is a parallel of $K$, $K\leftrightharpoons K^*$ (respectively, $K \leftrightharpoons_w K^*)$, and $\text{writhe}(K)=\text{writhe}(K^*)$, then there is a parallel $(K^*)'$ of $K^*$ such that $K \cup K' \leftrightharpoons K^* \cup (K^*)'$ (respectively, $K \cup K' \leftrightharpoons_w K \cup (K^*)'$).
\end{lemma}
\begin{proof} There is a sequence of moves from $K$ to $K^*$ such that all intermediate diagrams have the same writhe. Indeed, detours, $\Omega 2$, $\Omega 3$, and forbidden overpass moves do not change the writhe. If an $\Omega 1$ move appears in the sequence from $K$ to $K^*$, replace it by a pair of $\Omega 1$ moves as follows. If an $\Omega 1$ move of deletes a curl, add a single second curl with opposite crossing sign, so that the writhe remains unchanged. On the other hand, if an $\Omega 1$ move adds a single curve, add a pair of oppositely signed curls. The same sequence of moves can be applied to the virtual ribbons, where the effect of adding/deleting oppositely signed curls is to add/delete a pair of oppositely signed full twists (see Figure \ref{fig_opp_twists_cancel}). Hence, $K'$ is transported to a parallel $(K^*)'$ of $K^*$.
\end{proof}

\begin{figure}[htb]
\centerline{
\begin{tabular}{c} \\
\def\svgwidth{4.5in}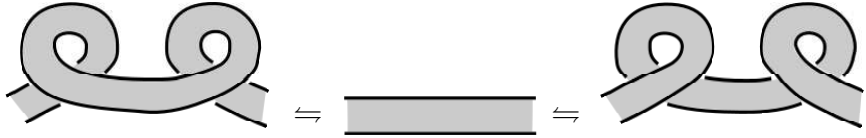
\end{tabular}
}
\caption{Opposite twists cancel.} \label{fig_opp_twists_cancel}
\end{figure}

\begin{definition}[Longitude] \label{defn_long} If $\ell$ is a parallel of $K$ and $\vlk(K,\ell)=0$, then $\ell$ is called a \emph{longitude} of $K$. If $K \cup K'$ is any 2-component link such that $K \cup K' \leftrightharpoons K \cup \ell $ for a longitude $\ell$ of $K$, then $K'$ will also be called a longitude of $K$.
\end{definition}

\begin{theorem} Longitudes are well-defined up to both virtual and welded equivalence.
\end{theorem}
\begin{proof} Let $L$ be an arbitrary virtual link diagram and suppose $L^*$ is another diagram with $L \leftrightharpoons L^*$. Using $\Omega 1$ moves, it may be ensured that each component of $L$ and $L^*$ has zero writhe. A virtual ribbon $R$ may be obtained by thickening each component $K$ of $L$. The boundary of $R$ is $K \cup \ell$, where $\ell$ is a longitude of $K$. By Lemma \ref{lemma_parallels_wd}, the parallel $\ell$ of $K$ is transported to a parallel $\ell^*$ of the corresponding component $K^*$ of $L$. Since the virtual linking number is preserved by Reidemeister and detour moves, this parallel is also a longitude.
\end{proof}

Let $K$ be a component of a virtual link diagram $L$. A meridian-parallel word pair in $G(L)$ can be recovered from a parallel $K'$ of $K$ as in the case of classical links. Choose an initial arc to travel along $K'$ which is parallel to some arc $a$ of $K$. While traversing the parallel, if an arc labeled with the generator $z$ of $G(L)$ crosses over $K'$, write the letter $z^{\pm 1}$ according to whether the crossing sign is $\pm$, respectively. The resulting word $\lambda$ is a parallel and $(a,\lambda)$ is a meridian-parallel word pair for the conjugacy class of generators of $G(L)$ corresponding to the arcs of $K$.

\begin{lemma} \label{lemma_meridian_longitude} The following hold for a meridian-longitude word pair $(a,\lambda)$. 
\begin{enumerate}
\item The exponent sum in $\lambda$ of the generators of $G(L)$ in the conjugacy class of $a$ is $0$.
\item The pair $(a,\lambda)$ is well-defined up to simultaneous conjugation in $G(L)$. 
\end{enumerate}
\end{lemma}
\begin{proof} The first claim follows from the fact that $\vlk(K,K')=0$. For the second claim, it is sufficient to show that the pair is well-defined up to conjugation by a single generator $z$ of $G(L)$. We will show that $(z a \bar{z},z \lambda \bar{z})$ is also a meridian-longitude pair for a link $L^*$ equivalent to $L$. Using detour moves and a single $\Omega 2$ move, the arc $z$ may be positioned so that it appears as in Figure \ref{fig_simultaneous_conjugation}, right. The new link $L^*$ is equivalent to $L$ and $G(L^*)$ has a presentation with two more generators $x,a^*$ and two more relations:
\[
a= \bar{z} x z,\,\,\, x = z a^* \bar{z}
\]  
Substituting the second relation into the first gives $a^*=a$ and $x=z a \bar{z}$. If a longitude word $\lambda^*$ is read off starting from $x$, we have $\lambda^*=z \lambda \bar{z}$. Since this does not affect the exponent sum of generators conjugate to $a$, $(z a \bar{z},z \lambda \bar{z})$ is also a meridian-longitude word pair. 
\end{proof}

\begin{figure}[htb]
\centerline{
\begin{tabular}{c} 
 \def\svgwidth{4.5in}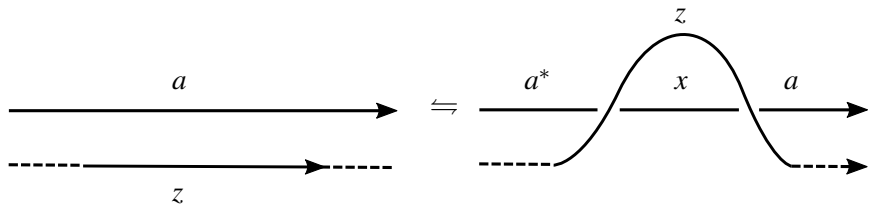
\end{tabular}
}
\caption{Meridian-longitude word pairs are well-defined up to simultaneous conjugation.} \label{fig_simultaneous_conjugation}
\end{figure}
 
Next we describe the correspondence between longitudes of virtual links and longitudes of ribbon torus links. First recall the isomorphism $G(L) \to \pi_1(S^4 \smallsetminus \text{Tube}(L))$ where $L=K_1 \cup \cdots \cup K_m$ is a virtual link diagram (see Audoux et al. \cite{abmw}). The broken surface diagram of $\text{Tube}(L)$ has a well-defined notion of \emph{inner annuli} and \emph{outer annuli}, as in Figure \ref{fig_m_alpha}, left. Fix a base point $b$ in the exterior of $\text{Tube}(L)$. A \emph{meridian} $m_{\alpha}$ of an outer annulus $\alpha$ is a closed path from $b$ that encircles once a point in $\partial \alpha$ (see Figure \ref{fig_m_alpha}, right). The arcs of $L$ are in one-to-one correspondence with the outer annuli of $\text{Tube}(L)$. The classical crossings of $L$ are in one-to-one correspondence with the inner annuli of $\text{Tube}(L)$.  There is a presentation for $\pi_1(S^4\smallsetminus \text{Tube}(L),b)$ whose generators are the meridians of the outer annuli and whose relations are of the form $m_{\delta}=m_{\beta}^{-\varepsilon} m_{\alpha} m_{\beta}^{\varepsilon}$. The map $a \to m_{\alpha}$ is an isomorphism $G(L) \cong \pi_1(S^4 \smallsetminus \text{Tube}(L),b)$ (\cite{abmw}, Propositions 2.20 and 3.13). Note that the abelianization of $\pi_1(S^4 \smallsetminus \text{Tube}(L),b)$ is $H_1(S^4 \smallsetminus \text{Tube}(L))\cong \mathbb{Z}^m$, where the $i$-th factor of $\mathbb{Z}$ is generated by a meridian of $\text{Tube}(K_i)$. Given a simple closed curve $\gamma$ in $S^4 \smallsetminus \text{Tube}(L)$, the $i$-th coordinate of $[\gamma] \in H_1(S^4 \smallsetminus \text{Tube}(L))$ is the linking number of $\gamma$ with $\text{Tube}(K_i)$.

\begin{figure}[htb]
\centerline{
\begin{tabular}{c}
\def\svgwidth{3.75in}\input{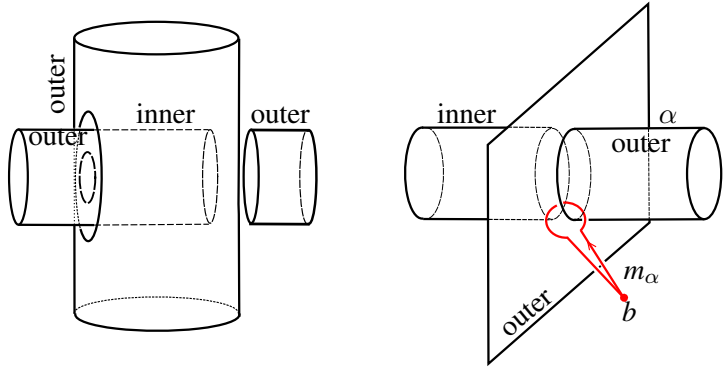}
\end{tabular}
}
\caption{A meridian $m_{\alpha}$ of an outer annulus $\alpha$.} \label{fig_m_alpha}
\end{figure}

A parallel $\gamma$ of $\text{Tube}(K_i)$ may be drawn directly on the broken surface diagram as follows. Choose a point $P$ on some annulus $\alpha$ of the broken surface diagram. Travel along $\text{Tube}(K_i)$ in the direction of $K_i$ and return back to $P$. Let $N=N(\text{Tube}(K_i))\approx S^1 \times S^1 \times B^2$ be a tubular neighborhood of $\text{Tube}(K_i)$ and push $\gamma$ off $\text{Tube}(K_i)$ into $\partial N$. The resulting curve is a \emph{parallel} of $\text{Tube}(K_i)$. Note that $H_1(\partial N)\cong H_1(S^1 \times S^1 \times S^1) \cong \mathbb{Z}\oplus \mathbb{Z} \oplus \mathbb{Z}$. For a parallel $\gamma$, the third integer coordinate of $[\gamma]\in H_1(\partial N)$ is the linking number of $\gamma$ with $\text{Tube}(K_i)$. If two parallels differ in homology only in the second coordinate, then they are isotopic in the exterior of $\text{Tube}(K_i)$. A \emph{longitude}\footnote{In \cite{abmw}, Audoux et al. call this the \emph{preferred longitude}.} of $\text{Tube}(K_i)$ is parallel having zero linking number with $\text{Tube}(K_i)$.

\begin{lemma} \label{lemma_tube_homotopy} Let $L=K_1 \cup \cdots \cup K_m$ be a virtual link diagram, $\ell_i$ be a longitude of $K_i$, and $(a_i,\lambda_i)$ some meridian-longitude word pair for $\ell_i$. Then the isomorphism $G(L) \to \pi_1(S^4\smallsetminus \text{Tube}(L),b)$ sends $\lambda_i$ to the homotopy class of a longitude of $\text{Tube}(K_i)$.
\end{lemma}
\begin{proof} Using $\Omega 1$ moves, we may assume $\text{writhe}(K_i)=0$. The isomorphism $G(L) \to \pi_1(S^4\smallsetminus \text{Tube}(L),b)$ sends $a \to m_{\alpha}$, where $a$ is an arc of $L$ and $\alpha$ is the corresponding annulus in $\text{Tube}(L)$. Then $\lambda_i$ is mapped to the product of meridians in $\text{Tube}(L)$. The merdians in this product correspond the successive undercrossings met while traversing $\text{Tube}(K_i)$ from a point on the initial annulus $\alpha_i$. Thus, $\lambda_i$ is mapped to the homotopy class of a parallel $\gamma_i$ of $\text{Tube}(K_i)$. Since $\text{writhe}(K_i)=0$, the linking number of $\gamma_i$ and $\text{Tube}(K_i)$ is zero. Hence, $\gamma_i$ is a longitude of $\text{Tube}(K_i)$.
\end{proof}

\subsection{Milnor invariants} \label{sec_mu} In this section, we define $\bar{\mu}$-invariants of multi-component virtual links and prove they are invariant under welded concordance. In the case of virtual links, the Chen-Milnor type theorem appears as follows. 

\begin{theorem} \label{cor_chen_milnor_links} Let $L=K_1 \cup \cdots \cup K_m$ be a virtual link diagram and for $1 \le i \le m$, let $(a_i,\lambda_i)$ be a meridian-longitude word pair for $K_i$ in $L$. Let $F$ be the free group on $a_1,\ldots,a_m$. Then the nilpotent quotients of $G=G(L)$ have a presentation of the form:
\[
Q_{q}(G) \cong \langle a_{1},\ldots,a_{m}| [a_{1},\phi^{(q)}(\lambda_1)],\ldots,[a_{m},\phi^{(q)}(\lambda_{ m})],F_q \rangle,
\]
\end{theorem}

\begin{remark} The above result was obtained previously by Dye and Kauffman \cite{dye_kauffman} in the case of virtual link homotopy. A geometric proof for links in thickened surfaces can be found in \cite{bbc}, Theorem 3.6.
\end{remark}

\begin{proof} $G$ has a serial Wirtinger-type presentation where there is one cyclic conjugacy class of generators for each $K_i$. By adding curls (i.e. $\Omega 1$ moves) to each $K_i$, we may assume that $\text{writhe}(K_i)=0$ for all $i$. Thus it may be assumed that each parallel word is a longitude word. The result now follows from Theorem \ref{thm_gen_chen_milnor}.   
\end{proof}

Recall the \emph{Magnus expansion} of a word in the free group $F$ on $m$ letters $a_1,\ldots,a_m$. Let $Y=\{a_1,\ldots,a_m\}$, and let $\mathbb{Z}[[Y]]$ denote the ring of formal power series in the non-commuting variables $a_1,\ldots,a_m$. For a word $f \in F$, the Magnus expansion associates an element $\epsilon(f)\in\mathbb{Z}[[Y]]$ via the map $a_i \to 1+a_i$ and $a_i^{-1} \to 1-a_i+a_i^2-\ldots$. For an (unreduced) word of the form $y_{j_1} y_{j_2} \cdots y_{j_r}$, with $y_{j_k} \in Y$, let $J$ denote the sequence $j_1j_2\cdots j_r$. Then there are integers $\epsilon_J(f)$ such that : 
\[
\epsilon(f)=1+\sum_{J=j_1j_2\cdots j_r} \epsilon_J(f) a_{j_1} a_{j_2} \cdots a_{j_r}.
\]
The Magnus expansion gives necessary and sufficient conditions for determining how far an element of $F$ lies in the lower central series: $f \in F_{q+1}$ if and only if $\epsilon_J(f)=0$ for all sequence $J$ of length $|J|\le q$ (see e.g. Fenn \cite{fenn}, Corollary 4.4.5).

Let $L=K_1 \cup \cdots \cup K_m$ be an $m$-component virtual link diagram and let $G=G(L)$. Let $J=j_1j_2j_3\cdots j_s$ be a sequence of numbers from the set $\{1,\ldots,m\}$. For $j \in \{1,\ldots,m\}$ define $J|j=j_1j_2\cdots j_sj$. Choose $q>s$ and let $(a_{j},\lambda_j)$ be a meridian-longitude word pair for $K_j$ and set $\lambda_j^{(q)}=\phi^{(q)}(\lambda_j)$. We define integers $\mu_{J|j}(L)$ by:
\[
\mu_{J|j}(L)=\epsilon_J(\lambda_j^{(q)}).
\]
It follows from Lemma \ref{thm_pq_commutes} that the value $\mu_{J|j}(L)$ is independent of $q$ if $q>s$. For a sequence $J=j_1\cdots j_r$, let $\Delta_J$ denote the greatest common divisor of $\mu_{\hat{J}}(L)$, where $\hat{J}$ ranges over all sequences obtained from $J$ by deleting one or more terms and cyclically permuting all of the remaining terms. Then define:
\[
\bar{\mu}_{J}(L) \equiv \mu_J(L) \pmod{\Delta_J}.
\]
To prove that $\bar{\mu}_J$ is a welded-concordance invariant, we need the following result of Stallings.

\begin{lemma}[Stallings \cite{stallings_central}, Theorem 3.4] \label{thm_stallings} Let $h:G \to G^*$ be a homomorphism of groups $G,G^*$ such that $H_1(h):H_1(G) \to H_1(G^*)$ is an isomorphism and $H_2(h):H_2(G) \to H_2(G^*)$ is a surjection. Then $h$ induces an isomorphism on the nilpotent quotients $Q_q(G) \to Q_q(G^*)$ for all $q \ge 2$.
\end{lemma}

\begin{lemma}\label{thm_bbc_nil_quo} Suppose $L,L^*$ are welded-concordant virtual links. Let $\Lambda$ be a concordance between $T=\text{Tube}(L)$ and $T^*=\text{Tube}(L^*)$. Then:
\[
Q_{q}(G(L)) \cong Q_{q}(\pi_1(S^4 \smallsetminus T)) \cong Q_{q}(\pi_1(S^4 \times I \smallsetminus \Lambda)) \cong Q_{q}(\pi_1(S^4 \smallsetminus T^*)) \cong Q_{q}(G(L^*)),
\]
and the isomorphisms preserve longitude words.
\end{lemma}
\begin{proof} There are inclusions $S^4 \smallsetminus T = S^4 \times 0 \smallsetminus \Lambda \subset S^4 \times I \smallsetminus \Lambda$ and $S^4 \smallsetminus T^* = S^4 \times 1 \smallsetminus \Lambda\subset S^4 \times I \smallsetminus \Lambda$. By a Mayer-Vietoris argument, there are isomorphisms $H_i(S^4 \smallsetminus T) \cong H_i(S^4 \times I \smallsetminus \Lambda) \cong H_i(S^4 \smallsetminus T^*)$ for $i=1,2$. The first claim then follows from Lemma \ref{thm_stallings} and Theorem \ref{thm_satoh}. 

By Lemma \ref{lemma_tube_homotopy}, longitude words will be preserved if the isomorphisms $Q_{q}(\pi_1(S^4 \smallsetminus T)) \cong Q_{q}(\pi_1(S^4 \times I \smallsetminus \Lambda)) \cong Q_{q}(\pi_1(S^4 \smallsetminus T^*))$ preserve longitudes. Let $T_i,T_i^*$ be corresponding components of $T,T^*$ and let $\Lambda_i\approx S^1 \times S^1 \times I$ be the component of $\Lambda$ giving a concordance between them. Let $\gamma_i$ be a longitude of $T_i$ and $N(\Lambda_i)\approx \Lambda_i \times B^2$ be a closed tubular neighborhood of $\Lambda_i$. Observe that there is an isotopy taking $\gamma_i$ to some parallel $\gamma_i^*$ of $T_i^*$ in $N(\Lambda_i)\smallsetminus \Lambda_i$. Again using a Mayer-Vietoris argument, $H_1(S^4 \smallsetminus T_i) \cong H_1(S^4 \times I \smallsetminus \Lambda_i) \cong H_1(S^4 \smallsetminus T_i^*)$. Since $\gamma_i$ is a longitude of $T_i$, $[\gamma_i]=0 \in H_1(S^4 \smallsetminus T_i)$. Hence, $[\gamma_i^*]=0 \in H_1(S^4 \smallsetminus T_i^*)$ and $\gamma_i^*$ is a longitude of $T_i^*$.
\end{proof}

\begin{theorem}\label{thm_mubar_links} If $L, L^*$ are welded-concordant virtual links, then for all sequences $J$ with $|J|\ge 2$:
\[
\bar{\mu}_{J}(L) \equiv \bar{\mu}_J(L^*) \pmod{\Delta_J}.
\]
\end{theorem}

\begin{proof} Let $G=G(L)$ and $G^*=G(L^*)$. Since $L,L^*$ are welded concordant, Lemma \ref{thm_bbc_nil_quo} implies that $Q_q(G) \cong Q_q(G^*)$ and that the isomorphism maps a longitude word for each component of $L$ to a longitude word of the corresponding component of $L^*$. By Lemma \ref{lemma_meridian_longitude}, the meridian-longitude word pairs of $L,L^*$ are well-defined up to simultaneous conjugation in $G,G^*$, respectively. Theorem \ref{cor_chen_milnor_links} therefore implies it is sufficient to prove that the residue classes $\bar{\mu}_J$ are preserved if (1) $a_j$ is replaced with a conjugate for some $1 \le j \le m$, (2) $\phi^{(q)}(\lambda_i)$ is replaced with a conjugate, (3) $\phi^{(q)}(\lambda_i)$ is multiplied by an element of the form $g^{-1}[a_j,\phi^{(q)}(\lambda_j)]g$, and (4) $\phi^{(q)}(\lambda_i)$ is multiplied by an element of $F_{q}$. These facts follow exactly as Milnor \cite{milnor}, Theorem 5, and hence their proofs are omitted (see also Fenn \cite{fenn}, Theorem 5.8.1). 
\end{proof}

\subsection{Extended Milnor invariants} \label{sec_wbar_zh} If $L,L^*$ are concordant virtual links, then $\Zh(L),\Zh(L^*)$ are semi-welded concordant. Pre-composing the $\bar{\mu}$-invariants of $(m+1)$-component links with $\Zh$ then gives a new family of concordance invariants of $m$-component links.

\begin{definition}[Extended $\bar{\mu}$-invariants] Let $L$ be an virtual link diagram. The \emph{extended $\bar{\mu}$-invariants of $L$} are the $\bar{\mu}$-invariants of $\Zh(L)$. 
\end{definition}

For computational purposes, it is convenient to relate these directly to the nilpotent quotients of extended groups. Theorem \ref{thm_gen_chen_milnor} implies a Chen-Milnor type theorem for extended groups. This will be phrased in terms of extended longitude words.

\begin{definition}[Extended meridian-longitude word] Let $L=K_1 \cup \cdots \cup K_m$ be a virtual link diagram. A meridian-longitude word $(a_i,\widetilde{\lambda}_i)$ for the component $K_i$ of $\Zh(L)$ in $G(\Zh(L))\cong \widetilde{G}(L)$ is called an \emph{extended meridian-longitude word pair} and $\widetilde{\lambda}_i$ is called an \emph{extended longitude word}.
\end{definition}

\begin{theorem} Let $L=K_1 \cup \cdots \cup K_m$ be an $m$-component virtual link diagram and $\widetilde{G}=\widetilde{G}(L)$. For $1 \le i \le m$, let $(a_i,\widetilde{\lambda}_i)$ be an extended meridian-longitude word pair for $K_i$. Let $F=F(m+1)$ be the free group on $a_1,\ldots,a_m,v$. Then the nilpotent quotients of $\widetilde{G}$ are given by: 
\[
Q_{q}(\widetilde{G}) \cong \langle a_{1},\ldots,a_{m},v| [a_{1},\phi^{(q)}(\widetilde{\lambda}_1)],\ldots,[a_{m},\phi^{(q)}(\widetilde{\lambda}_{ m})],F_{q} \rangle.
\]
\end{theorem}
\begin{proof} The $\omega$ component of $\Zh(L)$ has only overcrossings with the components of $L$. This implies that the conjugacy class of the generator $v$ has one element and no relations. The only nontrivial return relations (see Definition \ref{defn_serial}) are thus for the components of $L$ in $\Zh(L)$. The result then follows from Theorem \ref{thm_gen_chen_milnor}. 
\end{proof}

An immediate consequence of the above theorem is that $\mu_{J|(m+1)}(\Zh(L))=0$ for all $J$. In general, the nilpotent quotients of $\widetilde{G}(L)$ are stronger concordance invariants than those of $G(L)$. In the case of a virtual knot $K$, $Q_q(G(K)) \cong \langle a|\rangle$ for all $q \ge 2$ while $Q_q(\widetilde{G}(K)) \cong \langle a,v| [a,\widetilde{\lambda}],F(2)_q \rangle$. Hence, the nilpotent quotients can be highly nontrivial even for virtual knots. This will be apparent from the calculations in Section \ref{sec_slice_obstructions}. As there is only one extended longitude of interest in the case of virtual knots, we make the following definition.

\begin{definition}[$\overline{\zh}$-invariants of virtual knots] Let $K$ be a virtual knot diagram and  let $J$ be a sequence in the integers $\{1,2\}$. Choose $q > |J|$ and  $\widetilde{\lambda}^{(q)}=\phi^{(q)}(\widetilde{\lambda})$, where $(a,\widetilde{\lambda})$ is an extended meridian-longitude word pair. Define:
\begin{align*}
\overline{\zh}_J(K) &\equiv \epsilon_J(\widetilde{\lambda}^{(q)}) \equiv \mu_{J|1}\left(\Zh(K) \right) \pmod{\Delta_{J|1}}
\end{align*} 
Note that the terminal entry of the sequence defining the $\overline{\zh}$-invariants is not used here to specify the longitude to which the Magnus expansion is applied. The collection of the invariants $\overline{\zh}_J(K)$ will be called the \emph{$\overline{\zh}$-invariants} of $K$.
\end{definition}

\section{Artin $\&$ extended Artin representations of virtual string links} \label{sec_artin}

Following Habegger-Lin \cite{HL2}, we obtain concordance invariants for virtual string links from the Artin representation of $\vec{L}$ (Section \ref{sec_artin_links}). Using the $\Zh$ map, we obtain new concordance invariants of virtual string links and long virtual knots. These are the Artin representations of extended groups (Section \ref{sec_ex_artin}). Section \ref{sec_vsl} gives some background on virtual string links.    

\subsection{Virtual string links, $\Zh$, and Tube} \label{sec_vsl} An $m$-component \emph{virtual string link} $\vec{L}$ is a proper immersion of $m$ closed intervals $I=[0,1]$ into the rectangle $[0,1] \times [0,m+1]$, $\bigsqcup_{i=1}^{m}I \to [0,1] \times [0,m+1]$, so that the $i$-th copy of $I$ satisfies $0\to (0,i)$ and $1 \to (1,i)$. The transversal double points are each marked as either a classical or a virtual crossing. Each component is oriented from left to right. Two virtual string links are said to be \emph{equivalent} if they may be obtained from one another by a sequence of Reidemeister moves and detours that fix the endpoints of the intervals. Welded and semi-welded equivalence of virtual string links are defined similarly (see Definition \ref{defn_semi}). The virtual string link defined by $t_i \to (t_i,i)$ for $t_i$ in the $i$-th copy of $I$ is called the \emph{trivial string link}. A $1$-component virtual string link is called a \emph{long virtual knot}. 

Two $1$-component virtual string links $\vec{L}$ and $\vec{L}^*$ are said to be \emph{concordant} if there is a finite sequence of Reidemeister moves, detours, births, deaths, and saddles from $\vec{L}$ to $\vec{L}^*$ such that Equation \eqref{eqn_euler} is satisfied. Two $m$-component virtual string links are said to be \emph{concordant} if there is a cobordism that restricts to a concordance on each component. \emph{Welded concordance} and \emph{semi-welded concordance} of virtual string links are defined analogously (see Definition \ref{defn_semi}).

\begin{figure}[htb]
\centerline{
\begin{tabular}{c}
\def\svgwidth{5in}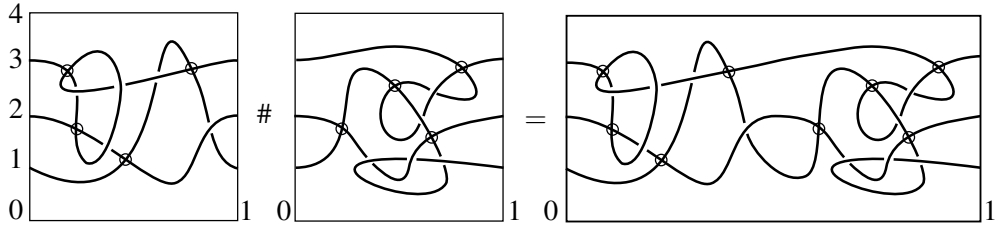
\end{tabular}
}
\caption{The concatenation of two virtual string links.} \label{fig_string_link_sum}
\end{figure}

The concordance classes of $m$-component virtual string links have a group structure defined by the concatenation operation. The concatenation $\vec{L}\#\vec{L}^*$ is formed by placing $\vec{L}^*$ to the right of $\vec{L}$ and scaling the $x$-coordinate by a factor of $1/2$ (see Figure \ref{fig_string_link_sum}). The inverse of $\vec{L}$ is found by reflecting $\vec{L}$ along the line $x=1$ and changing the orientation of each component. A proof that this defines a group structure can be found in \cite{band_pass}, Theorem 1. For $1$-component virtual string links, we obtain the concordance group of long virtual knots, or equivalently, Turaev's group of long knots on surfaces (see \cite{turaev_cobordism}, Section 6.5).

\begin{definition}[Long virtual knot concordance group $\mathscr{VC}$] \label{defn_vc} The \emph{long virtual knot concordance group}, denoted $\mathscr{VC}$, is the set of concordance classes of long virtual knots under the concatenation relation, with identity given by the trivial long virtual knot.
\end{definition}

\begin{figure}[htb]
\centerline{
\begin{tabular}{c}
\def\svgwidth{4in}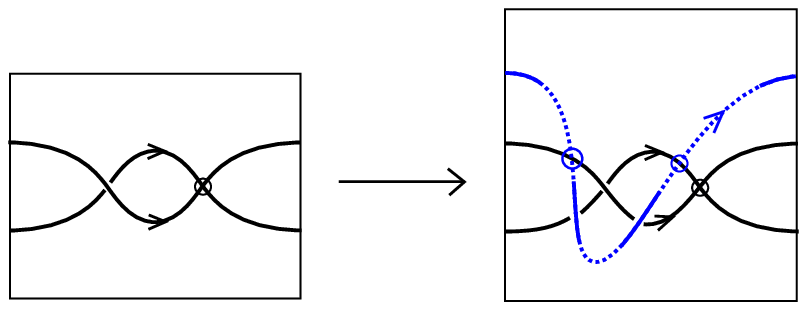
\end{tabular}
}
\caption{$\Zh$ for string links.} \label{fig_zh_string}
\end{figure}

The $\Zh$ map of an $m$-component virtual string link $\vec{L}$ is defined as follows. The first $m$-components are obtained by composing $\vec{L}$ with the inclusion $[0,1] \times [0,m+1] \hookrightarrow [0,1] \times [0,(m+1)+1]$. At each crossing of $\vec{L}$, we draw arcs of the new component $\vec{\omega}$ of $\Zh(\vec{L})$ as in Figure \ref{fig_zh_defn}. Starting at $(0,m+1) \in [0,1] \times [0,(m+1)+1]$ and ending at $(1,m+1)$, the arcs of $\vec{\omega}$ are connected together in an arbitrary order. If any new crossings are introduced, those new crossings are marked as virtual. The argument illustrated in Figure \ref{fig_zh_wd} shows that $\Zh(\vec{L})$ is well defined up to semi-welded equivalence. The same argument used to prove Theorem \ref{thm_semi_weld_conc} has the following immediate corollary.

\begin{corollary} If $\vec{L},\vec{L}^*$ are concordant virtual string links, then $\Zh(\vec{L}),\Zh(\vec{L}^*)$ are semi-welded concordant.
\end{corollary}

The Tube map may also be applied to a virtual string link diagram $\vec{L}$ (see Figure \ref{fig_tube_string}). The image $\text{Tube}(\vec{L})$ is a broken surface diagram of a \emph{ribbon tube}. Ribbon tubes are defined as follows. Let $B=B^2$ and choose distinct points $p_1,\ldots,p_m$ in $\text{int}(B)$. Let $B_1,\ldots,B_m$ be pairwise disjoint discs in $\text{int}(B)$ with $B_i$ centered at $p_i$. Henceforth we identify $B^3$ with $B^2 \times I$ and we set $B=B^2 \times \frac{1}{2}$. An \emph{$m$-component ribbon tube} is a smooth and proper embedding $\vec{T}$ of a disjoint union of $m$ annuli $A_1,\ldots,A_m$ into $B^3 \times I$,   such that for each $i$, the boundary components $S^1\times 0$ and $S^1 \times 1$ of $A_i\approx S^1 \times I$,  are identified with $\partial B_i \times 0$ and $\partial B_i \times 1$. Furthermore, it is required that the collection of $m$ $2$-spheres $(B_i \times 0) \cup A_i \cup (B_i \times 1)$ bound a collection of immersed $3$-balls that intersect themselves and each other only in ribbon singularities. This last requirement will not be used directly here; the interested reader is referred to Audoux et al. \cite{abmw} for more details. For a ribbon tube $\vec{T}$ in $B^3 \times I$, set $\partial_j \vec{T}=\vec{T} \cap B^3 \times j$ for $j=0,1$. Note that $\partial_j \vec{T}$ is an unlink $U$ in $B^3 \times j$. 

\begin{figure}[htb]
\centerline{
\begin{tabular}{c} \\ \\
\def\svgwidth{4in}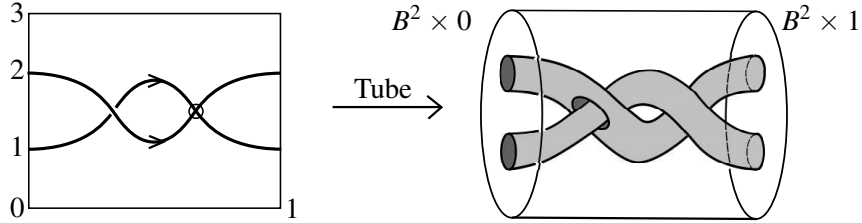
\end{tabular}
}
\caption{Tube for virtual string links.} \label{fig_tube_string}
\end{figure}

Two $m$-component ribbon tubes $\vec{T},\vec{T}^*$ are said to be \emph{concordant} if there is a smooth and proper embedding $\Lambda$ of $\bigsqcup_{i=1}^m (S^1 \times I) \times I$ into $B^3 \times I \times I$ such that $\bigsqcup_{i=1}^m (S^1 \times I) \times 0$ is identified with $-T$ and $\bigsqcup_{i=1}^m (S^1 \times I) \times 1$ is identified with $T^*$. The same argument as used in the proof of Theorem \ref{thm_tube_conc} shows that a welded concordance of virtual string links $\vec{L},\vec{L}^*$ gives a handle decomposition of a concordance between $\text{Tube}(\vec{L})$, $\text{Tube}(\vec{L}^*)$. Hence we have the following result.

\begin{corollary} \label{cor_tube_conc} If $\vec{L},\vec{L}^*$ are welded-concordant virtual string link diagrams, then $\text{Tube}(\vec{L})$ and $\text{Tube}(\vec{L}^*)$ are concordant ribbon tubes.
\end{corollary}

Lastly we define groups and extended groups of virtual string links. An \emph{arc} of $\vec{L}$ is a path on the diagram of $\vec{L}$ between two undercrossings or an undercrossing and a string endpoint. A \emph{short arc} of $\vec{L}$ is a path on the diagram of $\vec{L}$ between two classical crossings or a classical crossing and an endpoint. Virtual crossings are ignored in both arcs and short arcs.

\begin{definition}[The group $G(\vec{L})$ and its extension $\widetilde{G}(\vec{L})$] Let $\vec{L}$ be a virtual string link. The \emph{group $G(\vec{L})$ of $\vec{L}$} is the group whose generators are the arcs of $\vec{L}$ and whose relations are as shown in Figure \ref{fig_group_rels}. The \emph{group extension $\widetilde{G}(\vec{L})$ of $G(\vec{L})$} is the group whose generators are the short arcs of $\vec{L}$ and whose relations are as shown in Figure \ref{fig_ext_group_rels}.  
\end{definition}

As in the case of virtual knots, the group of $\vec{L}$ can be realized geometrically by the Tube map. Meridians of ribbon tubes are defined exactly as in the case of ribbon torus links (see Section \ref{sec_longitudes}). If $a$ is an arc of $\vec{L}$ and $\alpha$ its corresponding annulus in $\text{Tube}(\vec{L})$, the map $a \to m_{\alpha}$ is an isomorphism $G(\vec{L})\to \pi_1(B^3 \times I \smallsetminus \text{Tube}(\vec{L}),b)$.  The composition $\text{Tube}\circ\Zh$ then gives a geometric realization of $\widetilde{G}(\vec{L})$. These facts are recorded below.

\begin{theorem}[Audoux et al. \cite{abmw},Proposition 3.13] \label{thm_audoux} $G(\vec{L})\cong \pi_1(B^3 \times I \smallsetminus \text{Tube}(\vec{L}),b)$
\end{theorem}

\begin{corollary} $\widetilde{G}(\vec{L}) \cong G(\Zh(\vec{L})) \cong \pi_1(B^3 \times I \smallsetminus \text{Tube} \circ\Zh(\vec{L}),b)$.
\end{corollary}

\subsection{Artin representations} \label{sec_artin_links} The first step in defining the Artin representation is the following Chen-Milnor theorem for virtual string links. Given a virtual string link diagram $\vec{L}$, label the arcs of the $i$-th component $\vec{K}_i$ of $\vec{L}$ as follows. Let $a_{i,1}$ be the arc containing the left-hand endpoint. Successively label the arcs while traversing $\vec{K}_i$ as $a_{i,2}, \ldots,a_{i,n_i}$. The arc $a_{i,1}$ will be called the \emph{initial arc} and $a_{i,n_i}$ will be called the \emph{terminal arc}. The initial and terminal arc can be the same arc.

\begin{theorem} \label{cor_chen_milnor_string} Let $\vec{L}$ be an $m$-component virtual string link diagram, let $a_1,\ldots,a_m$ be the initial arcs of the components of $\vec{L}$, $G=G(\vec{L})$ and $F$ the free group on $a_1,\ldots,a_m$. Then:
\[
\xymatrix{
\phi^{(q)}:Q_{q}(G) \ar[r]^-{\cong}& Q_{q}(F).}
\]
\end{theorem}
\begin{proof}  Successively writing down the relations, the last relation is $a_{i,n_i}=\overline{u}_{i,n_i-1}a_{i,n_i-1} u_{i,n_i-1}$. Thus, $G(\vec{L})$ has a serial Wirtinger-type presentation such that each conjugacy classes of generators is either serial or trivially cyclic. A trivially cyclic conjugacy class occurs when the initial and terminal arcs coincide. As there are no nontrivial return relations, the nilpotent quotients have no relations other than the elements of $F_{q}$ and result follows from Theorem \ref{thm_gen_chen_milnor}.
\end{proof}

By Theorem \ref{cor_chen_milnor_string} there are isomorphisms $Q_{q}(F) \to Q_{q}(G(\vec{L}))$ defined as follows:
\[
\psi_0^{(q)},\psi_1^{(q)} : Q_{q}(F) \to Q_{q}(G(\vec{L})). 
\]
\begin{align*}
\psi_0^{(q)}(a_i) &= a_{i,1} \\
\psi_1^{(q)}(a_i) &= a_{i,n_i}
\end{align*}
In other words, $\psi_0^{(q)},\psi_1^{(q)}$ map the generators of $F$ to the initial and terminal arcs of $\vec{L}$, respectively.

\begin{definition}[Artin representation] The Artin representation of a virtual string link diagram $\vec{L}$ is the composition:
\[
\xymatrix{
A^{(q)}(\vec{L}): Q_{q+1}(F) \ar[rr]^-{\psi_0^{(q+1)}} & & Q_{q+1}(G(\vec{L})) \ar[rr]^-{\left(\psi_1^{(q+1)}\right)^{-1}} & & Q_{q+1}(F) 
}
\]
\end{definition}

\begin{theorem} \label{thm_artin_form} Suppose $\vec{L}=\vec{K}_1 \cup \cdots \cup \vec{K}_m$ is a virtual string link diagram and $\lambda_i$ is a longitude word for $\vec{K}_i$. Set $\lambda_i^{(q)}=\phi^{(q)}(\lambda_i)$  Then: 
\[
A^{(q)}(\vec{L})(a_i)=\lambda_i^{(q)} a_i \overline{\lambda}_i^{(q)}.
\]
\end{theorem}
\begin{proof} Since $\psi_0^{(q+1)}(a_i)=a_{i,1}$, the inverse of $\psi_0^{(q+1)}$ is $\phi^{(q+1)}$. Refer back to the definition of $\phi^{(q)}=\phi \circ p^{(q)}$ given in Section \ref{sec_gen_chen_milnor}. Then observe: 
\[
\left(\psi^{(q+1)}_0\right)^{-1}\circ \psi_1^{(q+1)}(a_i)=\phi \circ p^{(q+1)}(a_{i,n_i})=\phi \circ p^{(q)}(l_{i, n_i-1}^{-1} a_{i, 1} l_{i, n_i-1})=\overline{\lambda}_i^{(q)} \cdot a_i \cdot \lambda_i^{(q)}.
\]
The claim follows, since $A^{(q)}(\vec{L})=\left(\psi_1^{(q+1)}\right)^{-1}\circ \psi_0^{(q+1)}=\left(\left(\psi_0^{(q+1)}\right)^{-1}\circ \psi_1^{(q+1)}\right)^{-1}$.
\end{proof}

\begin{remark} \label{remark_exp_sum_zero} The Artin representation $A^{(r)}$ of $\vec{L}$ is unchanged if $\lambda_i^{(r)}$ is replaced with $\lambda_i^{(r)} a_i^k$ for some integer $k$. Therefore, we may assume that $\lambda_i^{(r)}$ is a longitude word (see Definition \ref{defn_word}). 
\end{remark}

Next we realize the Artin representation geometrically in terms of the Tube map. First we fix generators and a base point for $\pi_1(B^3 \times I \smallsetminus \vec{T})$, where $\vec{T}=\text{Tube}(\vec{L})$ and $\vec{L}$ is an $m$-component virtual string link. Set $b=(0,1,1/2) \in B^3 \subset \mathbb{R}^3$ and fix closed paths $\nu_i$ representing meridians of the $i$-th unknotted component of the unlink $U=\partial_0 \vec{T}$. Note that $\pi_1(B^3 \smallsetminus U,b)\cong F(m)$. An arc $a_{i,j}$ of $\vec{L}$ corresponds to a meridian $m_{\alpha_{i,j}}$ of an annulus $\alpha_{i,j}$ in the broken surface diagram. The isomorphism $G(\vec{L}) \to \pi_1(B^3 \times I \smallsetminus \text{Tube}(L),b)$ assigns $a_{i,j}$ to $m_{\alpha_{i,j}}$. Generators for the initial and terminal arcs of each component of $\vec{L}$ are fixed as follows. The meridians $m_{\alpha_{i, 1}}$ are identified with $\nu_i \times \{0\}$. Let $\sigma:[0,1] \to B^3 \times I$ be the path $\sigma(t)=(1-t)\cdot(0,1,1/2,0)+t\cdot(0,1,1/2,1)$ from $b \times 0$ to $b \times 1$. Take the meridians $m_{\alpha_{i, n_i}}$ to be $\sigma (\nu_i\times 1) \sigma^{-1}$. To realize the Artin representation, we use the following inclusion maps, where $j=0,1$.
\[
\tau_j:(B^3 \times j \smallsetminus \partial_j \vec{T},b \times j ) \to (B^3 \times I \smallsetminus \vec{T},b \times j). 
\]
Define $\eta_{\sigma}:\pi_1(B^3 \times I \smallsetminus \vec{T}, b \times 0) \to \pi_1(B^3 \times I \smallsetminus \vec{T}, b \times 1)$ to be the change of base point map determined by conjugation with $\sigma$.

\begin{theorem} \label{thm_Artin_real} The following diagram commutes, with all arrows isomorphisms.
\[
\xymatrix{
Q_{q+1}(F) \ar[r]^-{(\tau_0)_{\#}} \ar[d]_{A^{(q)}(\vec{L})} & Q_{q+1}(\pi_1(B^3 \times I \smallsetminus \text{Tube}(\vec{L}),b \times 0)) \ar[d]^{\eta_{\sigma}} \\ 
Q_{q+1}(F) & Q_{q+1}(\pi_1(B^3 \times I \smallsetminus \text{Tube}(\vec{L}),b \times 1))\ar[l]^-{(\tau_1)_{\#}^{-1}}  
}
\]
\end{theorem}

\begin{proof} It follows from the definitions of the meridians $m_{\alpha_{i,1}},m_{\alpha_{i,n_i}}$ that the diagram commutes. It must be shown that the horizontal maps are isomorphisms. Let $V,N$ be tubular neighborhoods of $\partial_j \vec{T}$, $\vec{T}$, respectively. Using the Mayer-Vietoris sequences for each of the decompositions $B^3 \times j=(B^3 \times j\smallsetminus \partial_j\vec{T}) \cup V$ and $B^3 \times I=(B^3 \times I\smallsetminus \vec{T}) \cup N$, it follows that $H_k(\tau_j)$ is an isomorphism on homology for $k=1,2$ and $j=0,1$. Lemma \ref{thm_stallings} then implies that $(\tau_j)_{\#}$ descends to an isomorphism of the nilpotent quotients. 
\end{proof}

\begin{theorem}\label{thm_artin} If $\vec{L}$, $\vec{L}^*$ are welded-concordant virtual string links, then $A^{(q)}(\vec{L})=A^{(q)}(\vec{L}^*)$.
\end{theorem}
\begin{proof} Let $\Lambda$ be a concordance between $\vec{T}=\text{Tube}(L),\vec{T}^*=\text{Tube}(L^*)$. For $i=0,1$, define inclusion maps of pointed spaces as follows:
\begin{align*}
\iota_i &: (B^3 \times I \smallsetminus \vec{T},b \times i) \to (B^3 \times I \times I \smallsetminus \Lambda, b \times i \times 0), \\
\iota_i^* &: (B^3 \times I \smallsetminus \vec{T}^* ,b \times i) \to (B^3 \times I \times I \smallsetminus \Lambda, b \times i \times 1).
\end{align*}
Using a Mayer-Vietoris sequence, as in the proof of Theorem \ref{thm_Artin_real}, it follows that $H_k(\iota_i), H_k(\iota_i^*)$ are isomorphisms on homology for $k=1,2$. Lemma \ref{thm_stallings} then implies that $(\iota_i)_{\#},(\iota_i^*)_{\#}$ induce isomorphisms on nilpotent quotients.

A welded concordance between $\vec{L},\vec{L}^*$ is a sequence of births, deaths, saddles, detour moves, Reidemeister moves, and forbidden overpass moves. This gives a handle decomposition of a concordance $\Lambda$ between $\vec{T}=\text{Tube}(\vec{L})$ and $\vec{T}^*=\text{Tube}(\vec{L}^*)$. Since these moves occur locally, they do not affect the ends of the ribbon tubes. Hence, it may be assumed that $\Lambda \cap (B^3 \times j \times t)$ is the unlink $U$ for $j=0,1$ and for all $t \in I$.  Let $M$ be the space obtained from $B^3 \times I \times I \smallsetminus \Lambda$ using the identifications: 
\begin{align*}
(B^3 \times 0 \times x) \smallsetminus (U \times 0 \times x) &\sim (B^3 \times 0 \times 0) \smallsetminus (U \times 0 \times 0) ,\, \forall x \in I,\\ 
(B^3 \times 1 \times x) \smallsetminus (U \times 1 \times x) &\sim (B^3 \times 1 \times 0) \smallsetminus (U \times 1 \times 0),\, \forall x \in I.
\end{align*}
This implies that $B^3 \times j \times 0 \smallsetminus \partial_j \vec{T} \times 0$ and $B^3 \times j \times 1 \smallsetminus \partial_j\vec{T}^*\times 1$ are identified to a single copy of $B^3 \smallsetminus U$ in $M$ for $j=0,1$. Observe also that $\pi_1(M,b \times j \times 0) \cong \pi_1(B^3 \times I \times I \smallsetminus \Lambda,b \times j\times 0)$ for $j=0,1$. For $j=0,1$, let $\eta_{\sigma \times j}:\pi_1(M, b \times 0 \times 0) \to \pi_1(M, b \times 1 \times 0)$ denote the change of base point map that conjugates by $\sigma \times j$. Now consider the diagram:
\small
\[
\xymatrix{
  & Q_{q+1}(\pi_1(B^3 \times I \smallsetminus \vec{T}, b \times 0)) \ar[r]_{\eta_{\sigma}} \ar[d]_{(\iota_0)_{\#}} & Q_{q+1}(\pi_1(B^3 \times I \smallsetminus \vec{T}, b \times 1)) \ar[d]_{(\iota_1)_{\#}}  & \\
Q_{q+1}(F) \ar@/^2pc/[ur]^-{(\tau_0)_{\#}} \ar[r] \ar@/_2pc/[dr]_-{(\tau_0)_{\#}} \ar[r] & Q_{q+1}(\pi_1(M,b \times 0 \times 0)) \ar@/^/[r]^-{\eta_{\sigma \times 0}} \ar@/_/[r]_-{\eta_{\sigma \times 1}} & Q_{q+1}(\pi_1(M,b \times 1 \times 0))   &\ar[l] Q_{q+1}(F) \ar@/_2pc/[lu]_-{(\tau_1)_{\#}} \ar@/^2pc/[ld]^-{(\tau_1)_{\#}}\\
  & Q_{q+1}(\pi_1(B^3 \times I \smallsetminus \vec{T}^*, b \times 0)) \ar[r]^{\eta_{\sigma}} \ar[u]^{(\iota_0^*)_{\#}} & Q_{q+1}(\pi_1(B^3 \times I \smallsetminus \vec{T}^*, b \times 1)) \ar[u]^{(\iota_1^*)_{\#}} &
}
\]
\normalsize
By Theorem \ref{thm_Artin_real}, the arrows along the top give the Artin representation of $\vec{L}$ and the arrows along the bottom give the Artin representation of $\vec{L}^*$. Triangles in the diagram commute as they contain only arrows that are induced by inclusions. In $M$, the paths $\sigma \times 0,\sigma \times 1$ are fixed endpoint homotopic. This implies that the squares above commute. Thus, $A^{(q)}(\vec{L})=A^{(q)}(\vec{L}^*)$.
\end{proof}

The last result in this section gives a computable concordance obstruction. It is used ahead in the proof that $\mathscr{VC}$ is not abelian.

\begin{corollary} \label{lemma_ab} Let $\vec{L}=\vec{K}_1 \cup \cdots \cup \vec{K}_m$, $\vec{L}^*=\vec{K}_1^* \cup \cdots \cup \vec{K}_m^*$ be $m$-component virtual string links and let $\lambda_i,\lambda_i^*$ be longitude words for $\vec{K}_i$, $\vec{K}_i^*$, respectively. Let $F=F(m)$. If $\vec{L}$ and $\vec{L}^*$ are concordant, then $\phi^{(q)}(\overline{\lambda}_i) \phi^{(q)}(\lambda^*_i) \in F_r$ for all $q \ge r$.
\end{corollary}
\begin{proof} Set $a=a_i$, $\lambda=\phi^{(r)}(\lambda_i)$, and $\lambda^*=\phi^{(r)}(\lambda_i^*)$.  Then Theorem \ref{thm_artin_form} and Theorem \ref{thm_artin} imply that: $\lambda a \overline{\lambda} \equiv \lambda^* a \overline{\lambda^*} \mod{F_{r+1}}$. Rearranging this gives: 
\[
a^{-1}(\overline{\lambda}\lambda^*)^{-1}a(\overline{\lambda} \lambda^*) \equiv 1 \mod F_{r+1}.
\]
Thus, $[a,\overline{\lambda}\lambda^*] \in F_{r+1}$. Since $\lambda,\lambda^*$ have exponent sum zero with respect to $a$, \cite{MKS} Corollary 5.12 (iii), implies that $\overline{\lambda}\lambda^*$ is in $F_{r}$. By Lemma \ref{thm_pq_commutes}, $\phi^{(q)}(\overline{\lambda}_i) \phi^{(q)}(\lambda^*_i) \equiv \phi^{(r)}(\overline{\lambda}_i) \phi^{(r)}(\lambda^*_i) \mod F_{r}$ for all $q \ge r$. This completes the proof.
\end{proof}

\subsection{Extended Artin representation} \label{sec_ex_artin} If $\vec{L}$, $\vec{L}^*$ are concordant $m$-component virtual string link diagrams, then $\Zh(\vec{L}),\Zh(\vec{L}^*)$ are semi-welded concordant $(m+1)$-component virtual string link diagrams. Then Theorem \ref{thm_artin} implies that $A^{(q)}(\Zh(\vec{L}))=A^{(q)}(\Zh(\vec{L}^*))$.

\begin{definition}[Extended Artin representation] The \emph{extended Artin representation} of a virtual string link $\vec{L}$ is the Artin representation of $\Zh(\vec{L})$.
\end{definition}

The extended Artin representation can be computed directly from the extension $\widetilde{G}=\widetilde{G}(\vec{L})$. By Theorem \ref{cor_chen_milnor_string}, $\phi^{(q)}:Q_{q}(\widetilde{G}) \to Q_q(F(m+1))$ is an isomorphism. If $(a_i,\widetilde{\lambda}_i)$ is a meridian-longitude word pair for a component $\vec{K}_i$ of $\vec{L}$ in $\Zh(\vec{L})$, then $A^{(q)}(\Zh(\vec{L}))(a_i)=\phi^{(q)}(\widetilde{\lambda}_i) a_i \phi^{(q)}(\widetilde{\lambda}_i)^{-1}$. Furthermore, $A^{(q)}(\Zh(\vec{L}))(v)=v$, where $v$ is the generator corresponding to the $\vec{\omega}$ component. Specializing to the case of long virtual knots, we make the following definition.

\begin{definition}[extended Artin representation of $\vec{K}$] Let $\vec{K}$ be a long virtual knot diagram and $F=F(2)$ the free group on $a,v$. For $q \ge 1$, the \emph{extended Artin representation} of $\vec{K}$ is the automorphism $A^{(q)}_{\zh}(\vec{K}):Q^{(q)}(F) \to Q^{(q)}(F)$ defined by $A^{(q)}_{\zh}(\vec{K})=A^{(q)}(\Zh(\vec{K}))$.
\end{definition}

\section{Relations $\&$ properties} \label{sec_prop_rel}
\subsection{The generalized Alexander polynomial} For an $m$-component virtual link a multi-variable Alexander polynomial can be defined in terms of its group $G=G(L)$. Recall that the abelianization of $G$, $H_1(G)=G/[G,G]$, is isomorphic to $\mathbb{Z}^m$ via the map $G \to \mathbb{Z}^m$ that sends a meridian of the $i$-th component of $L$ to a generator of the $i$-th factor of $\mathbb{Z}$. The \emph{Alexander invariant} is the abelianization $H_1([G,G])$ regarded as a module over $\mathbb{Z}[\mathbb{Z}^m]$. The group ring $\mathbb{Z}[\mathbb{Z}^{m}]$ is identified with the Laurent polynomial ring $\mathbb{Z}[t_1^{\pm 1},t_2^{\pm 1},\ldots,t_m^{\pm 1}]$. The \emph{Alexander matrix} is the Jacobian formed by taking the Fox derivative of each relation of $G$ with respect to each generator of $G$ and composing with the map sending each generator of the $i$-th component of $L$ to $t_i \in \mathbb{Z}[t_1^{\pm 1},t_2^{\pm 1},\ldots,t_m^{\pm 1}]$. For an $r \times c$ matrix $M$ with entries in $\mathbb{Z}[t_1^{\pm 1},t_2^{\pm 1},\ldots,t_m^{\pm 1}]$, the $k$-\emph{th elementary ideal} $\mathscr{E}_k$ is the ideal generated by the $(c-k) \times (c-k)$ minors of $M$. The $k$-\emph{th multi-variable Alexander polynomial} is the greatest common divisor of the $k$-th codimension minors of the Alexander matrix. When $k=1$, it is called the \emph{multi-variable Alexander polynomial}.

The \emph{generalized Alexander polynomial} was first defined by Jaeger, Kauffman, and Saleur \cite{JKS} as an invariant of knots in thickened surfaces. The original construction was motivated by statistical mechanics. These were later shown to be virtual knot invariants by Sawollek \cite{saw}. Equivalent versions can also be found in Silver-Williams \cite{silwill0,silwill1}, Kauffman-Radford \cite{kauffman_radford}, Crans-Henrich-Nelson \cite{CHN}, and Manturov \cite{manturov_GAP}. In \cite{acpaper}, Boden et al. defined a polynomial $\overline{H}_K(t,v)$ from the elementary ideal theory of the reduced virtual knot group, here denoted as $\widetilde{G}(K)$. Furthermore, they proved (\cite{acpaper}, Theorem 4.1) that the Sawollek polynomial is equivalent to $\overline{H}_K(t,v)$ by a change of variables. In \cite{bbc}, Section 4.2, it was shown that the Alexander invariants of $\widetilde{G}(K)$ are isomorphic to those of $G(\Zh(K))$. In particular, this implies that the multi-variable Alexander polynomial of the two-component semi-welded link $\Zh(K)$ and the generalized Alexander polynomial of $K$ may be obtained from one another by a change of variables.

Here we will show that the vanishing of the $\overline{\zh}$-invariants of $K$ implies the vanishing of the generalized Alexander polynomial of $K$. The proof relies on the following lemma.

\begin{lemma}[\cite{bbc}, Proposition 3.7]\label{lemma_bbc_free_nil_quo} Let $L$ be an $m$-component virtual link (or a link in a thickened surface) and $F$ the free group on $m$ letters. Then $Q_{q}(G(L)) \cong Q_{q}(F)$ for all $q \ge 2$ if and only if all the longitudes of $L$ are in the nilpotent residual $G(L)_{\omega}$.
\end{lemma}

\begin{theorem}\label{thm_GAP_vanishes} Let $K$ be a virtual knot. If $\overline{\!\zh}_J(K)=0$ for all sequences $J$, then the generalized Alexander polynomial of $K$ is trivial. 
\end{theorem}
\begin{proof} In \cite{bbc}, Theorem 3.8, it was shown that if the longitudes of a virtual link (or a link in a thickened surface) are in $\bigcap_{i=1}^{\infty} G_{i}G''$, where $G=G(L)$ and $G''=[G_2,G_2]$, then the Alexander ideal $\mathscr{E}_1$ is trivial (see also Hillman \cite{hill}). Our theorem is proved by applying this result to $L=\Zh(K)$ and $G=G(\Zh(K))$. Let $\widetilde{\lambda}$ be an extended longitude word of $K$ and $\widetilde{\lambda}^{(q)}=\phi^{(q)}(\widetilde{\lambda})$. Since $\overline{\zh}_J(K)=0$ for all $J$, $\epsilon_J(\widetilde{\lambda}^{(q)})=0$ for all $|J|<q$. This implies that $\widetilde{\lambda}^{(q)} \in F_{q}$. Then by Theorem \ref{cor_chen_milnor_links}, $Q_{q}(G) \cong Q_{q}(F)$, where $F$ is the free group on two letters $a,v$. Lemma \ref{lemma_bbc_free_nil_quo} then implies that the longitudes of $\Zh(L)$ lie in $G_{\omega}$. It follows that the longitudes lie in $\bigcap_{i=1}^{\infty} G_{i}G''$. Since $G_{\omega}<\bigcap_{i=1}^{\infty} G_{i}G''$, the Alexander ideal $\mathscr{E}_1$ is trivial and hence the generalized Alexander polynomial is trivial.
\end{proof}

In \cite{bcg1}, it was proved that the odd writhe \cite{kauffman_odd_writhe}, Henrich-Turaev polynomial \cite{henrich,turaev_cobordism}, and writhe (or affine index) polynomial \cite{cheng_gao,kauffman_affine} are concordance invariants of virtual knots. The following corollary implies that $\overline{\zh}$-invariants are non-vanishing whenever any these invariants is non-vanishing. 

\begin{corollary} \label{cor_index_trivial} Let $K$ be a virtual knot diagram. If $\overline{\!\zh}_J(K)=0$ for all sequences $J$, then the odd writhe, Henrich-Turaev polynomial, and writhe (or affine index) polynomial are all trivial.
\end{corollary}
\begin{proof} By Mellor \cite{mellor}, the generalized Alexander polynomial determines the odd writhe, writhe polynomial, and Henrich-Turaev polynomial. The result then follows from Theorem \ref{thm_GAP_vanishes}.
\end{proof}

\subsection{Almost classical knots} A homologically trivial knot $K \subset \Sigma \times I$ bounds a Seifert surface in $\Sigma \times I$. If a virtual knot has a homologically trivial representative in some thickened surface, then $K$ is said to be \emph{almost classical (AC)}. The original definition, due to Silver-Williams \cite{silwill}, is that a virtual knot is almost classical if it has a diagram with an Alexander numbering. The two definitions are equivalent (Boden et. al \cite{acpaper}). Of the 92800 virtual knots up to six classical crossings, 77 are almost classical and 10 are classical (Boden et al. \cite{acpaper}). However, there are many virtual knots that are not almost classical but are concordant to an AC knot. Up to six classical crossings, there are exactly 19 slice AC knots and at least 1281 non-AC slice knots \cite{bbc,bcg1,bcg2}. Furthermore, the Tristam-Levine signature functions were generalized to almost classical knots in \cite{bcg2}. From these one may obtain lower bounds on the topological slice genus of homologically trivial knots in thickened surfaces. Thus, it is of interest to find obstructions to a virtual knot being concordant to an AC knot. 

In \cite{bbc}, Boden and the author showed that every virtual knot concordant to an AC knot has vanishing generalized Alexander polynomial. Here we use a similar method to show that the $\overline{\zh}$-invariants vanish on any virtual knot concordant to an AC knot. Recall that a link $\mathscr{L} \subset \Sigma \times I$ is said to be a \emph{boundary link} if the components of $\mathscr{L}$ bound pairwise disjoint Seifert surfaces. 

\begin{theorem} If $K$ is concordant to an AC knot, then all $\overline{\zh}$-invariants of $K$ vanish.
\end{theorem}
\begin{proof} Suppose that $K$ is concordant to an almost classical knot $C$. By \cite{bbc}, Theorem 4.5, $\Zh(C)$ is a split two component semi-welded link. Since the $C$ is AC and the $\omega$ component of $\Zh(C)$ is the unknot, it follows that $\Zh(C)$ can be represented in some thickened surface $\Sigma \times I$ as a boundary link. By \cite{bbc}, Theorem 3.5, the nilpotent quotients of the group of a boundary link are isomorphic to those of a free group. Since $\Zh(K)$ and $\Zh(C)$ are concordant, Lemma \ref{thm_bbc_nil_quo} implies that: 
\[
Q_{q}(G(\Zh(K))) \cong Q_{q}(G(\Zh(C))) \cong Q_{q}(F),
\]
where $F=F(2)$ is the free group on two letters. Then Lemma \ref{lemma_bbc_free_nil_quo} implies that the longitudes of $\Zh(K)$ are in $G(\Zh(K))_{\omega}$. For any extended longitude word $\widetilde{\lambda}$ of $K$, $\phi^{(q)}(\widetilde{\lambda}) \in F_{q}$ for all $q\ge 2$. Hence, $\epsilon_J(\widetilde{\lambda}^{(q)})=0$ for any $q$ and sequence $J$ with $|J|<q$. 
\end{proof}

\subsection{Shuffle and cycle relations} \label{sec_shuffles} Milnor \cite{milnor} noted redundancies in the $\bar{\mu}$-invariants for links in $S^3$. Two relations between them are: (1) the shuffle relation, and (2) the cycle relation. Here we discuss the extent to which these hold true for virtual links. Let $J_1=j_{1,1}j_{1,2}\cdots j_{1,r}$ and $J_2=j_{2,1}j_{2,2}\cdots j_{2,s}$ be sequences. A sequence $J$ is said to be \emph{shuffle} of $J_1$ and $J_2$ if:
\begin{enumerate}
\item $J_1$ and $J_2$ are embedded in $J$ as subsequences, and
\item $J$ is the union $J_1$ and $J_2$.
\end{enumerate}
The embeddings of $J_1$ and $J_2$ into $J$ are part of the data of a shuffle, so that different embeddings into the same sequence $J$ represent different shuffles. The set of shuffles of $J_1,J_2$ will be denoted $S(P_1,P_2)$. A shuffle $J$ of $J_1$ and $J_2$ is \emph{proper} if the subsequences $J_1$ and $J_2$ of $J$ are disjoint in $J$. The set of proper shuffles of $J_1,J_2$ will be denoted $PS(J_1,J_2)$. The definition of shuffle and proper shuffle given here coincide with those of Milnor \cite{milnor}(see pages 41-42). The following proposition is merely a restatement of Milnor \cite{milnor}, Theorem 6, in terms of virtual links.

\begin{proposition} Let $L$ be an $m$-component virtual link, $J_1,J_2$ be sequences, and $1 \le k \le m$. Then:
\[
\sum_{J\in PS(J_1,J_2)} \bar{\mu}_{J|k}(L) \equiv 0 \pmod{\gcd\{\Delta_{J|k}:J\in PS(J_1,J_2)\}}.
\]
\end{proposition}
\begin{proof} The proof follows exactly as in the case of classical links. We briefly outline the argument. Choose a longitude word $\lambda_k^{(q)}$ for $q$ sufficiently large (see Section \ref{sec_mu}). By Chen-Fox-Lyndon \cite{chen_fox_lyndon}, Lemma 3.3, it follows that the sum of the proper shuffles of $J_1$ and $J_2$ of the corresponding Magnus expansion coefficients is $\varepsilon_{J_1}\left(\lambda_k^{(q)}\right) \cdot \varepsilon_{J_2}\left(\lambda_k^{(q)}\right)$. As this vanishes modulo the indeterminacy, the result follows.  
\end{proof}

As an application of the shuffle relation, we give spanning sets of $\overline{\zh}$-invariants for small $|J|$. By a \emph{spanning set} for the invariants of order $n$, we mean a set $\mathscr{S}_n=\{\overline{\zh}_{J_1},\ldots,\overline{\zh}_{J_z}\}$ with $|J_1|=\cdots=|J_z|=n$ such that every $\overline{\zh}_J$ can be written as an integral linear combination of elements of $\mathscr{S}_n$.  

\begin{proposition} \label{thm_spanning} Let $n \in \mathbb{N}$, $2 \le n \le 8$, and let $K$ be a virtual knot. Suppose that $\overline{\zh}_J(K)=0$ for all $J$ with $|J|<n$. Then there is a spanning set $\mathscr{S}_n$ of $\overline{\zh}$-invariants for $K$ having at most $E_n$ elements, where $E_n$ is given below. Spanning sets $\mathscr{S}_n$ with $|\mathscr{S}_n|=E_n$ are in Table \ref{table_spanning}.
\[
\begin{tabular}{|c|c|c|c|c|c|c|c|c|c|} \hline
n     & 2 & 3 & 4 & 5 & 6 & 7 & 8  \\ \hline
$E_n$ & 1 & 2 & 3 & 6 & 9 & 18 & 30 \\\hline
\end{tabular}
\] 
\end{proposition}
\begin{proof} Apply the shuffle relation to $\bar{\mu}_{J|1}(\Zh(K))$. Since all invariants of order less than $|J|$ are vanishing, each shuffle relation gives an equation over the integers. If $J_1,J_2$ are either both all ones or all twos, then the shuffle relation implies that $\bar{\mu}_{11\cdots 11}(K)=\bar{\mu}_{22\cdots 22}(K)=0$. Then for $n=2$, the only relation is that $\overline{\zh}_{12} (K)+\overline{\zh}_{21}(K)=0$.  This establishes the claim for $n=2$. For $n=3$, the shuffle relations are:
\begin{align*}
\overline{\zh}_{112}(K)+\overline{\zh}_{121}(K)+\overline{\zh}_{211}(K) &=0, \\
\overline{\zh}_{121}(K)+\overline{\zh}_{112}(K)+\overline{\zh}_{112}(K) &=0, \\
\overline{\zh}_{122}(K)+\overline{\zh}_{122}(K)+\overline{\zh}_{212}(K) &=0, \\
\overline{\zh}_{211}(K)+\overline{\zh}_{211}(K)+\overline{\zh}_{121}(K) &=0, \\
\overline{\zh}_{212}(K)+\overline{\zh}_{221}(K)+\overline{\zh}_{221}(K) &=0, \\
\overline{\zh}_{221}(K)+\overline{\zh}_{212}(K)+\overline{\zh}_{122}(K) &=0. 
\end{align*}
Solving these equations over the integers implies that every invariant $\overline{\zh}_J(K)$ is in the integral span of $\mathscr{S}_2=\{\overline{\zh}_{211}, \overline{\zh}_{221}\}$. For $4 \le n \le 8$, \emph{Mathematica} was used to write out all of the shuffle relations and reduce (over $\mathbb{Z}$) the resulting system of equations. 
\end{proof}

\begin{table}
\[\small
\begin{array}{|c|cccccc|} \hline
\mathscr{S}_2 & 21  &        &        &       &       & \\ \hline
\mathscr{S}_3 & 211 &  221 &        &       &       & \\ \hline
\mathscr{S}_4 & 2111  & 2211  & 2221  &       &       &  \\ \hline
\mathscr{S}_5 & 21111 & 21211 & 22111 & 22121 & 22211 & 22221 \\ \hline
\mathscr{S}_6 & 211111 & 212111 & 212211 & 221111 & 221211 & 222111 \\
& 222121 & 222211 & 222221 &   &   &   \\ \hline
\mathscr{S}_7 & 2111111 & 2112111 & 2121111 & 2121211 & 2122111 & 2122211 \\
& 2211111 & 2211211 & 2212111 & 2212121 & 2212211 & 2221111 \\
& 2221211 & 2221221 & 2222111 & 2222121 & 2222211 & 2222221 \\ \hline
\mathscr{S}_8 & 21111111 & 21121111 & 21122111 & 21211111 & 21211211 & 21212111 \\
& 21212211 & 21221111 & 21221211 & 21222111 & 21222211 & 22111111 \\
& 22112111 & 22121111 & 22121211 & 22122111 & 22122121 & 22122211 \\
& 22211111 & 22211211 & 22212111 & 22212121 & 22212211 & 22221111 \\
& 22221211 & 22221221 & 22222111 & 22222121 & 22222211 & 22222221 \\ \hline
\end{array}
\]
\caption{Spanning sets for $\overline{\zh}$-invariants of  order up to 8.} \label{table_spanning}
\end{table}

For a link $\mathscr{L}\subset S^3$, the cycle relation says that $\bar{\mu}_{j_1j_2\cdots j_s}(\mathscr{L})=\bar{\mu}_{j_sj_1j_2\cdots j_{s-1}}(\mathscr{L})$. For a $2$-component link $\mathscr{L}=\mathscr{J} \cup \mathscr{K}$, this recovers the familiar symmetric property of the linking number: $\lk(\mathscr{J},\mathscr{K})=\bar{\mu}_{12}(\mathscr{L})=\bar{\mu}_{21}(\mathscr{L})=\lk(\mathscr{K},\mathscr{J})$. Next we will show that the cycle relation does not hold in general for virtual links by relating the $\bar{\mu}$-invariants to the virtual linking number (see also Dye-Kauffman \cite{dye_kauffman}, Theorem 6.1). Indeed, as discussed in Section \ref{sec_vlk}, the virtual linking number is not symmetric.

\begin{proposition} \label{prop_linking_number} Let $L=J \cup K $ be a $2$-component virtual link diagram. Then:
\[
\bar{\mu}_{12}(J\sqcup K)=\vlk(J,K).
\]
In particular, for any virtual knot $K$, $\overline{\!\zh}_2(K)=0$.
\end{proposition}
\begin{proof} Set $G=G(L)$ and let $\lambda$ be a longitude word of $K$. In $\mathbb{Z}[[a_1,a_2]]$, $\bar{\mu}_{12}$ is the coefficient of $a_1$ in the Magnus expansion of $\phi^{(2)}(\lambda)$. Contributions to this coefficient are identified with the exponents of $a_{1,j}$ in the arcs of $J$. Each overcrossing of $J$ with $K$ introduces a letter with exponent $\pm 1$, according to whether the crossing is signed positively or negatively, respectively. The contribution to the coefficient of $a_1$ is $+1$ for each positive crossing of $J$ over $K$ and $-1$ for each negative crossing of $J$ over $K$. Thus, $\bar{\mu}_{12}(J\sqcup K)=\vlk(J,K)$.

For the second claim, recall that $\overline{\zh}_2(K)=\bar{\mu}_{21}(\Zh(K))$. By the first claim, $\bar{\mu}_{21}(\Zh(K))=\bar{\mu}_{21}(K \sqcup \omega)=\vlk(\omega,K)$. Since $\omega$ has two overcrossings with $K$ for every crossing of $K$, one of which is signed $+1$ and the other of which is signed $-1$ (see Figure \ref{fig_zh_defn}), it follows that $\vlk(\omega,K)=0$.
\end{proof}

\section{Calculations $\&$ applications} \label{sec_calc_app}
Extended Artin representations and $\overline{\zh}$-invariants can be calculated from a Gauss diagram. Section 7.1 gives a brief review of Gauss codes and Gauss diagrams of virtual knots. Section \ref{sec_manual} illustrates an example manual calculation of the first non-vanishing $\overline{\zh}$-invariants for the virtual knot $3.5$. In the remaining sections, these calculations are done with the aid of a computer. Section \ref{sec_slice_obstructions} applies the $\overline{\zh}$-invariants to the virtual knots of unknown slice status. In Section \ref{sec_vc}, the extended Artin representation is used to prove that $\mathscr{VC}$ is not abelian.  Section \ref{sec_beyond} discusses the $\overline{\zh}$-invariants beyond the first non-vanishing order. Section \ref{sec_movies} gives slice movies for 22 virtual knots of previously unknown slice status.

\subsection{Gauss diagrams} A virtual knot diagram $K$ is an immersion $S^1 \to \mathbb{R}^2$ that is decorated with classical and virtual crossings. A \emph{Gauss diagram} of $K$ can be obtained by connecting the pre-images of the classical crossings by an arrow in $S^1$ that points from the overcrossing arc to the undercrossing arc. The sign of each crossing is marked near the corresponding arrow: $\oplus$ for right-handed crossings and $\ominus$ for left-handed ones. An example for $K=3.5$ is given in Figure \ref{fig_gauss}. To obtain the \emph{Gauss code} of a Gauss diagram, mark a base point on the Gauss diagram and label the arrows successively while circumnavigating $S^1$. At each arrow endpoint, write an ``O'' for each arrow tail, a ``U'' for each arrow head, the arrow label, and the sign $(\pm)$ of the crossing. Repeat this process until all of the arrow endpoints are passed. See Figure \ref{fig_gauss}, center. 

\begin{figure}[htb]
\centerline{
\begin{tabular}{ccc} \begin{tabular}{c}
\def\svgwidth{1in}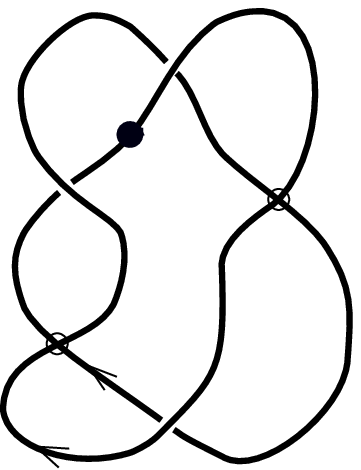 \end{tabular} & \begin{tabular}{c} $\text{O1-O2-O3-U1-U2-U3-}$\end{tabular} & \begin{tabular}{c}
\def\svgwidth{1.4in}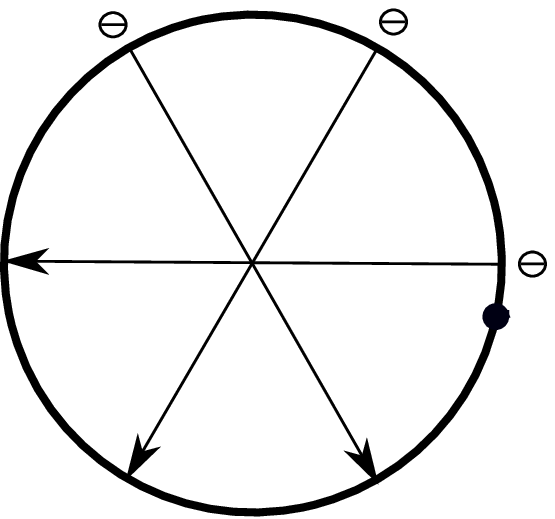 \end{tabular}    \end{tabular}
}
\caption{The Gauss code (center) and Gauss diagram (right) of $K=3.5$ (left).} \label{fig_gauss}
\end{figure}

Each of the Reidemeister moves may be translated into Gauss diagrams. Figure \ref{fig_gd_moves} shows one type of Gauss diagram move for each of the three Reidemeister moves. For a full list, the reader is referred to Polyak \cite{polyak_minimal}. For example, there are eight different types $\Omega 3$ moves. Using the notation of \cite{polyak_minimal}, they are: $\Omega 3a$, $\Omega 3b$, $\Omega 3c$, $\Omega 3d$, $\Omega3e$, $\Omega3f$, $\Omega3g$, $\Omega3h$. These will be used in the slice movies in Section \ref{sec_movies}.

\begin{figure}[htb]
\centerline{
\begin{tabular}{c}
\def\svgwidth{6in}
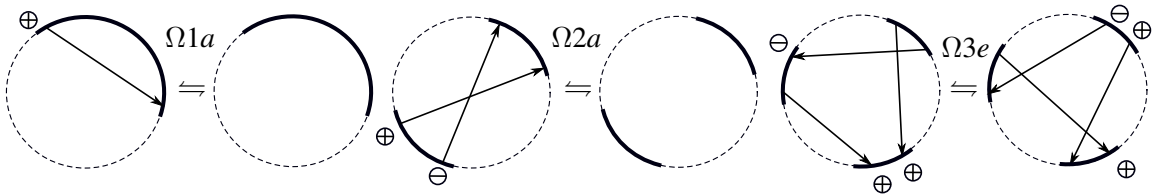
\end{tabular}
} 
\caption{Some Reidemeister moves in Gauss diagram form.}\label{fig_gd_moves}
\end{figure}

\subsection{Manual calculation} \label{sec_manual} Let $K=3.5$. Here we demonstrate a manual calculation of the first non-vanishing $\overline{\zh}$-invariants of $K$. A Gauss diagram of $K=3.5$ is drawn in Figure \ref{fig_three5_gauss}. There are three main steps in the calculation:
\begin{enumerate}
\item Computing $\widetilde{\lambda}^{(q)}=\phi^{(q)}(\lambda)$ for $\lambda$ an extended longitude word of $K$.
\item Writing $\widetilde{\lambda}^{(q)}$ as a product of commutators.
\item Computing the first non-zero coefficients of the Magnus expansion.
\end{enumerate}

\begin{figure}[htb]
\centerline{
\begin{tabular}{c}
\def\svgwidth{1.4in}
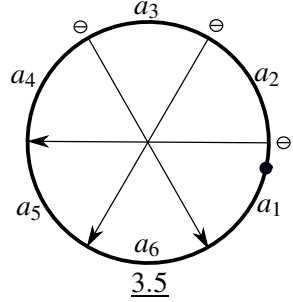\\ \\
\end{tabular}
} 
\caption{The Gauss diagram and base point used in the manual calculation.}\label{fig_three5_gauss}
\end{figure}

\noindent\underline{\emph{Step (1)}}: The following stratagem can be used to quickly write out a presentation of $\widetilde{G}=\widetilde{G}(K)$ from a Gauss diagram. Each crossing of $K$ will be passed twice while traversing the Gauss diagram from the base point. When passing under a crossing, write down the relation of $\widetilde{G}$ of the form $c=bvav^{-1}b^{-1}$ or $d=a^{-1}v^{-1}bva$, according to whether the sign is $-$ or $+$, respectively. When passing over a crossing, write down a relation of the form $d=v^{-1}bv$ or $c=vav^{-1}$, according to whether the crossing is signed $-$ or $+$, respectively. The relations now appear in the form of a serial Wirtinger-type presentation. The extended group $\widetilde{G}=\widetilde{G}(K)$ is then given by:
\begin{align*}
\widetilde{G}=\langle & a_1,a_2,a_3,a_4,a_5,a_6,v | \\  & a_2=\overline{v}a_1v,a_3=\overline{v}a_2v,a_4=\overline{v}a_3v,a_5=a_1va_4\wwbar{v}\overline{a}_1,a_6=a_2va_5\wwbar{v}\overline{a}_2,a_1=a_3va_6\wwbar{v}\overline{a}_3 \rangle.
\end{align*}
An extended parallel word $\widetilde{l} \in \widetilde{G}$ of $K$ is calculated from this presentation: 
\[
\widetilde{l}=v^2\wwbar{a}_1\wwbar{v}\overline{a}_2\wwbar{v}\overline{a}_3.
\] 
Next, the word $\widetilde{l}^{\,(4)}=\phi^{(4)}(\widetilde{l})$ in $F/F_{4}$ is determined inductively as in Theorem \ref{thm_gen_chen_milnor}. After calculating $\phi^{(4)}(\widetilde{l})$, we obtain an extended longitude $\widetilde{\lambda}^{(4)}$ by multiplying $\phi^{(4)}(\widetilde{l})$ by $a^k$, with $k$ chosen so that the exponent sum of $a$ is zero. This results in the following:
\begin{align*}
\phi^{(2)}(\widetilde{l})=\phi^{(3)}(\widetilde{l})=\phi^{(4)}(\widetilde{l}) &= v^2\wwbar{a}\overline{v}^2\wwbar{a}\overline{v}^2\wwbar{a}v^2\\
\widetilde{\lambda}^{(4)}&=v^2\wwbar{a}\overline{v}^2\wwbar{a}\overline{v}^2\wwbar{a}v^2a^3
\end{align*}

\noindent \underline{\emph{Step (2)}}: By Hall \cite{hall}, Theorem 11.2.4, every element $f$ in a free group on $m$ letters $F(m)$ can be written uniquely as:
\[
f\equiv c_1^{e_1} c_2^{e_2} \cdots c_t^{e_t} \pmod{F_{q}},
\]
where $c_1,\ldots,c_t$ are the finite collection of basic commutators with weight at most $q-1$, and $e_i \in \mathbb{Z}$. A basic commutator of weight $i$ is in $F_{i}$ but not in $F_{i+1}$, and the basic commutators up to weight $i$ can be computed recursively. There is an algorithm, called the \emph{collection process}, that writes a word $f$ in normal form as above. Applying the algorithm to $f=\widetilde{\lambda}^{(4)}$, we obtain:
\[
\widetilde{\lambda}^{(4)} \equiv [v,a,a]^4 [v,a,v]^4 \mod F_{4}
\]
Details of the collection process for $\widetilde{\lambda}^{(4)}$ above are available on the author's website for the interested reader.
\newline

\noindent \underline{\emph{Step (3)}}: The $\overline{\zh}$-invariants can now be calculated using the properties of $\epsilon_J$. The relevant properties for our purposes are as follows:

\begin{lemma}\label{lemma_fenn} Let $f \in F_{q}$, $g \in F_r$, and $J$ a sequence.
\begin{enumerate}
\item If $|J|<q$, then $\epsilon_J(f)=0$.
\item If $|J|\le\min\{q,r\}$, then $\epsilon_J(fg)=\epsilon_J(f)+\epsilon_J(g)$.
\item Suppose $J=J_1|J_2=I_1|I_2$ where $|J_1|=|I_2|=q$ and $|J_2|=|I_1|=r$, then:
\[
\epsilon_J([f,g])=\epsilon_{J_1}(f)\epsilon_{J_2}(g)-\epsilon_{I_1}(g)\epsilon_{I_2}(f).
\]
\end{enumerate}
\end{lemma}
\begin{proof} See for example, Fenn \cite{fenn}, Lemma 4.4.1.
\end{proof}

We now proceed with our example calculation. For an arbitrary sequence $J$ with $|J|=3$, partition $J$ in two ways as $J=J_1|J_2=I_1|I_2$, where $|I_1|=|J_2|=1$ and $|I_2|=|J_1|=2$. Then:
\begin{align}
\overline{\zh}_J(K) &= 4 \cdot \epsilon_{J_1}([v,a])(\epsilon_{J_2}(a)+\epsilon_{J_2}(v))-4\cdot\epsilon_{I_2}([v,a])(\epsilon_{I_1}(a)+\epsilon_{I_1}(v)) \label{eqn_three5} 
\end{align}
Thus, to calculate all of the values $\overline{\zh}_J(K)$ with $|J|=3$, we need only need know the Magnus expansions of $a$, $v$, and $[v,a]$ up to order two. The Magnus expansion of the commutator is:
\begin{align*}
[v,a]= v^{-1}a^{-1}va &\rightarrow (1-v+v^2+\cdots)(1-a+a^2+\cdots)(1+v)(1+a) \\
      &= 1+va-av+\text{higher order terms}
\end{align*}
Table \ref{table_3_5_zh} gives all the contributions of terms in Equation (\ref{eqn_three5}) for $J$ with $|J|=3$. The last column gives $\overline{\zh}_J(K)$ for each such $J$. Note that this verifies the result of Theorem \ref{thm_spanning}, that all invariants of order $3$ can be written as a linear combination of $\overline{\zh}_{211}$ and $\overline{\zh}_{221}$.

\begin{table}
\[
\begin{tabular}{|c||ccccccc|} \hline
$J$ & $\epsilon_{J_1}([v,a])$ & $\epsilon_{J_2}(a)$ & $\epsilon_{J_2}(v)$ & $\epsilon_{I_2}([v,a])$ & $\epsilon_{I_1}(a)$ & $\epsilon_{I_1}(v)$ & $\overline{\zh}_J$  \\ \hline \hline
111 & 0  & 0 & 0 & 0  & 0 & 0 & 0 \\ 
112 & 0  & 0 & 1 & -1 & 1 & 0 & 4 \\
121 & -1 & 1 & 0 & 1  & 1 & 0 & -8\\
211 & 1  & 1 & 0 & 0  & 0 & 1 & 4 \\ 
122 & -1 & 0 & 1 & 0  & 1 & 0 & -4\\ 
212 & 1  & 0 & 1 & -1 & 0 & 1 & 8\\ 
221 & 0  & 1 & 0 & 1  & 0 & 1 & -4 \\ 
222 & 0  & 0 & 0 & 0  & 0 & 0 & 0 \\ \hline
\end{tabular}
\]
\caption{Final calculation of the first non-vanishing $\overline{\zh}$-invariants of $K=3.5$.} \label{table_3_5_zh}
\end{table}

\subsection{Slice obstructions for virtual knots} \label{sec_slice_obstructions} Table \ref{table_status} shows the cumulative data for the slice status of the virtual knots up to six classical crossings. The table is compiled from \cite{bbc,bcg2,bcg1,karimi,rush}. The remaining 38 virtual knots of unknown slice status are given in Table \ref{table_status_unknown}. Each has trivial graded genus, generalized Alexander polynomial, and Rasmussen invariant. The Rasmussen invariants for these 38 virtual knots were calculated by H. Karimi. Here we will show that 11 of the virtual knots in Table \ref{table_status_unknown} are not slice. A twelfth virtual knot will be shown to be not slice in Section \ref{sec_vc}. Section \ref{sec_movies} gives slice movies for 22 of the virtual knots from Table \ref{table_status_unknown}. 
\begin{table}[htb]
\[
\begin{tabular}{|c||c|c|c|c|}
\hline
 Crossing & Virtual & Not & Slice  & Status     \\
number & knots & slice &  knots & unknown  \\
\hline \hline
2 & 1 & 1 &  0 & 0\\ 
3 & 7 & 7 &  0 & 0\\ \
4 & 108 & 95 & 13 & 0\\ 
5 & 2448 & 2401 & 45 & 2\\
6 & 90235 & 88958 & 1241 & 36\\ \hline
\end{tabular}
\]
\caption{Summary of results from \cite{bbc,bcg2,bcg1,karimi,rush}.} \label{table_status}
\end{table}
\begin{table}[htb]
\[
\begin{tabular}{|cccccccc|} \hline
5.1216 & 5.1963 & 6.5588 & 6.5958 & 6.6589 & 6.7070 & 6.7388 & 6.8451 \\ 
6.14778 & 6.14781 &  6.15200 & 6.15952 & 6.31455 & 6.33334 & 6.37879 & 6.38158 \\ 
6.38183 & 6.43763 & 6.46936 & 6.46937 & 6.47024 & 6.47172 & 6.47512 & 6.49338 \\
6.52373 & 6.62002 & 6.69085 & 6.70767 & 6.71306 & 6.71848 & 6.72353 & 6.72431 \\ 
6.76251 & 6.76488 & 6.77331 & 6.77735 & 6.86951 & 6.89218 & &  \\ \hline
\end{tabular}
\]
\caption{The 38 virtual knots of unknown slice status in Table \ref{table_status}.} \label{table_status_unknown}
\end{table}

Each of the three steps outlined in Section \ref{sec_manual} can be performed on a computer. \emph{Mathematica} \cite{wolfram} programs were written by the author to accomplish steps (1) and (3). \emph{GAP} \cite{GAP4} was used to perform the commutator collection process used in step (2). More specifically, the word $\widetilde{\lambda}^{(q)}$ is first computed in \emph{Mathematica}. This is then imported into \emph{GAP}. The \emph{GAP} package \emph{ANU NQ}, written by W. Nickel, is then used to compute the nilpotent quotients $F/F_{q}$. Mapping $\widetilde{\lambda}^{(q)}$ to $F/F_{q}$ writes the word in normal form as a product of commutators having weight at most $q-1$. It is important to note that \emph{ANU NQ} uses a different commutator basis than the Hall basis discussed in Section \ref{sec_manual}. Lemma \ref{lemma_fenn} implies that it is sufficient for our purposes to write $\widetilde{\lambda}^{(q)}$ as any product of commutators. Details on the commutator basis used by \emph{ANU NQ} can be found in Nickel \cite{nickel}. Lastly, the normal form is imported back into \emph{Mathematica} and the coefficients $\epsilon_J(\widetilde{\lambda}^{(q)})$ of the Magnus expansion are computed. The \emph{Mathematica} and \emph{GAP} programs are available on the author's website.

Consider the 11 virtual knots in Table \ref{table_the_eleven}. It is useful to observe that the length of $\widetilde{\lambda}^{(q)}$ can sometimes be reduced by shifting the base point of the Gauss code to maximize the overcrossings appearing near the beginning of the code. Table \ref{table_the_eleven} gives this shifted code. The rightward shift needed to obtain this Gauss code from the Gauss code in Green's table is indicated in parentheses. Each of these virtual knots has its first non-vanishing $\overline{\zh}$-invariant at order five. The commutator basis used by \emph{ANU NQ} is:
\begin{align*}
g_1 &=a           &g_8 &= [v,a,v,v] \\
g_2 &= v          &g_9 &= [v,a,a,a,a] \\
g_3 &= [v,a]     &g_{10} &= [v,a,a,a,v] \\ 
g_4 &= [v,a,a]   &g_{11} &= [v,a,v,a,a] \\
g_5 &= [v,a,v]   &g_{12} &= [v,a,v,a,v] \\
g_6 &= [v,a,a,a] &g_{13} &= [v,a,v,v,a] \\
g_7 &= [v,a,v,a] &g_{14} &= [v,a,v,v,v]
\end{align*}
Table \ref{table_the_eleven} shows the normal form of $\widetilde{\lambda}^{(6)}\!\!\mod F_{6}$.  The $\overline{\zh}$-invariants in the spanning set $\mathscr{S}_5$ (see Table \ref{table_spanning}) are given in Table \ref{table_zh_the_eleven}. As there is at least one non-vanishing $\overline{\zh}$-invariant in each case, none of these virtual knots are slice.

\begin{table}[htb]
\small
\renewcommand{\arraystretch}{1.3}
\[
\begin{array}{|c||c|c|c|}\hline
  K      & \text{Gauss code (right shift from \cite{green})} &  \text{length of } \widetilde{\lambda}^{(6)} &  \widetilde{\lambda}^{(6)} \!\!\! \mod F_{6} \\ \hline \hline
6.6589 & \text{O1-O2-O3-O4+U3-O5+U4+O6+U2-U1-U5+U6+}  (2) & 22   & g_{10} \overline{g}_{11} g_{12} \overline{g}_{13} \\ \hline
6.7070  & \text{O1+O2-O3-U2-O4-U3-U5+U4-O6+U1+O5+U6+} (1) & 792  & \overline{g}_{10}g_{11} \\ \hline
6.15200 & \text{O1+O2+O3-O4+U3-O5-O6-U5-U1+U6-U4+U2+} (2) & 20   & \overline{g}_{10}g_{11} g_{12} \overline{g}_{13} \\ \hline
6.15952 & \text{O1-O2-U1-O3-O4+U3-U5+O6+O5+U6+U2-U4+} (0) & 1120 & g_{10} \overline{g}_{11} \overline{g}_{12} g_{13} \\ \hline
6.43763 & \text{O1+U2+O3+O4-O2+U4-O5-U6-U1+U5-O6-U3+} (3) & 586  & \overline{g}_{10} g_{11} \\ \hline
6.47172 & \text{O1+O2-O3-O4-U3-O5+U2-U6+U1+U4-O6+U5+} (2) & 634 & g_{10}^2\wwbar{g}_{11}^{2} \\ \hline
6.47512 & \text{O1+O2+O3-O4+U3-O5-U2+U6-U1+U5-O6-U4+} (2) & 934 & \overline{g}_{10}^{2} g_{11}^2 \\ \hline
6.71848 & \text{O1-O2+O3-U1-O4-U2+O5+U6+U4-U5+O6+U3-} (0) & 586 & g_{10} \overline{g}_{11} \\ \hline
6.72431 & \text{O1-O2+O3-U1-O4-U3-O5+U2+O6+U5+U4-U6+} (0) & 14 & g_{10} \overline{g}_{11} \\ \hline
6.76251 & \text{O1-O2+O3-U1-O4-U5+O6+U2+O5+U6+U4-U3-} (0) & 498 & g_{10} \overline{g}_{11} \\ \hline
6.89218 & \text{O1-O2+U3+O4+U2+O3+U5-O6-U1-O5-U6-U4+} (0) & 2278 & \overline{g}_{10}g_{11} \\ \hline
\end{array}
\]
\caption{Summary of the first two steps in calculating the $\overline{\zh}$-invariants.} \label{table_the_eleven}
\end{table}

\begin{table}[htb]
\[
\begin{tabular}{|c||ccccccccccc|} \hline
 & \multicolumn{11}{|c|}{$\overline{\zh}_J(K)$}  \\
$J$ & \rotatebox{75}{6.6589} & \rotatebox{75}{6.7070} & \rotatebox{75}{6.15200} &\rotatebox{75}{6.15952} &\rotatebox{75}{6.43763} & \rotatebox{75}{6.47172} & \rotatebox{75}{6.47512} & \rotatebox{75}{6.71848} & \rotatebox{75}{6.72431} & \rotatebox{75}{6.76251} & \rotatebox{75}{6.89218} \\ \hline \hline
21111 & 0 & 0 & 0 & 0 & 0 & 0 & 0 & 0 & 0 & 0 & 0 \\
21211 & 1 & -1 & -1 & 1 & -1 & 2 & -2 & 1 & 1 & 1 & -1 \\
22111 & 0 & 0 & 0 & 0 & 0 & 0 & 0 & 0 & 0 & 0 & 0 \\
22121 & 1 & 0 & 1 & -1 & 0 & 0 & 0 & 0 & 0 & 0 & 0 \\
22211 & 0 & 0 & 0 & 0 & 0 & 0 & 0 & 0 & 0 & 0 & 0 \\
22221 & 0 & 0 & 0 & 0 & 0 & 0 & 0 & 0 & 0 & 0 & 0 \\\hline
\end{tabular}
\]
\caption{The $\overline{\zh}$-invariants in the spanning set $\mathscr{S}_5$ for the knots in Table \ref{table_the_eleven}.} \label{table_zh_the_eleven}
\end{table}

\subsection{$\mathscr{VC}$ is not abelian} \label{sec_vc} In this section, we use the extended Artin representation to study the concordance group of long virtual knots. As a second application, we will also show that the (closed) virtual knot $6.8451$ is not slice.

\begin{figure}[htb]
\centerline{
\begin{tabular}{c} \\
\def\svgwidth{4in}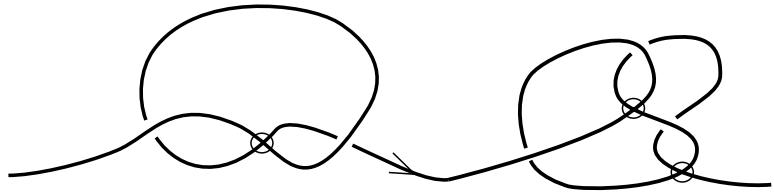
\end{tabular}
}
\caption{A connected sum of the long virtual knots $2.1$ and $3.1$.} \label{fig_cs_defn}
\end{figure}

\begin{theorem} \label{thm_vc_not_abelian} The concordance group $\mathscr{VC}$ of long virtual knots is not abelian.
\end{theorem}

\begin{proof} Consider the long virtual knots given by the following Gauss codes: 
\[
\text{O1-O2-U1-U2-} \text{ and } \text{O1-O2-U1-O3+U2-U3+}.
\]
These are $2.1$ and $3.1$, respectively, as virtual knots in Green's table. Let $\vec{K}=2.1 \# 3.1$ and $\vec{K}^*=3.1 \# 2.1$. We will show that $A^{(q)}_{\zh}(\vec{K}) \ne A^{(q)}_{\zh}(\vec{K}^*)$ for some sufficiently large $q$. Let $\widetilde{\lambda}$ be an extended longitude word of $\vec{K}$ in $\Zh(\vec{K})$ and $\widetilde{\lambda}^*$ an extended longitude word for $\vec{K}^*$ in $\Zh(\vec{K}^*)$. By Lemma \ref{lemma_ab}, it is sufficient to show that $
\phi^{(q)}(\widetilde{\lambda}^{-1})\phi^{(q)}(\widetilde{\lambda}^*)$ is not in $F_{q}$ for some $q$, where $F$ is the free group on two letters $a,v$. For $q \le 7$,  $\phi^{(q)}(\widetilde{\lambda}^{-1})\phi^{(q)}(\widetilde{\lambda}^*) \equiv 1 \mod F_{q}$. For $q=8$, we have: 
\begin{eqnarray*}
\phi^{(8)}(\widetilde{\lambda}) &=& v a^{-1} v^{-2} a^{-1} v a v^2 a v^{-1} a^{-1} v (a^{-1} v^{-2})^2 v^{-1} a v^2 a v^{-1} a^{-1} v a^{-1} v^{-2} a^{-1} (v^2 a)^2 v^{-1}\cdot \\
& & a v a^{-1} v^{-2} a^{-1} v^{-1} a v^2 a v^{-1} a v a^{-1} v^{-2} a^{-1} v a v^2 a v^{-1} a^{-1} v a^{-1} v^{-2} a^{-1} v a^3\\
\phi^{(8)}(\widetilde{\lambda}^*) &=& v a^{-1} v^{-3} a^{-1} v^2 a v^{-1} a v a^{-1} v a v^{-1} a^{-1} v a^{-1} v^{-2} a v^3 a v^{-1} a^{-1} v a^{-1} v^{-3} a^{-1} v^2 a v^{-1} \cdot \\
& & a v a^{-1} v^{-2} a v^{-1} a^{-1} v a^{-1} v^{-2} a v^3 a v^{-1} a^{-1} v a^{-1} v^{-3} a^{-1} v^2 a v^{-1} a v a^{-1} v a^3
\end{eqnarray*}
Projecting to $F/F_{8}$ and computing $
\phi^{(8)}(\widetilde{\lambda}^{-1})\phi^{(8)}(\widetilde{\lambda}^*)$ (via \emph{GAP}) , we have:
\begin{eqnarray*}
\phi^{(8)}(\widetilde{\lambda}^{-1})\phi^{(8)}(\lambda^*) & \equiv & \overline{g}_{25}^2g_{26}^8g_{27}^4\wwbar{g}_{28}^6\wwbar{g}_{29}^5g_{32}^3g_{33}^3\wwbar{g}_{34}^2\wwbar{g}_{35}\overline{g}_{36}^6\wwbar{g}_{37}^2g_{38}^
4g_{39}^3\wwbar{g}_{40} \mod F_{8},
\end{eqnarray*}
where the commutators $g_{k}$ above are given by:
\begin{align*}
g_{25} &= [ v, a, a, a, a, a, v ] &g_{34} &= [ v, a, v, v, a, a, a ] \\
g_{26} &= [ v, a, a, a, v, a, a ] &g_{35} &= [ v, a, v, v, a, a, v ] \\
g_{27} &= [ v, a, a, a, v, a, v ] &g_{36} &= [ v, a, v, v, a, v, a ]  \\ 
g_{28} &= [ v, a, v, a, a, a, a ] &g_{37} &= [ v, a, v, v, a, v, v ] \\
g_{29} &= [ v, a, v, a, a, a, v ] &g_{38} &= [ v, a, v, v, v, a, a ] \\
g_{32} &= [ v, a, v, a, v, a, a ] &g_{39} &= [ v, a, v, v, v, a, v ] \\
g_{33} &= [ v, a, v, a, v, a, v ] &g_{40} &= [ v, a, v, v, v, v, a ]
\end{align*}
Thus, the long virtual knots $\vec{K}=2.1\#3.1$ and $\vec{K}^*=3.1\#2.1$ are not concordant and we conclude that $\mathscr{VC}$ is not abelian. 
\end{proof}

The \emph{closure} of a long virtual knot $\vec{K}$ is the virtual knot $\text{cl}(\vec{K})$ obtained by identifying its endpoints. If $\vec{K}$ and $\vec{K}^*$ are concordant, then $\text{cl}(\vec{K})$ and $\text{cl}(\vec{K}^*)$ are concordant as virtual knots (see \cite{band_pass}, Section 3.1). Boden and Nagel \cite{boden_nagel} observed that a long virtual knot is slice if and only if its closure is slice. A natural question is whether this is true for all concordance classes of long virtual knots. The above example shows that this is not the case. 

\begin{corollary} There exist non-concordant long virtual knots with concordant closure.
\end{corollary}
\begin{proof} The long virtual knots $\vec{K}=2.1\#3.1$ and $\vec{K}^*=3.1\#2.1$ have identical closure and hence have concordant closure. By the proof of Theorem \ref{thm_vc_not_abelian}, $\vec{K}$ and $\vec{K}^*$ are not concordant as long virtual knots.
\end{proof}

The method used in the proof of Theorem \ref{thm_vc_not_abelian} can also be used to show that (closed) virtual knots are not slice in some cases. Suppose that a closed virtual knot can be represented as $K\leftrightharpoons\text{cl}(\vec{K}\# \vec{K}^*)$ for some long virtual knots $\vec{K},\vec{K}^*$. This will be slice if and only if $\vec{K}\# \vec{K}^*=1 \in \mathscr{VC}$, which occurs if and only if $\vec{K}=(\vec{K}^*)^{-1}$ in $\mathscr{VC}$. Then Lemma \ref{lemma_ab} implies that $\phi^{(q)}(\widetilde{\lambda})\phi^{(q)}(\widetilde{\lambda}^*)\in F_{q}$, where $\widetilde{\lambda}$, $\widetilde{\lambda}^*$ are extended longitude words for $\vec{K},\vec{K}^*$, respectively.  
%This is because the inverse in VC is found by changing all of the crossing signs and writing the Gauss code backwards. This implies that the longitude is l^(-1). Since p^q(l^-1)=(p^q(l))^-1, the statement follows.
\begin{figure}[htb]
\centerline{
\begin{tabular}{c}
\def\svgwidth{1.4in}
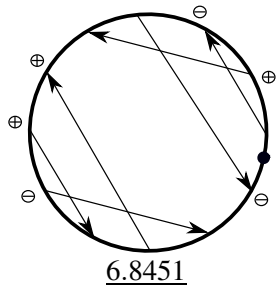 \\ \\
\end{tabular}
}
\caption{6.8451 is the closure of the product of two long virtual knots, after a shift in base point. The depicted base point coincides with Green's table \cite{green}.}\label{fig_six8451_gauss}
\end{figure}

After shifting the base point to the right by one endpoint, the virtual knot $K=6.8451$ can be represented as $\text{cl}(\vec{K}\#\vec{K}^*)$ (see Figure \ref{fig_six8451_gauss}):
\begin{align*}
K &\leftrightharpoons \text{cl}(\text{U1-O2-O3+U2-O1-U3+U4+O5+O6-U5+O4+U6-}), \\
 \vec{K} &= \text{U1-O2-O3+U2-O1-U3+},\\
 \vec{K}^* &= \text{U1+O2+O3-U2+O1+U3-}.
\end{align*}
The first $q$ for which $\phi^{(q)}(\widetilde{\lambda})\phi^{(q)}(\widetilde{\lambda}^*) \not\in F_{q}$ is $6$. The length of both $\phi^{(6)}(\widetilde{\lambda})$, $\phi^{(6)}(\widetilde{\lambda}^*)$ is 2788. They can be written as products of commutators as follows.
\begin{align*}
\phi^{(6)}(\widetilde{\lambda}) &\equiv\wwbar{g}_4\wwbar{g}_5 g_6^2 g_7^2 g_8\wwbar{g}_9^3 \overline{g}_{10}\overline{g}_{11}g_{12}\overline{g}_{13}^2\wwbar{g}_{14} & \mod F_{6}  \\
\phi^{(6)}(\widetilde{\lambda}^*) &\equiv g_4g_5\wwbar{g}_8g_{10}\overline{g}_{11}^2\wwbar{g}_{12} g_{14} & \mod F_{6} \\
\phi^{(q)}(\widetilde{\lambda})\phi^{(q)}(\widetilde{\lambda}^*) &\equiv g_6^2g_7^2\wwbar{g}_9^3\wwbar{g}_{11}^3\wwbar{g}_{13}^2 & \mod F_{6}
\end{align*}
Thus, $6.8451$ is not slice as a (closed) virtual knot. 

\subsection{Beyond the first non-vanishing invariants} \label{sec_beyond} Here we give some example calculations of the $\overline{\zh}$-invariants beyond the first non-vanishing order. As the properties of $\epsilon_J$ in Lemma \ref{lemma_fenn} no longer apply in this case, the Magnus expansion coefficients will be calculated directly from their definition. This was done with a \emph{Mathematica} program, which is available on the author's website. It calculates the invariant $\overline{\zh}_J$ for $J=22\cdots 211\cdots 1$.  

\begin{figure}[htb]
\centerline{
\begin{tabular}{c}
\def\svgwidth{3.2in}
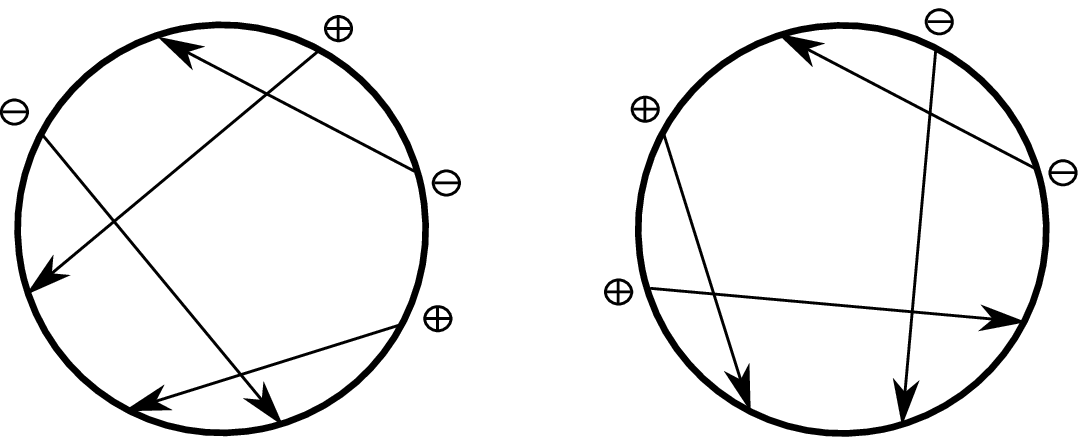\\ \\
\end{tabular}
}
\caption{Two virtual knots with the same first non-vanishing $\overline{\zh}$-invariants but with different higher order $\overline{\zh}$-invariants.}\label{fig_four19_four32}
\end{figure}

Consider the virtual knots $K=4.19$ and $K^*=4.32$ shown in Figure \ref{fig_four19_four32}. These knots have the same generalized Alexander polynomial \cite{green} and hence the same odd writhe, Henrich-Turaev polynomial, and affine index polynomial. By \cite{bcg1}, $K,K^*$ have the same graded genus and the same slice genus. 

The first non-vanishing $\overline{\zh}$-invariants for $K,K^*$ occur at order three. Denote by $\widetilde{\lambda},\widetilde{\lambda}^*$ extended longitude words for $K,K^*$, respectively. For the invariants of order three, we find that: 
\[
\phi^{(4)}(\widetilde{\lambda}) \equiv \phi^{(4)}(\widetilde{\lambda}^*) \equiv g_4^2g_5^2 \mod F_{4}.
\]
Thus, $\overline{\zh}_J(K)=\overline{\zh}_J(K^*)$ for all $J$ with $|J|=3$. Furthermore, we see that the $\overline{\zh}$-invariants of order four are defined only $\!\!\pmod 2$. For the invariants of order four, we find:
\begin{align*}
\phi^{(5)}(\widetilde{\lambda}) &=v \overline{a} \overline{v} a \overline{v} a v \overline{a}\overline{v}^2 \overline{a} v^3 a \overline{v} a v \overline{a}\overline{v} \\
\phi^{(5)}(\widetilde{\lambda}^*) &= \overline{v}\overline{a}v a v a^2 v \overline{a} \overline{v} \overline{a} \overline{v}
\end{align*}
The $\overline{\zh}$-invariants are given by the residue classes:
\begin{align*}
\overline{\zh}_{2221}(K) & \equiv 1 \pmod 2 & \overline{\zh}_{2211}(K) & \equiv 0 \pmod 2 & \overline{\zh}_{2111}(K) & \equiv 1 \pmod 2 \\
\overline{\zh}_{2221}(K^*) & \equiv 1 \pmod 2 & \overline{\zh}_{2211}(K^*) & \equiv 1 \pmod 2 & \overline{\zh}_{2111}(K^*) & \equiv 0 \pmod 2
\end{align*}
Hence, $4.19$ and $4.32$ are not concordant.

\subsection{Slice movies} \label{sec_movies} In this section, we give slice movies for 22 of the virtual knots in Table \ref{table_status_unknown}. We will use the technique of slicing Gauss diagrams that was introduced in \cite{bcg1}. A saddle move can be represented in a Gauss diagram with a \emph{saddle chord} that connects the two arcs involved in the move. The correspondence is depicted in Figure \ref{fig_saddle_chord}. Note that the surgery itself is not drawn in the Gauss diagram. The two sides of the saddle chord are the arcs of the new link diagram after the move. An arrow endpoint may be slid over a saddle chord, as long as the endpoint stays on the same side of the chord during the slide. If one side of a chord contains no arrow endpoints, then this indicates an unknotted and unlinked component circle that can be eliminated with a death move. 
\newline

\begin{figure}[htb]
\centerline{
\begin{tabular}{c}
\def\svgwidth{2.5in}
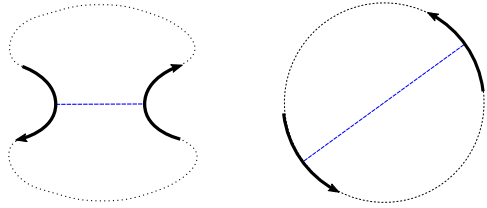\\ \\
\end{tabular}
}
\caption{A saddle chord (blue dashed line) as it appears in a virtual knot diagram (left) and in a Gauss diagram (right).}\label{fig_saddle_chord}
\end{figure}

\noindent\textbf{Examples (5.1216, 5.1963):} Slice movies for each of these two virtual knots is given in Figure \ref{fig_last_five} (see also Figure \ref{fig_conc_example}). Both can be sliced using the same technique. First an $\Omega 3$ move is used to rearrange the arrows. The arcs involved in the $\Omega 3$ move are emphasized with the red shaded arcs drawn outside of the circle.  Then a saddle chord is added. After the saddle, an $\Omega 2$ move is performed. The rightmost frame in each movie is a trivial $2$-component link.  This can be seen by performing an $\Omega 2$ move followed by an $\Omega 1$ move. These movies complete the slice status classification of all virtual knots up to $5$ classical crossings.
\newline 

\begin{figure}[htb]
\centerline{
\begin{tabular}{cc}
\underline{5.1216:} & 
\begin{tabular}{c}
\def\svgwidth{4.9in}
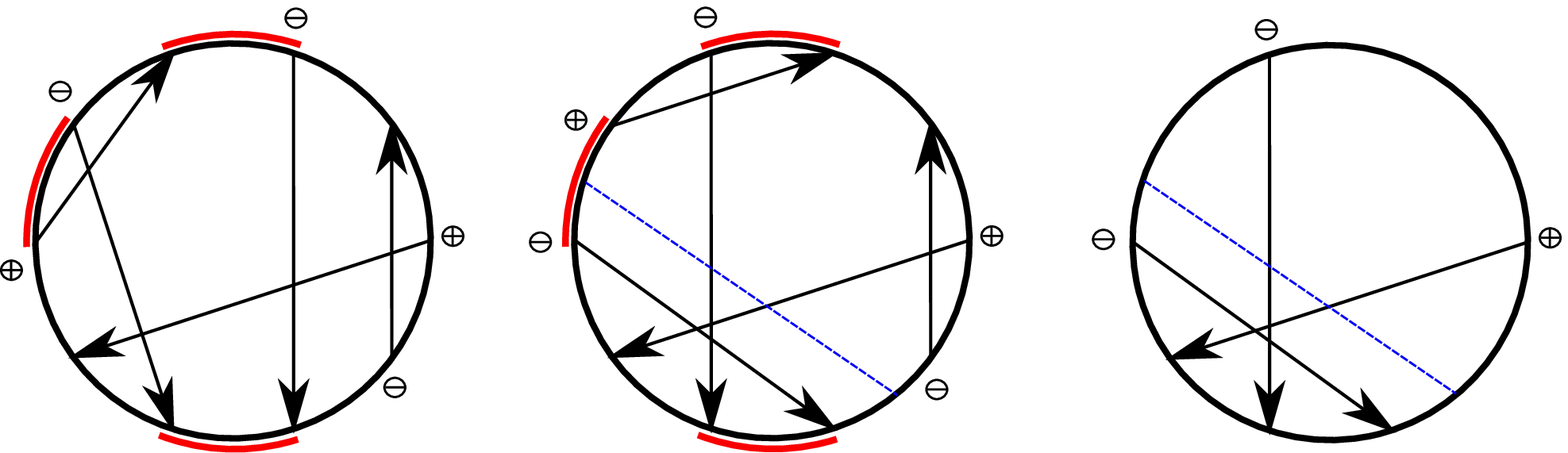
\end{tabular} \\ 
\underline {5.1963:} & \begin{tabular}{c}
\def\svgwidth{4.9in}
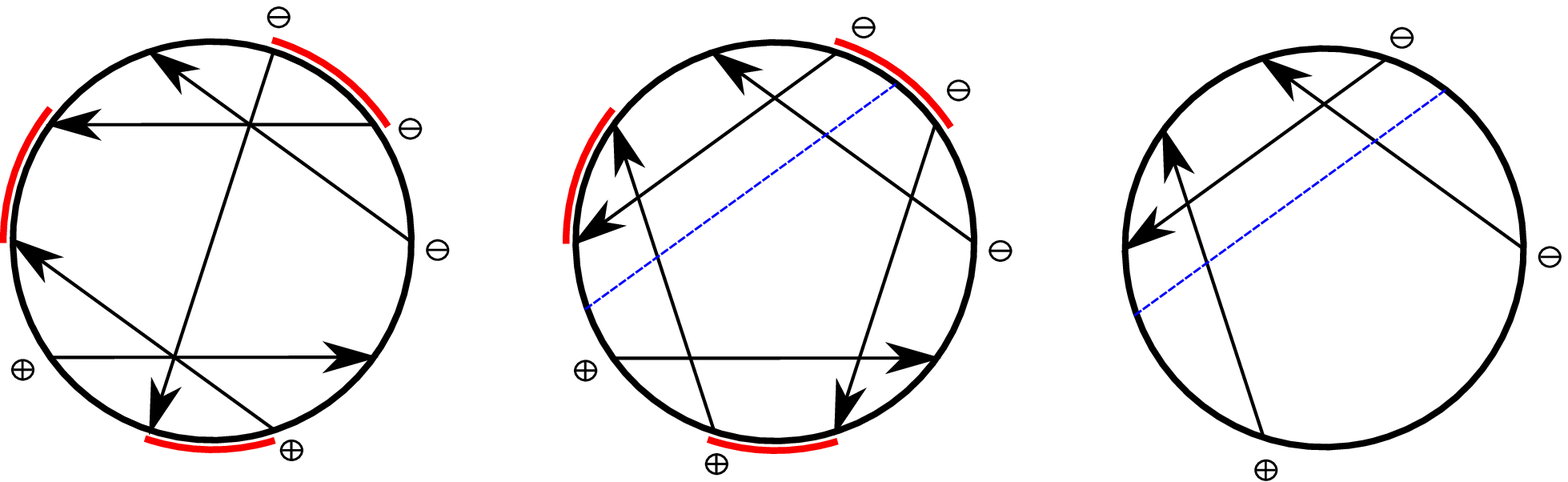
\end{tabular}
\end{tabular}
}
\caption{Slice movies for the two five crossing virtual knots in Table \ref{table_status_unknown}.} \label{fig_last_five}
\end{figure}

\noindent\textbf{Examples (6.14778, 6.14781, 6.46936, 6.46937)}: A slice movie for 6.14778, is given in Figure \ref{fig_easy_conc_example}. A saddle chord is drawn in the first frame. After the saddle, the two green arrows are rearranged so that an $\Omega 2$ move is visible. After deleting the green arrows, we see that the $2$-component link has an unknotted and split component. This component is eliminated with a death move. Next, we perform the indicated $\Omega 3$ move (red shaded arcs). Two $\Omega 2$ moves (blue and gold shaded arcs) can then be used to reduce to the unknot. Thus, 6.14778 is slice. The same procedure works for the other virtual knots listed. The first frame of the movie for each is given Figure \ref{fig_6_jones}.
\newline

\begin{figure}[htb]
\centerline{
\begin{tabular}{c}
\def\svgwidth{4.9in}
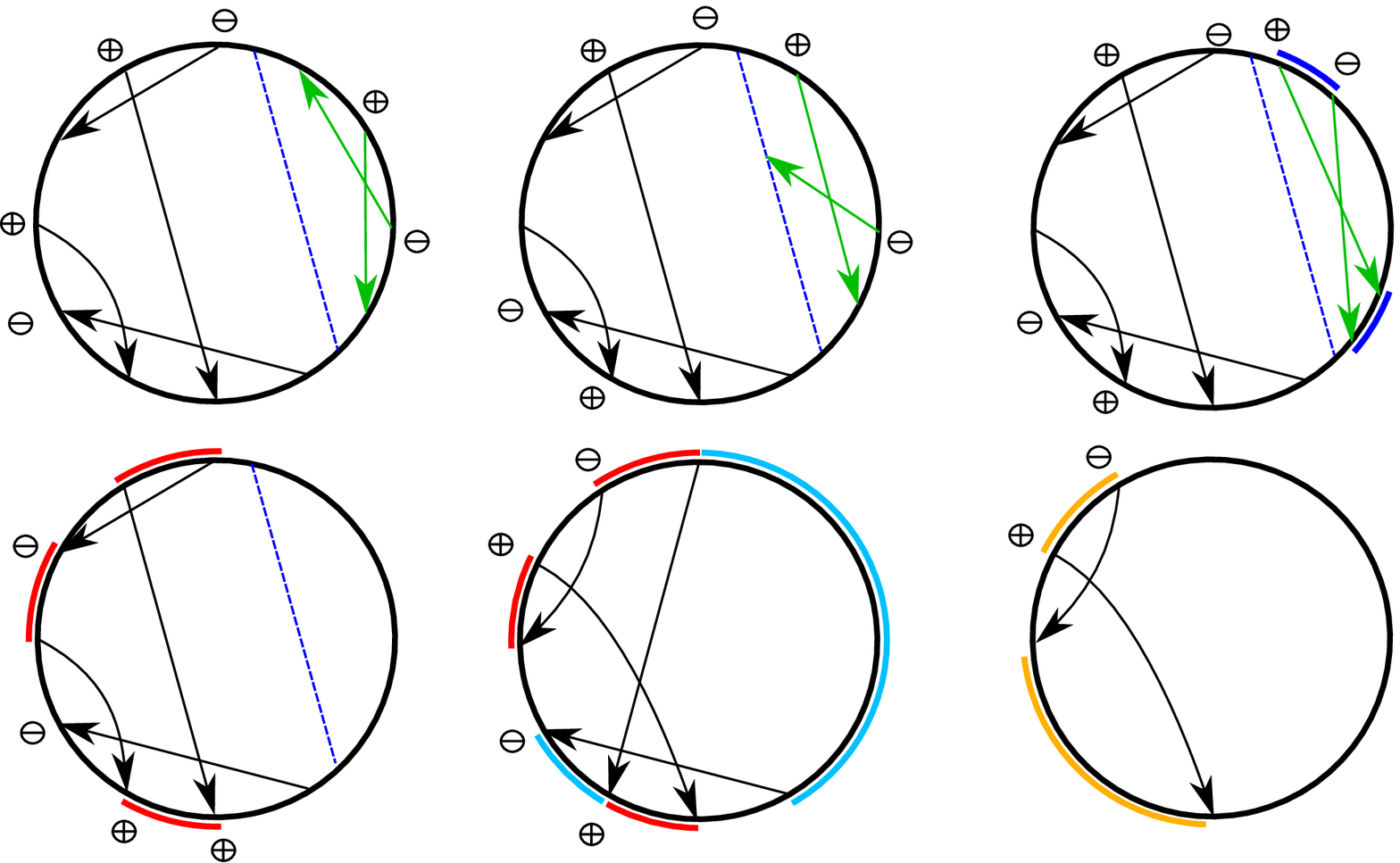
\end{tabular}
}
\caption{A slice movie for 6.14778 (top left). The movie starts with a saddle, goes left-to-right on the first row, and then left-to-right on the second row. } \label{fig_easy_conc_example}
\end{figure}

\begin{figure}[htb]
\centerline{
\begin{tabular}{c}
\def\svgwidth{6in}
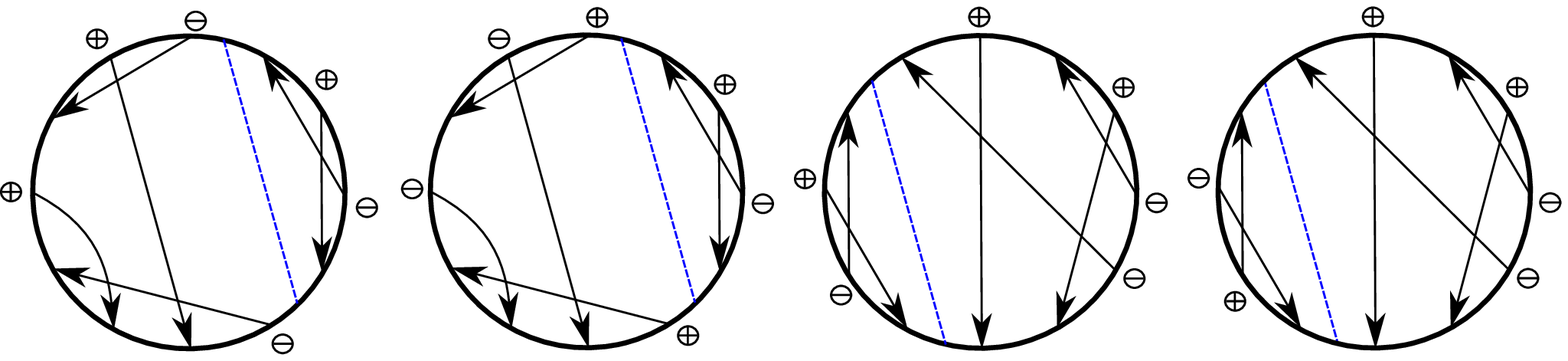 \\ \\
\end{tabular}
} 
\caption{The same procedure as in Figure \ref{fig_easy_conc_example} works for these four knots.} \label{fig_6_jones}
\end{figure}

\noindent\textbf{Example (6.5958):} The slice movie is given in Figure \ref{fig_6_5958}. After performing the saddle, we rearrange the arrow highlighted in green. This sets up an $\Omega3 e$ move. The red arcs indicate the three arcs involved in the move. After the $\Omega 3e$ move, two $\Omega2$ moves are available (blue shaded arcs). A last $\Omega 2$ move (gold shaded arcs) gives a two component unlink. A death move applied to one of the unknots gives a concordance to the the unknot. Hence, 6.5958 is slice.
\newline

\begin{figure}[htb]
\centerline{
\begin{tabular}{c}
\def\svgwidth{4.9in}
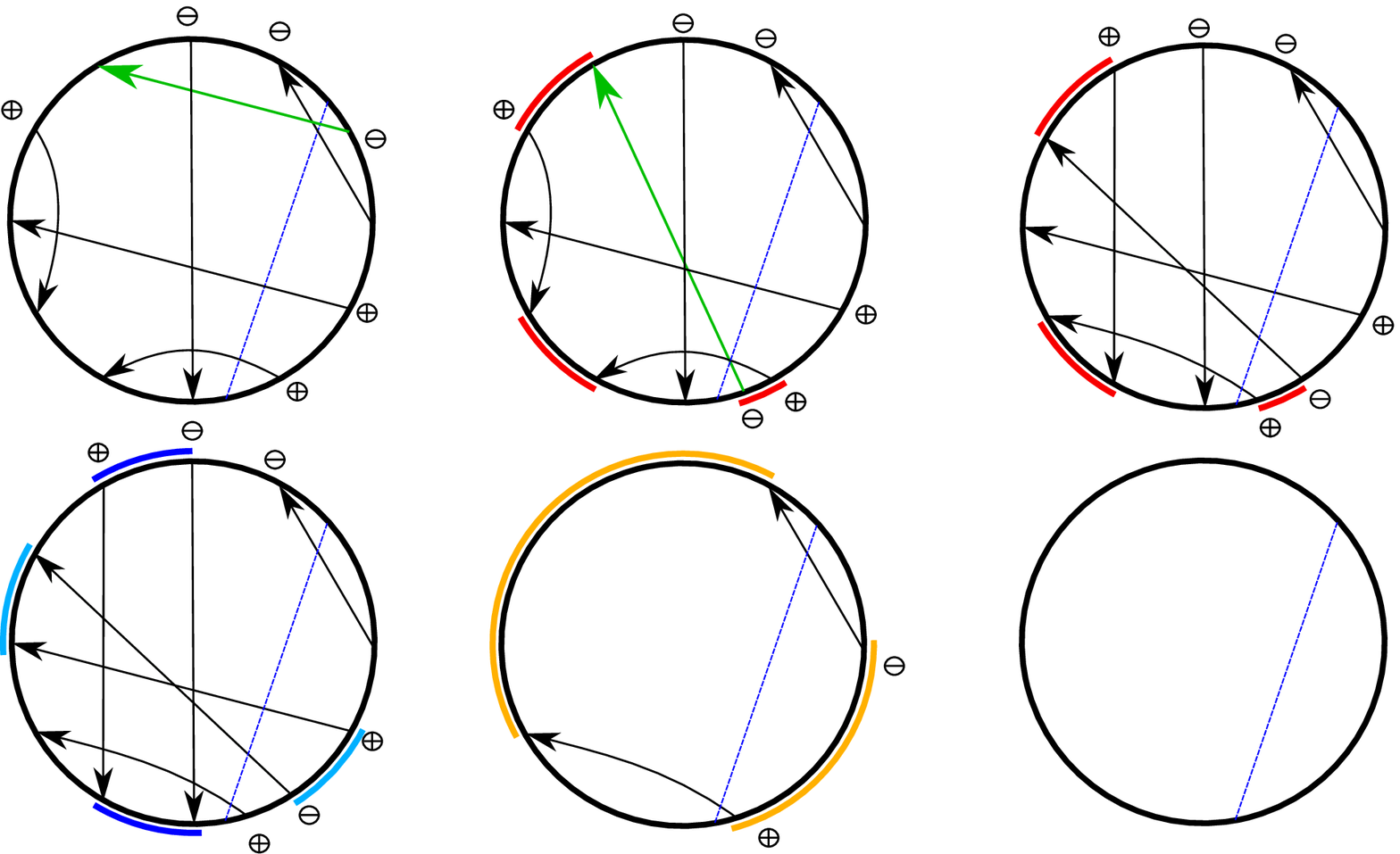
\end{tabular}
}
\caption{A slice movie for 6.5958 (top left). The movie starts with a saddle, goes left-to-right on the first row, and then left-to-right on the second row.} \label{fig_6_5958}
\end{figure}

\noindent\textbf{Examples (Figures \ref{fig_5_done}, \ref{fig_5_done_II}, \ref{fig_5_done_III}):} The technique used for 5.1216 and 5.1963 also slices 15 additional virtual knots from Table \ref{table_status_unknown}.

\begin{figure}[htb]
\centerline{
\begin{tabular}{c}
\def\svgwidth{6in}
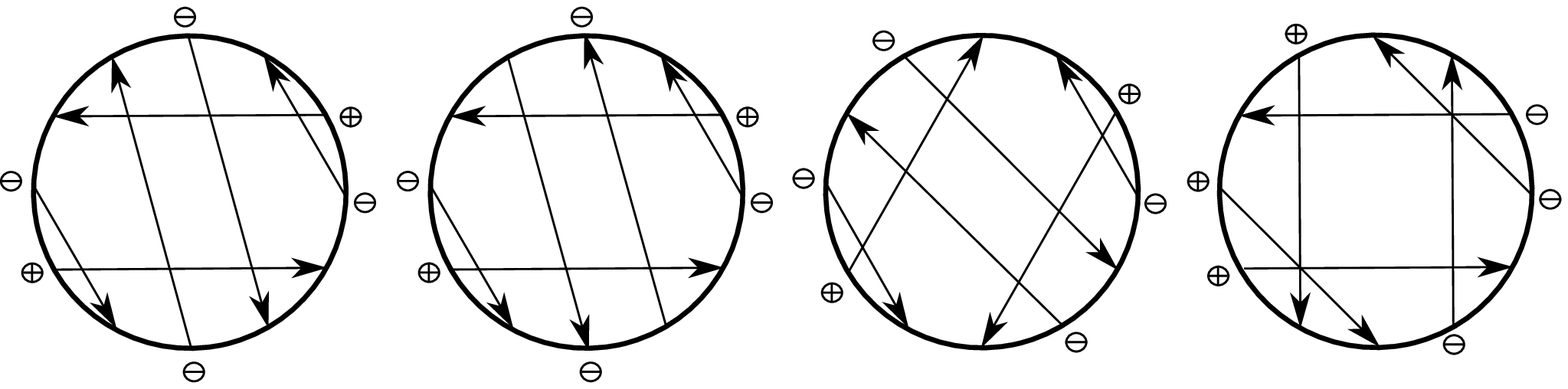\\ \\
\end{tabular}
} 
\caption{The last four (out of 92800) virtual knots of up to six classical crossings having unknown slice status.} \label{fig_4_unknown}
\end{figure}

\section{Questions for further research} \label{sec_questions}
\subsubsection*{(Question 1) Which of the virtual knots in Figure \ref{fig_4_unknown} is slice?} Figure \ref{fig_4_unknown} shows the four virtual knots whose slice status remains unknown. Their $\overline{\zh}$-invariants vanish for every order calculated. For 6.31455, the $\overline{\zh}$-invariants vanish up to order at least nine. For 6.52373 and 6.86951, they vanish up to order at least eight. For 6.62002, they vanish up to order at least seven. Further calculation was not possible due to length of the extended longitudes. For example, the extended longitude $\widetilde{\lambda}^{(9)}$ for 6.86951 has length 1109468 after reduction in \emph{GAP}.

\subsubsection*{(Question 2) If $\overline{\zh}_J(K)=0$ for all $J$, is $K$ concordant to an AC knot?}  The answer to the corresponding question for classical links is negative. Cochran and Orr \cite{cochran_orr} proved that there are links in $S^3$ that have vanishing $\bar{\mu}$-invariants but are not concordant to any boundary link.

\subsubsection*{(Question 3) Are the $\overline{\zh}$-invariants of finite-type?} In the case of knots in $S^3$, it follows from a result of Ng \cite{ng}, that there are no nontrivial rational valued finite-type concordance invariants of knots in $S^3$. For multi-component string links, the $\bar{\mu}$-invariants are of finite-type (see Lin \cite{lin_power}, Bar-Natan \cite{bar_natan_milnor}). Furthermore, Habegger and Mausbaum \cite{HM} showed that the only rational valued finite-type concordance invariants of string links are Milnor's concordance invariants and products thereof. Integer valued concordance invariants of virtual string links can be defined analogously using the Magnus expansion. Can they be used to define a universal finite-type concordance invariant of long virtual knots? 

\subsubsection*{(Question 4) Which $\overline{\zh}$-invariants control the vanishing of the generalized Alexander polynomial?} For $2$-component links in $S^3$, Traldi \cite{traldi} determined which $\bar{\mu}$-invariants determine the vanishing of the multi-variable Alexander polynomial: $\Delta_L(x,y)=0$ if and only if all invariants of the form $\bar{\mu}_{11\cdots 22}(L)$ are vanishing. What are the corresponding results for the generalized Alexander polynomial?

\subsubsection*{(Question 5) What is the correct form for the cycle relation for our $\bar{\mu}$-invariants?} It follows from a result of Cimasoni and Turaev \cite{ct} that the deviation from symmetry of the virtual linking number is captured by the intersection form. If the $2$-component virtual link $J \cup K$ is represented by a link $\mathscr{J} \sqcup \mathscr{K}$ in a thickened surface $\Sigma \times I$, $\vlk(J,K)-\vlk(K,J)=[\rho(\mathscr{J})]\cdot[\rho(\mathscr{K})]$. Here $\rho:\Sigma \times I \to \Sigma$ is projection to the first factor and $\cdot$ denotes the intersection form on $\Sigma$. What is the appropriate correction term for the cycle relation for the $\bar{\mu}$-invariants of order greater than two?

\subsubsection*{(Question 6) How many independent $\overline{\zh}$-invariants are there for every order?} The question for classical links was answered by Orr \cite{orr}. Spanning sets of small order were given in Table \ref{table_spanning}, but these are unlikely to be minimal.

\subsubsection*{(Question 7) Is there a theory of virtual link derivatives?} Cochran \cite{cochran} showed that the first non-vanishing $\bar{\mu}$-invariants of a link in $S^3$ can be efficiently calculated using link derivatives. Recall that link derivatives are found by intersecting Seifert surfaces of the link components.  Can something similar be developed for links in $\Sigma \times I$ with homologically trivial components? As the $\overline{\zh}$-invariants vanish for all homologically trivial knots in thickened surfaces, this would not help in calculating the $\overline{\zh}$-invariants. It may be useful, however, for understanding the $\bar{\mu}$-invariants of virtual links. 

\subsection*{Acknowledgments} The author is grateful for many helpful discussions with H. U. Boden, H. A. Dye, and R. Todd.  Thanks are also due to A. Nicas, P. Pongtanapaisan W. Rushworth for their advice about an earlier version of this paper. Calculations of the Rasmussen invariant were provided by H. Karimi and calculations for the generalized Alexander polynomial were provided by L. White. The author would also like to thank the Ohio State University (Marion campus) for funds and release time.

\begin{figure}[htb]
\centerline{
\begin{tabular}{cc}
\underline {6.5588:} & \begin{tabular}{c}
\def\svgwidth{4.9in}
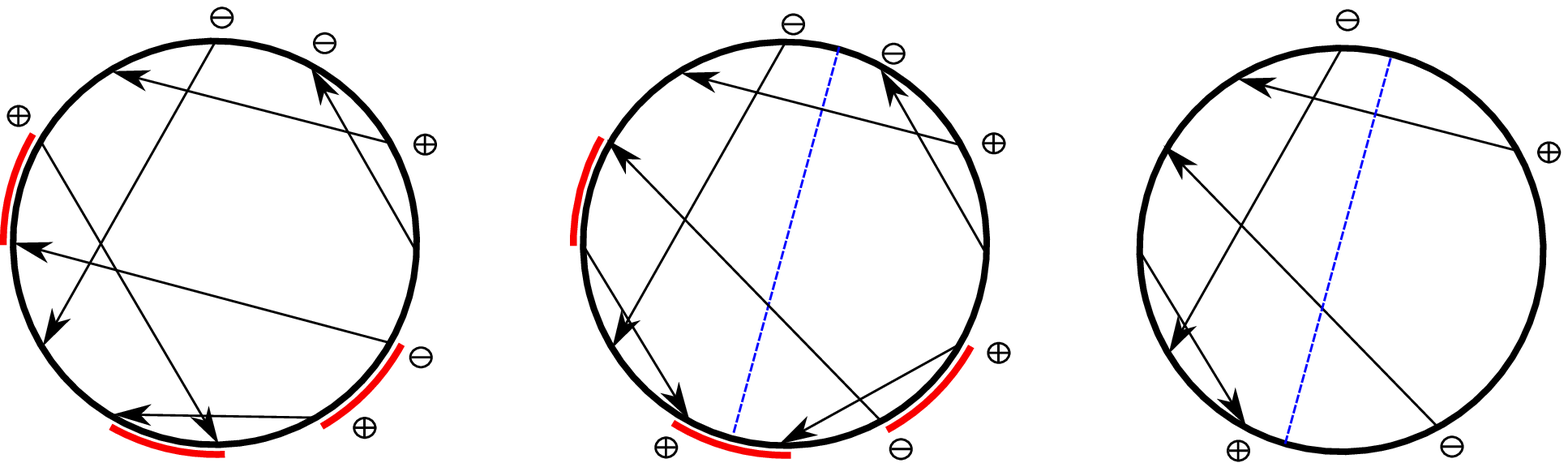
\end{tabular}
\\ 
\underline {6.7388:} & \begin{tabular}{c}
\def\svgwidth{4.9in}
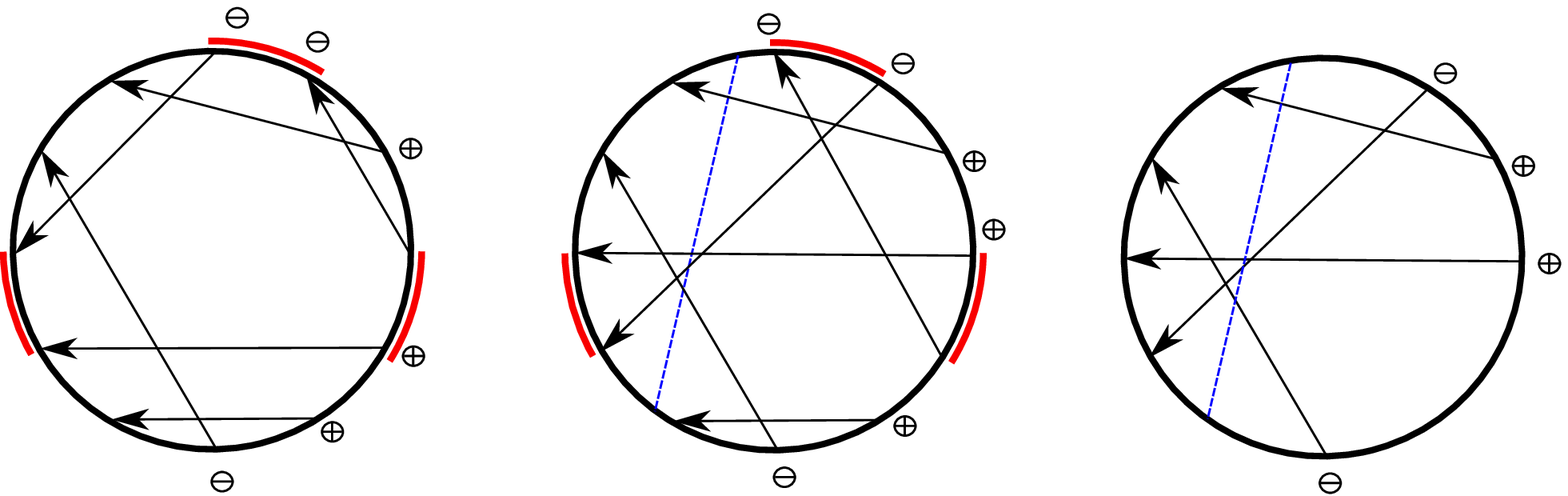
\end{tabular}
\\ 
\underline {6.33334:} & \begin{tabular}{c}
\def\svgwidth{4.9in}
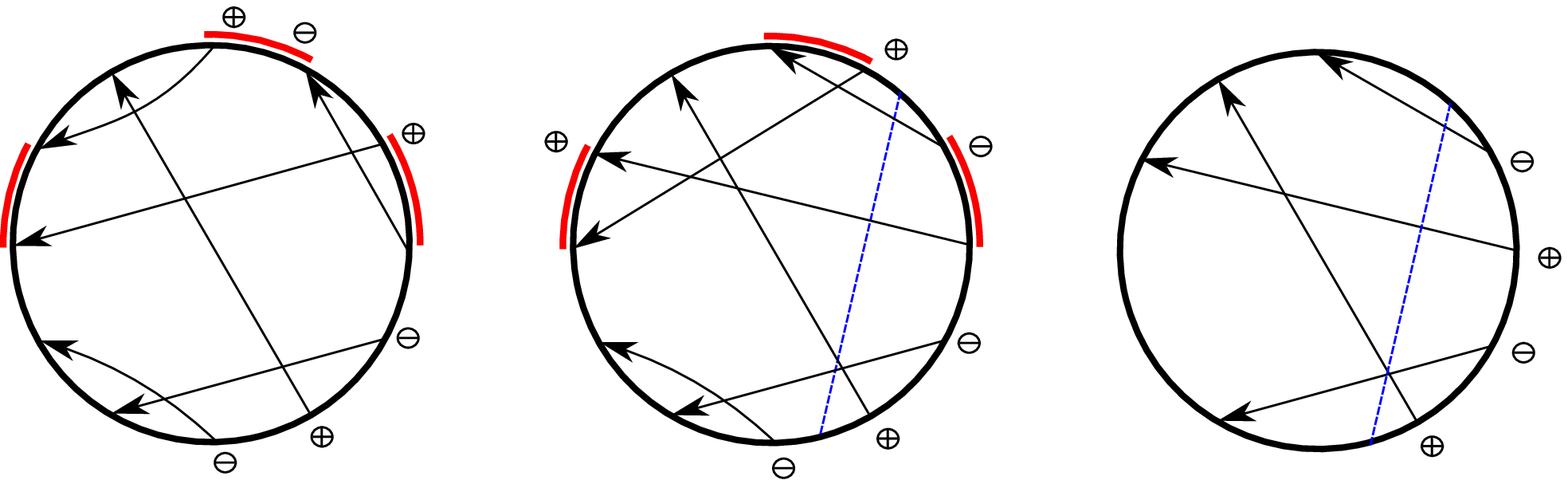
\end{tabular}
\\
\underline {6.37879:} & \begin{tabular}{c}
\def\svgwidth{4.9in}
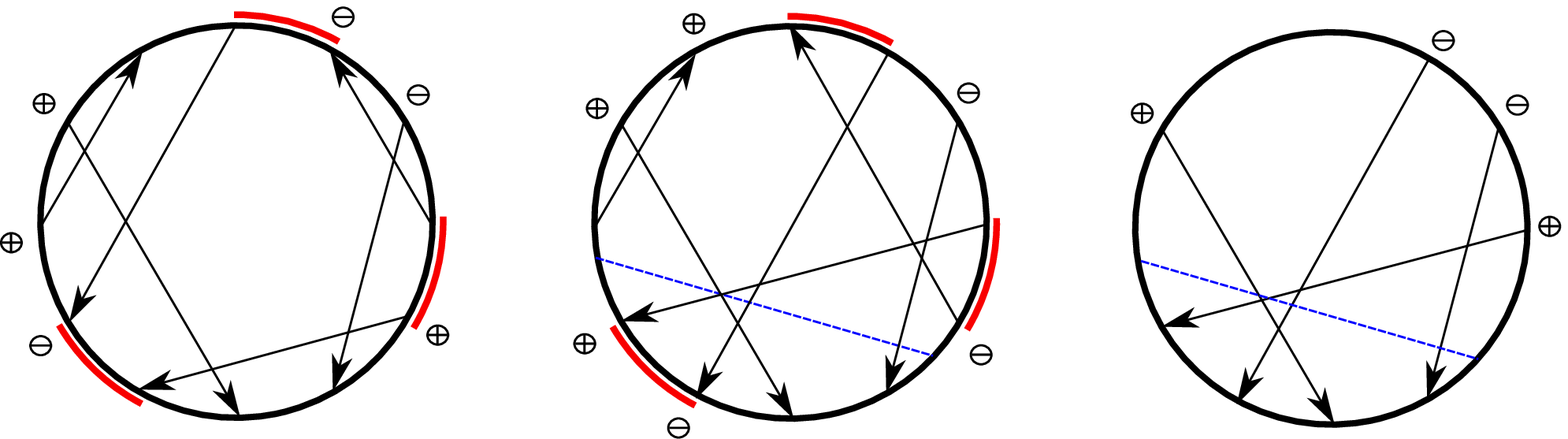
\end{tabular}
\\
\underline {6.38158:} & \begin{tabular}{c}
\def\svgwidth{4.9in}
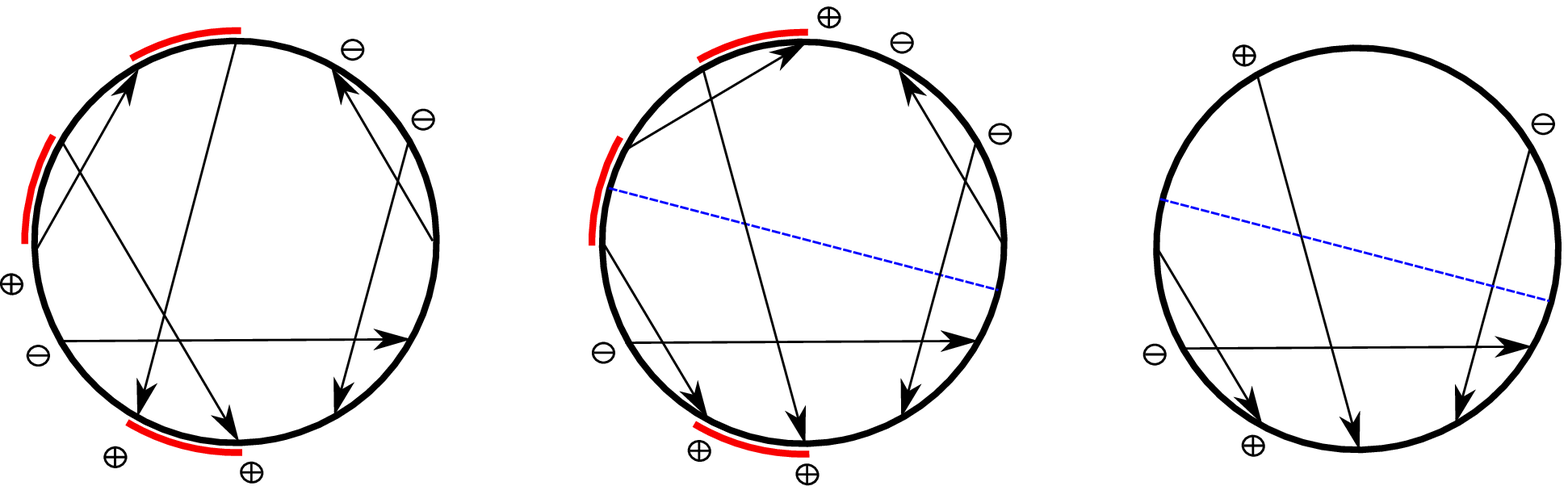
\end{tabular}
\end{tabular}
}
\caption{Slice movies for some six crossing knots.} \label{fig_5_done}
\end{figure}

\begin{figure}[htb]
\centerline{
\begin{tabular}{cc} 
\underline {6.38183:} & \begin{tabular}{c}
\def\svgwidth{4.9in}
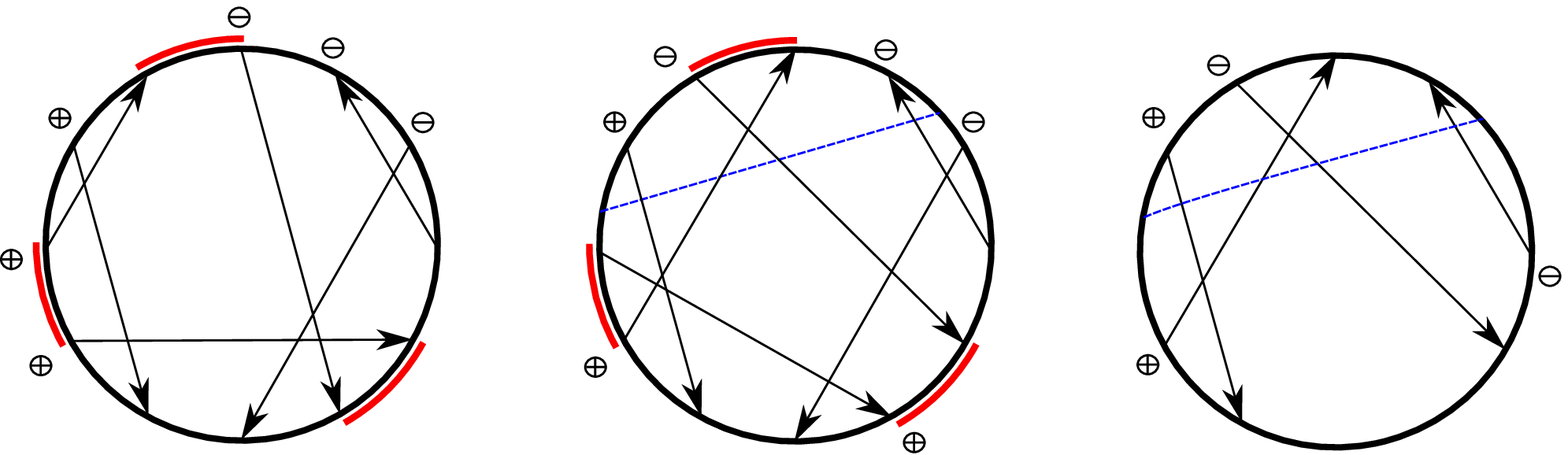
\end{tabular}
\\
\underline {6.47024:} & \begin{tabular}{c}
\def\svgwidth{4.9in}
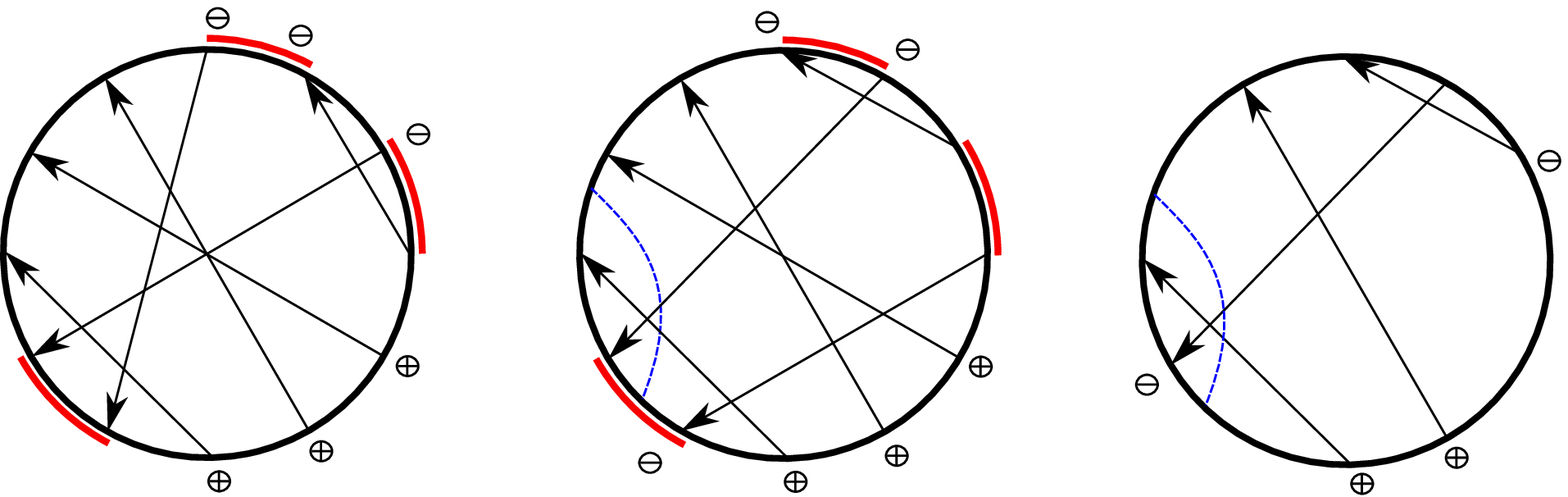
\end{tabular}
\\
\underline {6.49338:} & \begin{tabular}{c}
\def\svgwidth{4.9in}
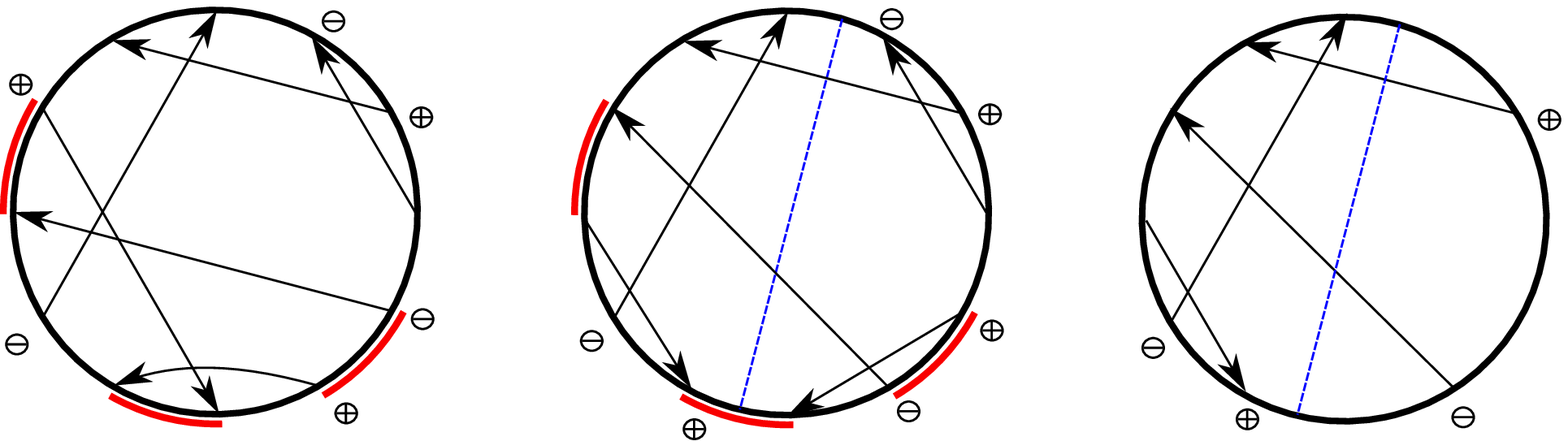
\end{tabular}
\\
\underline {6.69085:} & \begin{tabular}{c}
\def\svgwidth{4.9in}
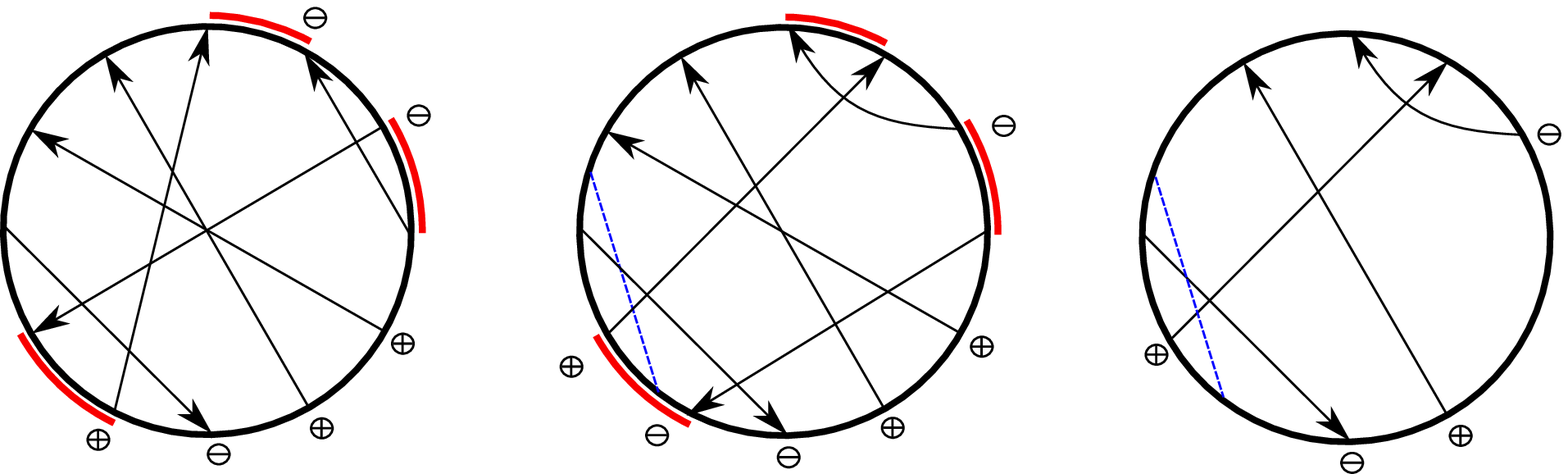
\end{tabular}
\\
\underline {6.70767:} & \begin{tabular}{c}
\def\svgwidth{4.9in}
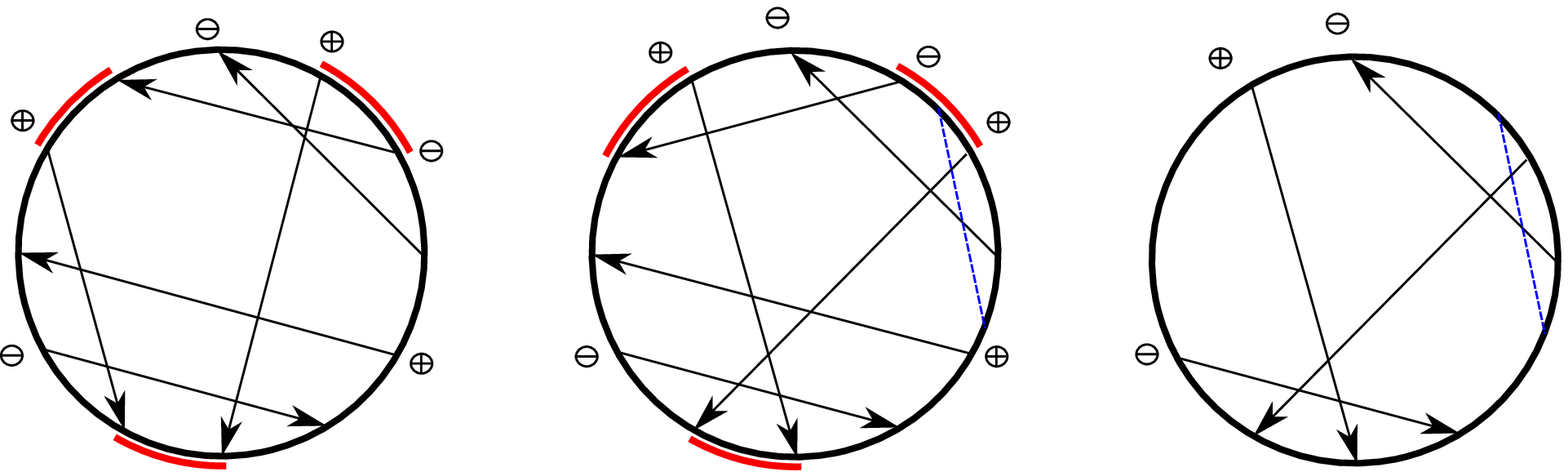
\end{tabular}
\end{tabular}
}
\caption{Slice movies for some six crossing knots.}\label{fig_5_done_II}
\end{figure}

\begin{figure}[htb]
\centerline{
\begin{tabular}{cc} 
\underline {6.71306:} & \begin{tabular}{c}
\def\svgwidth{4.9in}
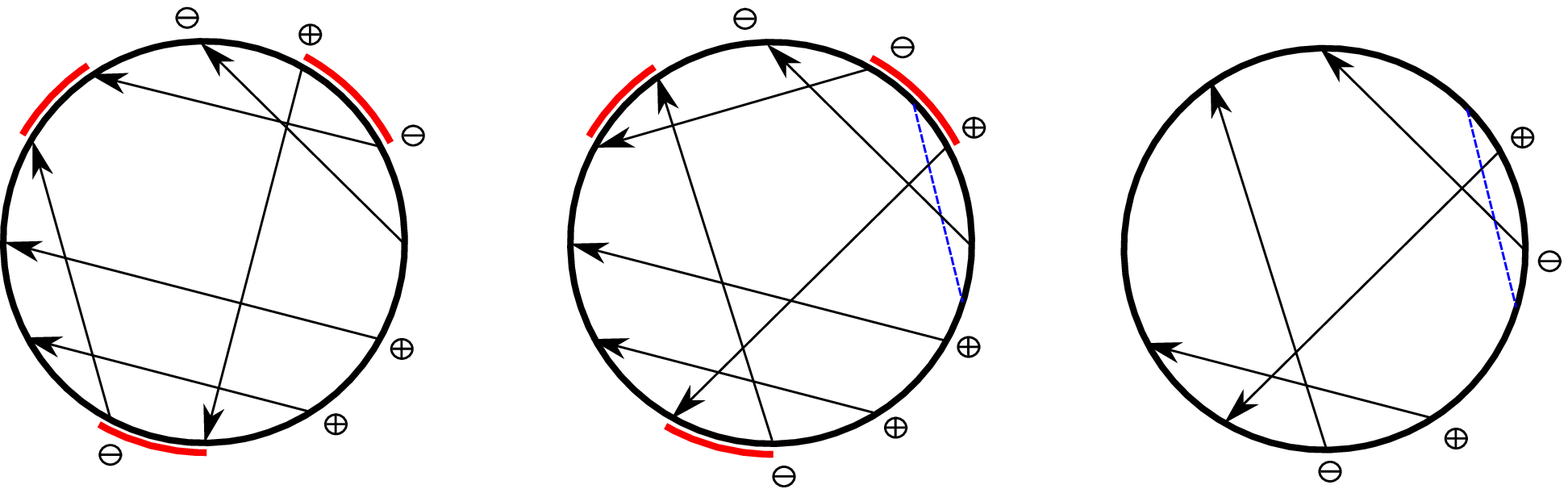
\end{tabular}
\\
\underline {6.72353:} & \begin{tabular}{c}
\def\svgwidth{4.9in}
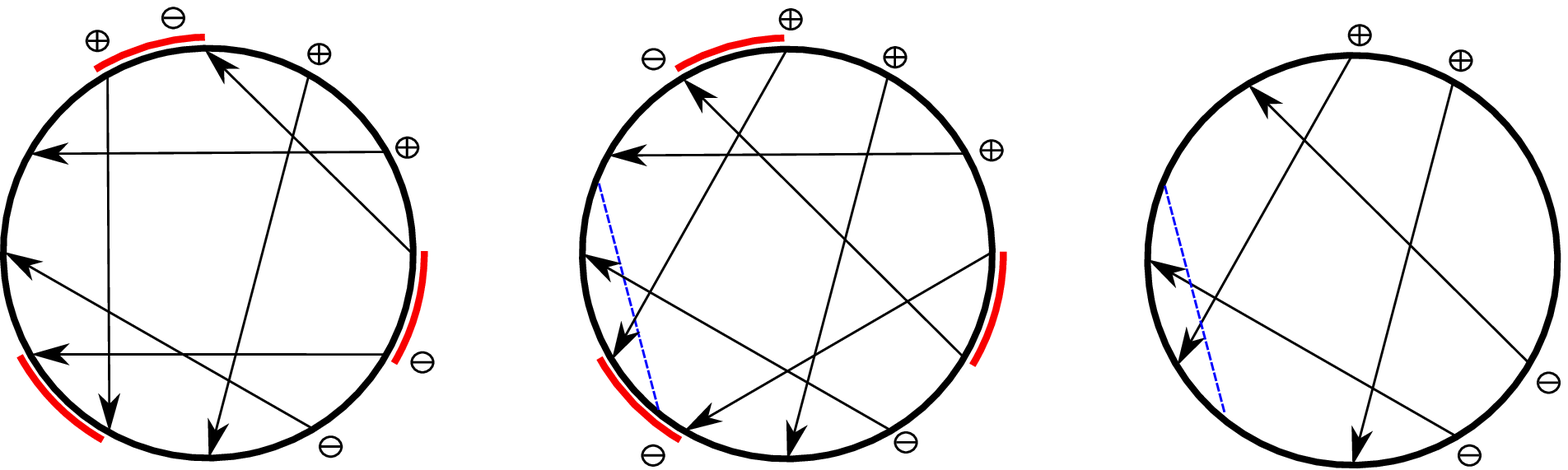
\end{tabular}
\\ 
\underline {6.76488:} & \begin{tabular}{c}
\def\svgwidth{4.9in}
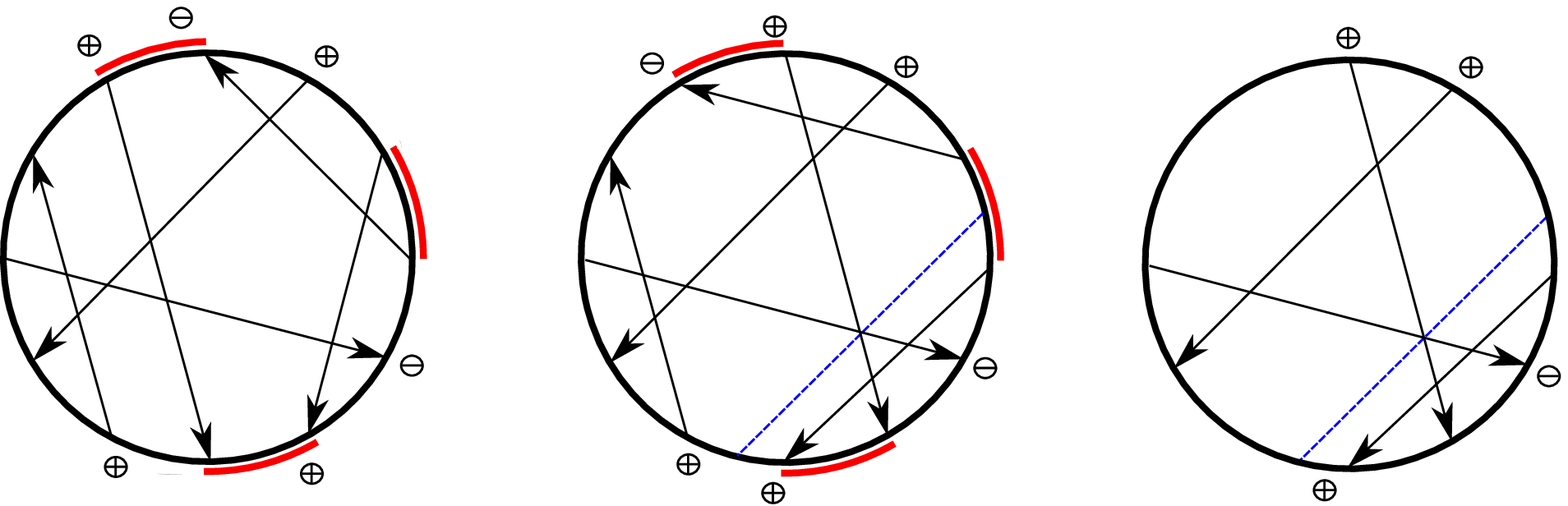
\end{tabular}
\\
\underline {6.77331:} & \begin{tabular}{c}
\def\svgwidth{4.9in}
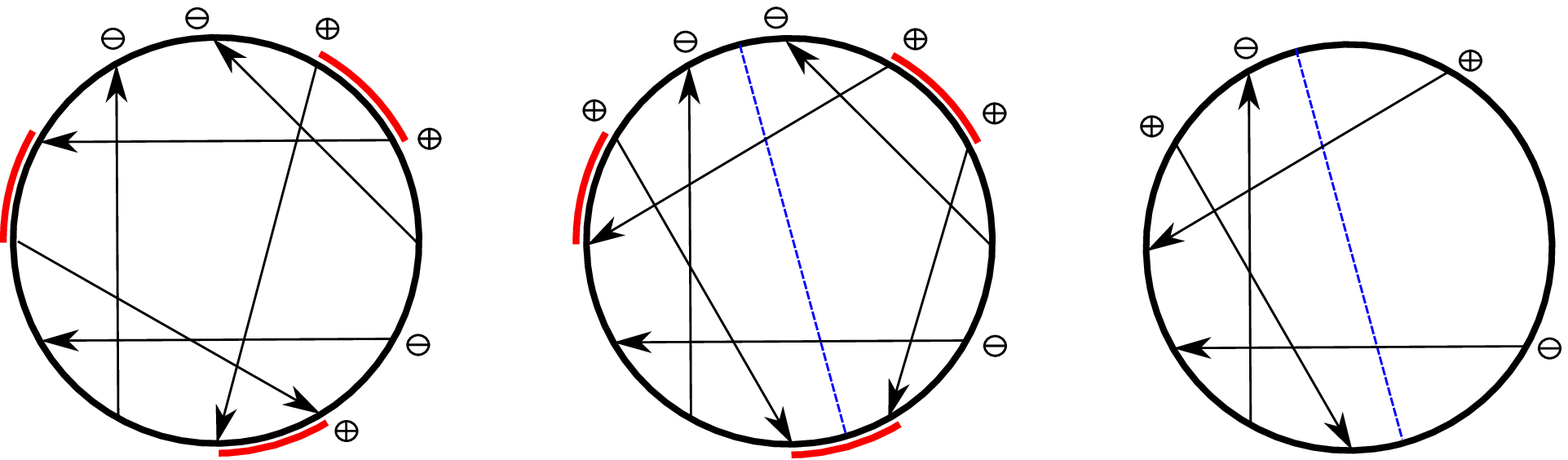
\end{tabular}
\\
\underline {6.77735} & \begin{tabular}{c}
\def\svgwidth{4.9in}
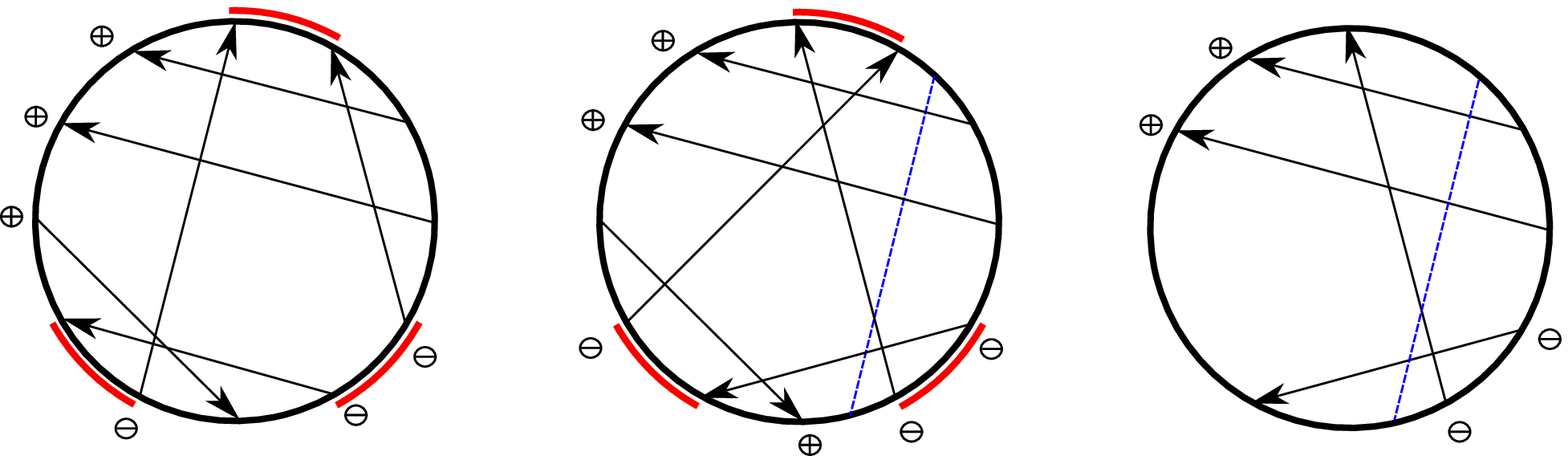
\end{tabular}
\end{tabular}
}
\caption{Slice movies for some more six crossing knots.} \label{fig_5_done_III}
\end{figure}
\clearpage

\bibliographystyle{gtart}

\bibliography{mil_plus_zh_bib}

\begin{thebibliography}{}
\providecommand\bibmarginpar{\leavevmode\marginpar}
\def\urlstyle#1{{\tt #1}}

\bibitem{abmw}
\textbf{B Audoux}, \textbf{P Bellingeri}, \textbf{J-B Meilhan}, \textbf{E
  Wagner}, \emph{Homotopy classification of ribbon tubes and welded string
  links}, Ann. Sc. Norm. Super. Pisa Cl. Sci. (5) 17 (2017) 713--761

\bibitem{bar_natan_milnor}
\textbf{D Bar-Natan}, \href{http://dx.doi.org/10.1142/S021821659500003X}
  {\emph{Vassiliev homotopy string link invariants}}, J. Knot Theory
  Ramifications 4 (1995) 13--32

\bibitem{dror}
\textbf{D Bar-Natan}, \emph{Crossing the crossings},
  \url{http://drorbn.net/AcademicPensieve/2015-11/xtx/xtx.pdf}  (2015)

\bibitem{bardakov_bellingeri}
\textbf{V\,G Bardakov}, \textbf{P Bellingeri},
  \href{http://dx.doi.org/10.1142/S021821651450014X} {\emph{Groups of virtual
  and welded links}}, J. Knot Theory Ramifications 23 (2014) 1450014, 23

\bibitem{bbc}
\textbf{H\,U Boden}, \textbf{M Chrisman},
  \href{https://arxiv.org/pdf/1903.08737.pdf} {\emph{Virtual concordance and
  the generalized Alexander polynomial}}, to appear, J. Knot Theory
  Ramifications  (2019)

\bibitem{bcg1}
\textbf{H\,U Boden}, \textbf{M Chrisman}, \textbf{R Gaudreau},
  \href{http://dx.doi.org/10.1080/10586458.2017.1422160} {\emph{Virtual knot
  cobordism and bounding the slice genus}}, Exp. Math. 28 (2019) 475--491

\bibitem{bcg2}
\textbf{H\,U Boden}, \textbf{M Chrisman}, \textbf{R Gaudreau},
  \href{http://dx.doi.org/10.1512/iumj.2020.69.8215} {\emph{Signature and
  concordance of virtual knots}}, Indiana Univ. Math. J. 69 (2020) 2395--2459

\bibitem{alex_boden}
\textbf{H\,U Boden}, \textbf{E Dies}, \textbf{A\,I Gaudreau}, \textbf{A
  Gerlings}, \textbf{E Harper}, \textbf{A\,J Nicas},
  \href{http://dx.doi.org/10.1142/S0218216515500091} {\emph{Alexander
  invariants for virtual knots}}, J. Knot Theory Ramifications 24 (2015)
  1550009, 62

\bibitem{acpaper}
\textbf{H\,U Boden}, \textbf{R Gaudreau}, \textbf{E Harper}, \textbf{A\,J
  Nicas}, \textbf{L White}, \href{https://doi.org/10.4064/fm80-9-2016}
  {\emph{Virtual knot groups and almost classical knots}}, Fund. Math. 238
  (2017) 101--142

\bibitem{boden_nagel}
\textbf{H\,U Boden}, \textbf{M Nagel},
  \href{https://doi.org/10.1090/proc/13667} {\emph{Concordance group of virtual
  knots}}, Proc. Amer. Math. Soc. 145 (2017) 5451--5461

\bibitem{CKS}
\textbf{J\,S Carter}, \textbf{S Kamada}, \textbf{M Saito},
  \href{http://dx.doi.org/10.1142/S0218216502001639} {\emph{Stable equivalence
  of knots on surfaces and virtual knot cobordisms}}, J. Knot Theory
  Ramifications 11 (2002) 311--322Knots 2000 Korea, Vol. 1 (Yongpyong)

\bibitem{chen_fox_lyndon}
\textbf{K-T Chen}, \textbf{R\,H Fox}, \textbf{R\,C Lyndon},
  \href{http://dx.doi.org/10.2307/1970044} {\emph{Free differential calculus.
  {IV}. {T}he quotient groups of the lower central series}}, Ann. of Math. (2)
  68 (1958) 81--95

\bibitem{cheng_gao}
\textbf{Z Cheng}, \textbf{H Gao},
  \href{http://dx.doi.org/10.1142/S0218216513410022} {\emph{A polynomial
  invariant of virtual links}}, J. Knot Theory Ramifications 22 (2013) 1341002,
  33

\bibitem{band_pass}
\textbf{M Chrisman}, \href{http://dx.doi.org/10.1142/S0218216517500572}
  {\emph{Band-passes and long virtual knot concordance}}, J. Knot Theory
  Ramifications 26 (2017) 1750057, 11

\bibitem{ct}
\textbf{D Cimasoni}, \textbf{V Turaev},
  \href{http://projecteuclid.org/euclid.ojm/1189717421} {\emph{A generalization
  of several classical invariants of links}}, Osaka J. Math. 44 (2007) 531--561

\bibitem{cochran}
\textbf{T\,D Cochran}, \href{http://dx.doi.org/10.1090/memo/0427}
  {\emph{Derivatives of links: {M}ilnor's concordance invariants and {M}assey's
  products}}, Mem. Amer. Math. Soc. 84 (1990) x+73

\bibitem{cochran_orr}
\textbf{T\,D Cochran}, \textbf{K\,E Orr},
  \href{http://dx.doi.org/10.2307/2946555} {\emph{Not all links are concordant
  to boundary links}}, Ann. of Math. (2) 138 (1993) 519--554

\bibitem{COT}
\textbf{T\,D Cochran}, \textbf{K\,E Orr}, \textbf{P Teichner},
  \href{http://dx.doi.org/10.4007/annals.2003.157.433} {\emph{Knot concordance,
  {W}hitney towers and {$L^2$}-signatures}}, Ann. of Math. (2) 157 (2003)
  433--519

\bibitem{CHN}
\textbf{A\,S Crans}, \textbf{A Henrich}, \textbf{S Nelson},
  \href{http://dx.doi.org/10.1142/S021821651340004X} {\emph{Polynomial knot and
  link invariants from the virtual biquandle}}, J. Knot Theory Ramifications 22
  (2013) 134004, 15

\bibitem{DKK}
\textbf{H\,A Dye}, \textbf{A Kaestner}, \textbf{L\,H Kauffman},
  \href{https://doi.org/10.1142/S0218216517410012} {\emph{Khovanov homology,
  {L}ee homology and a {R}asmussen invariant for virtual knots}}, J. Knot
  Theory Ramifications 26 (2017) 1741001, 57

\bibitem{dye_kauffman}
\textbf{H\,A Dye}, \textbf{L\,H Kauffman},
  \href{http://dx.doi.org/10.1142/S0218216510008200} {\emph{Virtual homotopy}},
  J. Knot Theory Ramifications 19 (2010) 935--960

\bibitem{fenn}
\textbf{R\,A Fenn}, \emph{Techniques of geometric topology}, volume~57 of
  \emph{London Mathematical Society Lecture Note Series}, Cambridge University
  Press, Cambridge (1983)

\bibitem{green}
\textbf{J Green}, \emph{A table of virtual knots},
  \url{http://www.math.toronto.edu/drorbn/Students/GreenJ}  (2004)

\bibitem{HL}
\textbf{N Habegger}, \textbf{X-S Lin}, \href{http://dx.doi.org/10.2307/1990959}
  {\emph{The classification of links up to link-homotopy}}, J. Amer. Math. Soc.
  3 (1990) 389--419

\bibitem{HL2}
\textbf{N Habegger}, \textbf{X-S Lin},
  \href{http://dx.doi.org/10.1112/S0024609398004494} {\emph{On link concordance
  and {M}ilnor's {$\overline {\mu}$} invariants}}, Bull. London Math. Soc. 30
  (1998) 419--428

\bibitem{HM}
\textbf{N Habegger}, \textbf{G Masbaum},
  \href{http://dx.doi.org/10.1016/S0040-9383(99)00041-5} {\emph{The
  {K}ontsevich integral and {M}ilnor's invariants}}, Topology 39 (2000)
  1253--1289

\bibitem{hall}
\textbf{M Hall, Jr}, \emph{The theory of groups}, Chelsea Publishing Co., New
  York (1976)Reprinting of the 1968 edition

\bibitem{henrich}
\textbf{A Henrich}, \href{http://dx.doi.org/10.1142/S0218216510007917} {\emph{A
  sequence of degree one {V}assiliev invariants for virtual knots}}, J. Knot
  Theory Ramifications 19 (2010) 461--487

\bibitem{hill}
\textbf{J\,A Hillman}, \href{http://dx.doi.org/10.1112/blms/10.1.105}
  {\emph{Alexander ideals and {C}hen groups}}, Bull. London Math. Soc. 10
  (1978) 105--110

\bibitem{JKS}
\textbf{F Jaeger}, \textbf{L\,H Kauffman}, \textbf{H Saleur},
  \href{http://dx.doi.org/10.1006/jctb.1994.1047} {\emph{The {C}onway
  polynomial in {${\bf R}^3$} and in thickened surfaces: a new determinant
  formulation}}, J. Combin. Theory Ser. B 61 (1994) 237--259

\bibitem{karimi}
\textbf{H Karimi}, \href{https://arxiv.org/pdf/1904.12235.pdf} {\emph{The
  {K}hovanov homology of alternating virtual links}}, arXiv/math.GT  (2019)

\bibitem{KaV}
\textbf{L\,H Kauffman}, \emph{Virtual knot theory}, European J. Combin. 20
  (1999) 663--690

\bibitem{kauffman_odd_writhe}
\textbf{L\,H Kauffman}, \href{http://dx.doi.org/10.4064/fm184-0-10} {\emph{A
  self-linking invariant of virtual knots}}, Fund. Math. 184 (2004) 135--158

\bibitem{kauffman_affine}
\textbf{L\,H Kauffman}, \href{http://dx.doi.org/10.1142/S0218216513400075}
  {\emph{An affine index polynomial invariant of virtual knots}}, J. Knot
  Theory Ramifications 22 (2013) 1340007, 30

\bibitem{lou_cob}
\textbf{L\,H Kauffman}, \href{http://dx.doi.org/10.1142/9789814630627_0009}
  {\emph{Virtual knot cobordism}}, from ``New ideas in low dimensional
  topology'', Ser. Knots Everything 56, World Sci. Publ., Hackensack, NJ (2015)
   335--377

\bibitem{kauffman_radford}
\textbf{L\,H Kauffman}, \textbf{D Radford},
  \href{http://dx.doi.org/10.1090/conm/318/05548} {\emph{Bi-oriented quantum
  algebras, and a generalized {A}lexander polynomial for virtual links}}, from
  ``Diagrammatic morphisms and applications ({S}an {F}rancisco, {CA}, 2000)'',
  Contemp. Math. 318, Amer. Math. Soc., Providence, RI (2003)  113--140

\bibitem{kuperberg}
\textbf{G Kuperberg}, \emph{What is a virtual link?}, Algebraic and Geometric
  Topology 3 (2003) 587--591

\bibitem{lin_power}
\textbf{X-S Lin}, \emph{Power series expansions and invariants of links}, from
  ``Geometric topology ({A}thens, {GA}, 1993)'', AMS/IP Stud. Adv. Math. 2,
  Amer. Math. Soc., Providence, RI (1997)  184--202

\bibitem{MKS}
\textbf{W Magnus}, \textbf{A Karrass}, \textbf{D Solitar}, \emph{Combinatorial
  group theory}, second edition, Dover Publications, Inc., Mineola, NY
  (2004)Presentations of groups in terms of generators and relations

\bibitem{manturov_GAP}
\textbf{V\,O Manturov}, \href{http://dx.doi.org/10.1023/A:1016258728022}
  {\emph{On invariants of virtual links}}, Acta Appl. Math. 72 (2002) 295--309

\bibitem{manturov_long}
\textbf{V\,O Manturov}, \href{http://dx.doi.org/10.1090/S0077-1554-08-00168-4}
  {\emph{Compact and long virtual knots}}, Tr. Mosk. Mat. Obs. 69 (2008) 5--33

\bibitem{mellor}
\textbf{B Mellor}, \href{http://dx.doi.org/10.1142/S0218216516500504}
  {\emph{Alexander and writhe polynomials for virtual knots}}, J. Knot Theory
  Ramifications 25 (2016) 1650050, 30

\bibitem{milnor}
\textbf{J Milnor}, \emph{Isotopy of links. {A}lgebraic geometry and topology},
  from ``A symposium in honor of {S}. {L}efschetz'', Princeton University
  Press, Princeton, N. J. (1957)  280--306

\bibitem{ng}
\textbf{K\,Y Ng}, \href{http://dx.doi.org/10.1016/S0040-9383(97)00037-2}
  {\emph{Groups of ribbon knots}}, Topology 37 (1998) 441--458

\bibitem{nickel}
\textbf{W Nickel}, \href{http://dx.doi.org/10.1016/S0260-6917(96)80020-7}
  {\emph{Computing nilpotent quotients of finitely presented groups}}, from
  ``Geometric and computational perspectives on infinite groups ({M}inneapolis,
  {MN} and {N}ew {B}runswick, {NJ}, 1994)'', DIMACS Ser. Discrete Math.
  Theoret. Comput. Sci. 25, Amer. Math. Soc., Providence, RI (1996)  175--191

\bibitem{orr}
\textbf{K\,E Orr}, \href{http://dx.doi.org/10.1007/BF01393902} {\emph{Homotopy
  invariants of links}}, Invent. Math. 95 (1989) 379--394

\bibitem{polyak_minimal}
\textbf{M Polyak}, \href{http://dx.doi.org/10.4171/QT/10} {\emph{Minimal
  generating sets of {R}eidemeister moves}}, Quantum Topol. 1 (2010) 399--411

\bibitem{roseman}
\textbf{D Roseman}, \emph{Reidemeister-type moves for surfaces in
  four-dimensional space}, from ``Knot theory ({W}arsaw, 1995)'', Banach Center
  Publ. 42, Polish Acad. Sci. Inst. Math., Warsaw (1998)  347--380

\bibitem{rush}
\textbf{W Rushworth}, \href{http://dx.doi.org/10.1016/j.topol.2018.11.028}
  {\emph{Computations of the slice genus of virtual knots}}, Topology Appl. 253
  (2019) 57--84

\bibitem{satoh}
\textbf{S Satoh}, \href{http://dx.doi.org/10.1142/S0218216500000293}
  {\emph{Virtual knot presentation of ribbon torus-knots}}, J. Knot Theory
  Ramifications 9 (2000) 531--542

\bibitem{saw}
\textbf{J {Sawollek}}, \href{https://arxiv.org/pdf/math/9912173.pdf} {\emph{{On
  Alexander-Conway Polynomials for Virtual Knots and Links}}}, arXiv/math.GT
  (1999)

\bibitem{silwill_vkg}
\textbf{D\,S Silver}, \textbf{S\,G Williams},
  \href{http://dx.doi.org/10.1142/9789812792679_0027} {\emph{Virtual knot
  groups}}, from ``Knots in {H}ellas '98 ({D}elphi)'', Ser. Knots Everything
  24, World Sci. Publ., River Edge, NJ (2000)  440--451

\bibitem{silwill0}
\textbf{D\,S Silver}, \textbf{S\,G Williams},
  \href{http://dx.doi.org/10.1142/S0218216501000792} {\emph{Alexander groups
  and virtual links}}, J. Knot Theory Ramifications 10 (2001) 151--160

\bibitem{silwill1}
\textbf{D\,S Silver}, \textbf{S\,G Williams},
  \href{http://dx.doi.org/10.1142/S0218216503002901} {\emph{Polynomial
  invariants of virtual links}}, J. Knot Theory Ramifications 12 (2003)
  987--1000

\bibitem{silwill}
\textbf{D\,S Silver}, \textbf{S\,G Williams},
  \href{https://doi.org/10.1142/S0218216506004956} {\emph{Crowell's derived
  group and twisted polynomials}}, J. Knot Theory Ramifications 15 (2006)
  1079--1094

\bibitem{stallings_central}
\textbf{J Stallings}, \href{http://dx.doi.org/10.1016/0021-8693(65)90017-7}
  {\emph{Homology and central series of groups}}, J. Algebra 2 (1965) 170--181

\bibitem{GAP4}
{The GAP~Group}, \emph{{GAP -- Groups, Algorithms, and Programming, Version
  4.10.2}} (2019)
\ Available at \setbox0\hbox{\makeatletter\@url
{https://www.gap-system.org}}
\href{https://www.gap-system.org}
{\unhbox0}

\bibitem{traldi}
\textbf{L Traldi}, \href{http://dx.doi.org/10.2307/1999294} {\emph{Milnor's
  invariants and the completions of link modules}}, Trans. Amer. Math. Soc. 284
  (1984) 401--424

\bibitem{turaev_cobordism}
\textbf{V Turaev}, \href{http://dx.doi.org/10.1112/jtopol/jtn002}
  {\emph{Cobordism of knots on surfaces}}, J. Topol. 1 (2008) 285--305

\bibitem{wolfram}
\textbf{{Wolfram Research, Inc}}, \emph{Mathematica, {V}ersion 12.0}Champaign,
  IL, 2019
\ Available at \setbox0\hbox{\makeatletter\@url
{https://www.wolfram.com/mathematica}}
\href{https://www.wolfram.com/mathematica}
{\unhbox0}

\end{thebibliography}

\end{document}